\documentclass[arxiv,reqno,twoside,a4paper,12pt]{amsart}

\usepackage[utf8]{inputenc}
\usepackage{amsmath}
\usepackage{amsfonts}
\usepackage{amssymb}
\usepackage{amsthm}
\usepackage{float}
\usepackage{mathtools}
\usepackage{csquotes}
\usepackage{upgreek,esint,enumerate}
\usepackage[numbers]{natbib}
\usepackage{dsfont}
\usepackage{enumitem}
\usepackage{graphicx}

\usepackage[colorlinks=true,linkcolor=blue,citecolor=blue]{hyperref}

\usepackage[scaled]{helvet} % ss
\usepackage{courier} % tt
\usepackage[mathbf]{euler}
\usepackage[dvipsnames]{xcolor}

\usepackage[driver=pdftex,margin=3cm,heightrounded=true,centering]{geometry}

\usepackage{tikz-cd}
\usetikzlibrary{arrows,matrix,shapes,decorations.text, mindmap,shapes.misc}
\usepackage{metalogo}

\newcommand{\C}{\mathbb{C}}

\renewcommand{\S}{\mathbb{S}}
\renewcommand{\P}{\mathbb{P}}
\newcommand{\R}{\mathbb{R}}
\newcommand{\norm}[1]{\left\lVert#1\right\rVert}
\newcommand{\ip}[2]{\left\langle #1,#2\right\rangle}

\newcommand{\tr}[0]{\text{tr}}

\newcommand{\Vol}[0]{\text{Vol}}
\newcommand{\diam}{\text{diam}}
\usepackage[numbers]{natbib}

\newcommand{\cC}{\mathcal{C}}
\newcommand{\bB}{\mathbb{B}}
\newcommand{\cD}{\mathcal{D}}
\newcommand{\cO}{\mathcal{O}}
\newcommand{\cE}{\mathcal{E}}
\newcommand{\cS}{\mathcal{S}}

\newcommand{\cM}{\mathcal{M}}

\tikzset{cross/.style={cross out, draw=black, minimum size=2*(#1-\pgflinewidth), inner sep=0pt, outer sep=0pt},
cross/.default={1pt}}

\theoremstyle{plain}
\newtheorem{Thm}{Theorem}[section]

\newtheorem{Lem}[Thm]{Lemma}

\newtheorem{Prop}[Thm]{Proposition}
\newtheorem{Def}[Thm]{Definition}
\newtheorem{Rem}[Thm]{Remark}

\theoremstyle{plain}
\newtheorem{thm}{theorem}
\newtheorem{Assump}[thm]{Assumption}

\usepackage[toc,page]{appendix}

\numberwithin{equation}{section}

\title{}
\date{}
\author{}

\begin{document}

\title[Convergence of the singular Yamabe flow]
{Convergence of the Yamabe flow on singular spaces with positive Yamabe constant}

\author{Gilles Carron}
\address{D\'{e}partement de Math\'{e}matiques, Universit\'{e} de Nantes}
\address{2 rue de la Houssini\`{e}re,  BP 92208, 44322 Nantes cedex 03, France}
\email{Gilles.Carron@univ-nantes.fr}

\author{J\o rgen Olsen Lye}
\address{Mathematisches Institut,
Universit\"at Oldenburg,
26129 Oldenburg,
Germany}
\email{jorgen.olsen.lye@uni-oldenburg.de}

\author{Boris Vertman}
\address{Mathematisches Institut,
Universit\"at Oldenburg,
26129 Oldenburg,
Germany}
\email{boris.vertman@uni-oldenburg.de}

\subjclass[2000]{53C44; 58J35; 35K08}
\date{\today}

\begin{abstract}
{In this work, we study the convergence of the normalized Yamabe flow 
with positive Yamabe constant on a class of pseudo-manifolds that includes stratified spaces with iterated cone-edge metrics. We establish convergence under a low-energy condition. We also prove a concentration--compactness dichotomy, and investigate what the alternatives to convergence are. We end by investigating a non-convergent example due to Viaclovsky in more detail.}
\end{abstract}

\maketitle
\tableofcontents

%%%%%%%%%%%%%%%%%%%%%%%%%%%%%%%%%
\section{Introduction and statement of the main results}\label{intro-section}
%%%%%%%%%%%%%%%%%%%%%%%%%%%%%%%%%

The Yamabe conjecture states that for any compact, smooth
Riemannian manifold $(M,g_0)$ there exists a constant scalar curvature metric, conformal
to $g_0$. The first proof of this conjecture was initiated by Yamabe \cite{Yamabe} and continued by Trudinger \cite{Trudinger}
Aubin \cite{Aubin}, and Schoen \cite{Schoen}. The proof is based on the calculus of variations and elliptic partial differential 
equations. An alternative tool for proving the conjecture is due to Hamilton \cite{Hamilton}, who introduced the normalized Yamabe flow of on a Riemannian manifold $(M,g_0)$,
which is a family $g\equiv g(t), t\in [0,T]$ of Riemannian metrics on $M$ such 
that the following evolution equation holds
 \begin{equation}\label{eq:YF}
 \partial_t g=-(S-\sigma)g, \quad \sigma := \Vol_g(M)^{-1} \int_M S\, d\Vol_g.
 \end{equation}
Here $S$ is the scalar curvature of $g$, $\Vol_g(M)$ the total volume of $M$ with respect to $g$ and $\sigma$ is the average scalar curvature of $g$.
The normalization by $\sigma$ ensures that the total volume does not change along the flow.
Hamilton \cite{Hamilton} showed the long time existence of the Yamabe flow. 
It preserves the conformal class of $g_0$ and ideally should converge
to a constant scalar curvature metric, thereby establishing the Yamabe conjecture by parabolic methods. \medskip

Establishing convergence of the normalized Yamabe flow is intricate already in the setting of smooth, compact manifolds.
In case of scalar negative, scalar flat and locally conformally flat scalar positive cases, convergence is due to Ye \cite{Ye}.  
The case of a non-conformally flat $g_0$ with positive scalar curvature is delicate and has been studied by Schwetlick-Struwe \cite{SS}  and Brendle 
\cite{Brendle, BrendleYF} amongst others.\medskip

More specifically, \cite[Theorem 3.1]{SS} is a concentration-compactness dichotomy, and \cite[Section 5]{SS},  \cite[p. 270]{Brendle}, and  \cite[p. 544]{BrendleYF} invoke the positive mass theorem to rule out concentration (also known as the formation of bubbles), which is 
where the dimensional restriction in \cite{SS, Brendle} and the spin assumption in \cite[Theorem 4]{BrendleYF} come from. 
Assuming \cite{SY17} to be correct, \cite{Brendle} and \cite{BrendleYF} cover all closed manifolds which are not conformally equivalent to spheres. 
\medskip

In the non-compact setting, our understanding is limited. On complete manifolds, long-time existence 
has been discussed in various settings by Ma \cite{Ma}, Ma and An \cite{MaAn}, and the recent contribution by Schulz \cite{Schulz}.
On incomplete surfaces, where Ricci and Yamabe flows coincide, Giesen and Topping \cite{Topping, Topping2}
constructed a flow that becomes instantaneously complete. 
In higher dimensions, we have the work of Roidos \cite{Conic}, which establishes existence of the Yamabe-flow in the presence of a cone singularity as long as the initial scalar curvature is in $L^{q}(M)$  for some $q>\frac{n}{2}$. See the short discussion in Subsection \ref{subsection: S_0Regularity} below for the geometric interpretation. Convergence is not discussed in \cite{Conic}.
\medskip

In this work, we study the convergence of the Yamabe flow on a general class of spaces that includes incomplete spaces with 
cone-edge (wedge) singularities or, more generally, stratified spaces with iterated cone-edge singularities. 
This continues a program initiated in \cite{ShortTime, LongTime}, where 
existence and convergence of the Yamabe flow has been established in case of negative Yamabe invariant, and \cite{LongTime2} where long time existence is studied in the case of a positive Yamabe constant. This is also an extension of the work of the first author with collaborators, \cite{ACM}, \cite{ACM2}, and \cite{ACT}.
\medskip

\noindent
 We now proceed with explaining the assumptions in detail.

%%%%%%%%%%%%%%%%%%%%%%%%%%%%%%%%%
\subsection{Normalized Yamabe flow and Yamabe constant}
%%%%%%%%%%%%%%%%%%%%%%%%%%%%%%%%%
Consider a Riemannian manifold $(M,g_0)$, with $g_0$ normalized such that the total volume $\Vol_{g_0}(M)=1$. We assume $M\subset \overline{M}$ to be the regular part of a compact metric measure space $\overline{M}$, meaning $\overline{M}$ carries a distance function\footnote{This distance does not evolve in time. We will not consider the distance associated to the evolving metric.}  $d\colon \overline{M}\times \overline{M}\to \R$ which coincides with the distance induced from $g_0$ on $M$.  We will often suppress the metric $d$, and leave it out of the notation. We will assume that $\overline{M}$ satisfies the \cite[Hypothesis i)-iv)a)]{ACM}. We will state these assumptions explicitly below as Assumption \ref{admissible-assump}. 
\medskip
 
The Yamabe flow \eqref{eq:YF} preserves the conformal class
of the initial metric $g_0$ and, assuming throughout the paper always $\dim M = n \geq 3$, we can write $g=u^{\frac{4}{n-2}}g_0$ 
for some function $u>0$ on $M_T = M \times [0,T]$ for some upper time limit $T>0$. Then the normalized Yamabe flow equation
can be equivalently written as an equation for $u$
\begin{equation}
\partial_t \left(u^{\frac{n+2}{n-2}}\right) = \frac{n+2}{4} \left(\sigma u^{\frac{n+2}{n-2}}- L_0(u)\right),
\quad L_0 := S_0-\frac{4(n-1)}{n-2} \Delta_0,
\label{eq:YamabeFlow}
\end{equation}
where $L_0$ is the conformal Laplacian of $g_0$, defined in terms of the scalar curvature $S_0$
and the Laplace Beltrami operator $\Delta_0$ associated to the initial metric $g_0$. 
The scalar curvature $S$ of the evolving metric $g$ can be written 
\begin{equation}
S=u^{-\frac{n+2}{n-2}} L_0(u)=S_0 u^{-\frac{4}{n-2}} - 4\frac{n-1}{n-2} u^{-\frac{n+2}{n-2}}\Delta_0 u,
\label{eq:EvolvingScalar} 
\end{equation}
and the volume form of $g=u^{\frac{4}{n-2}}g_0$ is given by $d\Vol_g=u^{\frac{2n}{n-2}} d\mu$,
where we write $d\mu := d\Vol_{g_0}$ for the time-independent initial volume form. 
One computes 
\begin{align}\label{eq:DVol}
\partial_t d\Vol_g &=-\frac{n}{2} (S-\sigma)\, d\Vol_g.
\end{align}
Hence, the total volume of $(M,g)$ is constant, and thus equal to $1$ along the flow.
The average scalar curvature takes the form
\begin{equation}
\sigma=\int_M S\, d\Vol_g= \int_M L_0(u) u^{-\frac{n+2}{n-2}} u^{\frac{2n}{n-2}} \, d\mu= \int_{M}\frac{4(n-1)}{n-2}  \vert \nabla u\vert^2_{g_0}+S_0 u^2\, d\mu.
\label{eq:rho}
\end{equation} 
Explicit computations lead to the following evolution equation for the average scalar curvature
\begin{align} 
\label{eq:rhoEvol}
\partial_t \sigma &=-\frac{n-2}{2} \int_M \left( \sigma-S\right)^2 u^{\frac{2n}{n-2}}\, d\mu.
\end{align}
The latter evolution equation in particular implies that $\sigma \equiv \sigma(t)$ is non-increasing 
along the flow. We conclude the exposition with defining some Sobolev spaces and the Yamabe constant of $g_0$, which 
incidentally provides a lower bound for $\sigma$.  We define the $L^q(M)$ spaces with respect to the integration measure $d\mu$. \medskip

\noindent We define the first Sobolev space $H^1(M)$ as the completion of the Lipschitz functions $\text{Lip}(\overline{M})$ with respect to the $H^1$-norm,  
\begin{equation}
\norm{v}^2_{H^1(M)} := \int_M \nu^2 \, d\mu +\int_M \vert \nabla \nu\vert_{g_0}^2\, d\mu,
\label{eq:H1Norm}
\end{equation}
where $\vert \nabla v\vert^2_{g_0}\coloneqq g_{0}^{ij} (\partial_i v)(\partial_j v)$.
Similarly, we define $H^1(M,g)$ by using $d\Vol_g$ instead of $d\mu$ to define $L^2(M,g)$, and $\vert \nabla \nu\vert_{g}^2=u^{-\frac{4}{n-2}} \vert \nabla \nu\vert_{g_0}^2$ instead of $\vert \nabla \nu\vert_{g_0}^2$. Let $u$ be a solution of \eqref{eq:YamabeFlow}.  If $u$ and $u^{-1}$ are both bounded, one easily checks $H^1(M) = H^1(M,g)$ with equivalent norms. 
\medskip

\noindent We define the Yamabe constant of $g_0$ as follows
\begin{equation}\begin{split}
Y(M,g_0) &:= \inf_{v\in H^1(M)\setminus \{0\}}  
\frac{\int_{M} \frac{4(n-1)}{n-2} \vert \nabla v\vert^2_{g_0}+S_0 v^2\, d\mu}{\norm{v}^2_{L^{\frac{2n}{n-2}}(M)}}
\\ &\leq \int_{M}\frac{4(n-1)}{n-2}  \vert \nabla u\vert^2_{g_0}+S_0 u^2\, d\mu \stackrel{\eqref{eq:rho}}{=} \sigma,
\label{eq:YamabeConst}
\end{split}\end{equation}
where in the inequality we have used that for any solution $u$ of \eqref{eq:YamabeFlow}, 
$\norm{u}_{L^{\frac{2n}{n-2}}(M)} = d\Vol_g(M) \equiv 1$.
How one proceeds will depend heavily on the sign of the Yamabe constant. In this paper\footnote{See \cite{ShortTime, LongTime} for the case of $Y(M,g_0)\leq 0$.} we will assume $Y(M,g_0)>0$. 
This in particular implies directly from \eqref{eq:YamabeConst}  that the average scalar curvature $\sigma$ is positive and uniformly bounded away from zero along the normalized Yamabe flow.
\smallskip

\begin{Assump}\label{Y-assump}
The Yamabe constant $Y(M,g_0)$ is positive.
\end{Assump} \ \\[-8mm]

We also need the local Yamabe constant $Y_{\ell}(\overline{M},[g_0])$, which is defined as follows. For any $p\in \overline{M}$, we let $B(p,R)$ denote the open ball centred at $p$ of radius $R$. We then let
\begin{equation}
Y_{\ell}(\overline{M},g_0)\coloneqq \inf_{p\in \overline{M}}\lim_{R\to 0} Y(B(p,R),g_0),
\end{equation}
where the Yamabe constant $Y(U)$ for open sets $U\subset \overline{M}$ is defined as \eqref{eq:YamabeConst}, where all the integrals are over $U$ and $v\in H^1(U)$ with $\text{supp}(u)\subset U\cap M$. In Section \ref{Section:Dichotomy}, we will need the Sobolev constant $\cS(U)$ for an open set $U\subset \overline{M}$ and the local Sobolev constant $\cS_{\ell}(\overline{M})$ as well, so we record the definitions here.
\begin{equation}\label{eq:SobolevConst} \begin{split}
\cS(U) &= \inf \left\{ 4\frac{n-1}{n-2}\int_M \vert \nabla \phi\vert^2_g\, d\Vol_g\, \colon\, \phi\in H^{2} (U),\right. \\
& \quad \left. \text{supp}(u)\subset U,\,  \norm{\phi}_{L^{\frac{2n}{n-2}}(U)}=1\right\}, \\
\cS_{\ell}(\overline{M}) &=\inf_{p\in \overline{M}} \lim_{R\to 0} \cS(B(p,R)).
\end{split}\end{equation}

%%%%%%%%%%%%%%%%%%%%%%%%%%%%%%%%%
\subsection{Admissibility assumptions}
%%%%%%%%%%%%%%%%%%%%%%%%%%%%%%%%%

This work, like \cite{LongTime2}, is strongly influenced by the work of Akutagawa, the first author, and Mazzeo, \cite{ACM}  on the Yamabe problem on stratified spaces.
We will carry over hypothesis $i)-iii)$ and $iv) a)$ from \cite[p.  5]{ACM}.

\begin{Def}\label{admissible}
Let $(M,g_0)$ be a smooth Riemannian manifold of dimension $n$. We call 
$(M,g_0)$ admissible, if it satisfies the following conditions.
\begin{itemize}
\item $M$ is the regular part of a compact, metric measure space $\overline{M}$.
\item Smooth, compactly supported functions $C^\infty_c(M)$ are dense\footnote{This can be phrased as $H^1_0(M)=H^1(M)$.
Note that this rules out $M$ being the interior of a manifold with a codimension 1 boundary.} in 
$H^1(M)$.
\item The Hausdorff $n$-dimensional measure is absolutely continuous with respect to $d\mu$, and both are Ahlfors $n$-regular, i.e.
\begin{equation}
C^{-1} R^n\leq \Vol(B(p,R))\leq C R^n
\label{eq:Ahlfors-Regular}
\end{equation}
for some $C>0$ and every $p\in \overline{M}$ and $2R< \text{diam}(\overline{M})\coloneqq \sup \{d(p,q)\, \colon\, p,q\in \overline{M}\}$.
\item $(M,g_0)$ admits a Sobolev inequality of the following kind.
There exist $A_0,B_0>0$ such that for all $f\in H^1(M)$
\begin{equation}
\norm{f}_{L^{\frac{2n}{n-2}}(M)}^2 \leq A_0 \norm{\nabla f}_{L^2(M)}^2 +B_0 \norm{f}_{L^2(M)}^2.
\label{eq:Sobolev}
\end{equation}
\item The scalar curvature of $g_0$ satisfies $\norm{S_0}_{L^{q}(M)}<\infty$  for some $q>\frac{n}{2}$.
\end{itemize}
\end{Def}

The main examples we have in mind are closed manifolds and regular parts of smoothly stratified spaces,
endowed with iterated cone-edge metrics. 
See \cite[Section 2.1]{ACM} or Appendix \ref{App:Bubbles} for a definition of the latter. That the Sobolev inequality holds 
in this case is shown in \cite[Proposition 2.2]{ACM}.
Note that for most of this work,  we do not specify explicitly how the metric $g_0$ looks near the singular strata of $\overline{M}$.
Restrictions on the local behaviour of the metric will instead be coded in either $L^q$-data, like requiring the initial scalar curvature $S_0$ to be in 
$L^q(M)$ , or in geometric conditions  like the Ahlfors $n$-regularity \eqref{eq:Ahlfors-Regular}.
\begin{Rem}
This work does not extend to finite volume, complete, non-compact manifolds, since (see \cite[Lemma 3.2, pp. 18-19]{Sobolev}, \cite[Remark 2), pp. 56-57]{Sobolev}) any finite volume, complete manifold satisfying the Sobolev inequality is compact. The Sobolev inequality is our most important tool. 
\end{Rem}
\smallskip

\begin{Assump}\label{admissible-assump}
$(M,g_0)$ is an admissible Riemannian manifold.
\end{Assump} \ \\[-8mm]

For the convergence results of Section \ref{Section:Convergence} and \ref{Section: Eigenvalues} we also need a third assumption on $\overline{M}$, namely a local Poincar\'{e} inequality.
\begin{Assump}\label{P-assump}$\overline{M}$ satisfies the Poincar\'e inequality, that is to say there is a constant $\uplambda>0$ such that for any ball $B\subset M$ of radius $r\le \diam (M)$, we get
$$\| \varphi-\varphi_B\|_{L^2(B)}\le \lambda r\| \nabla\varphi\|_{L^2(B)}$$
for any Lipschitz function $\varphi\colon B\rightarrow \R$, where we have set
\[\varphi_B\coloneqq \fint_B \varphi\, d\mu=\frac{1}{\mu(B)}\int_B \varphi\, d\mu.\]
\end{Assump} \ \\[-8mm]

\noindent In what follows we want to relate the assumption of the Sobolev inequality 
\eqref{eq:Sobolev} in Definition \ref{admissible} to positivity of the Yamabe constant $Y(M,g_0)$, showing that there is some redundancy in the above assumptions. Since several of our arguments revolve around having the Sobolev inequality, we still leave it in to stress its importance.

\begin{Prop}\label{SY}
Assume $S_0\in L^{q}(M)$ for some $q>\frac{n}{2}$ and $0<Y(M,g_0)$. 
Then the Sobolev inequality \eqref{eq:Sobolev} holds. 
\end{Prop}
\noindent We prove this as Proposition \ref{Prop:Sobolev} below. \medskip

\noindent With the assumptions so far, we can say something about the local Yamabe constant.
\begin{Prop}{\cite[Prop. 1.4b)]{ACM}}
Under Assumption \ref{admissible-assump}, we have $Y_{\ell}(\overline{M},g_0)>0$.
\end{Prop}

\begin{Rem}
\label{Rem:SchoenYau}
It is important to notice that we allow points in $p \in \overline{M}\setminus M$ in the definition of $Y_{\ell}$. Indeed, if $p\in M$ is a smooth point then 
\[\lim_{R\to 0}Y(B(p,R),g_0)=Y(\S^n,g_{round})=n(n-1)\left(\Vol_{g_{round}}(\S^n)\right)^{\frac{2}{n}}>0.\] 
See for instance \cite[Lemma A.1, p. 225]{SchoenYauBook}. So $Y_{\ell}(\overline{M},g_0)>0$ is always true for smooth manifolds. 
\end{Rem}

%%%%%%%%%%%%%%%%%%%%%%%%%%%%%%%%%
\subsection{Regularity of the initial scalar curvature}
%%%%%%%%%%%%%%%%%%%%%%%%%%%%%%%%%

\label{subsection: S_0Regularity}

In this work, we show (Theorem \ref{Thm:FlowExistence}) that for $S_0\in L^{q}(M)$  for some $q>\frac{n}{2}$, we will have $S(t)\in L^{\infty}(M)$ for $t>0$ . We close
this subsection with the observation, that on stratified spaces, 
$S_0 \in L^q(M)$ for $q \geq  n/2$ and $S_0 \in L^\infty(M)$ basically carry the same geometric restriction to leading order. 
Indeed, consider a cone $(0,1) \times N$ over a Riemannian manifold $(N,g_0)$,
with metric $g_0 = dx^2 \oplus x^2 g_N + h$, where $h$ is smooth in $x\in [0,1]$ and 
$|h|_{\overline{g}} = O(x)$ as $x\to 0$, where we write $\overline{g}:= dx^2 \oplus x^2 g_N$. Then 

\begin{equation}
S_0 = \frac{\textup{scal}(g_N) - \dim N (\dim N -1)}{x^{2}} + O(x^{-1}), \quad \textup{as} \ x \to 0,
\end{equation}
where the higher order term $O(x^{-1})$ comes from the perturbation $h$. 
Both assumptions $S_0 \in L^\infty(M)$ and $S_0 \in L^q(M)$ for $q\geq  n/2$
imply that the leading term of the metric $g_0$ is scalar-flat, i.e. $\textup{scal}(g_N) = \dim N (\dim N -1)$. Of course, if $S_0\in L^q(M)$ for $\frac{n}{2}<q<n$, the $\frac{1}{x}$-term could still be there.

%%%%%%%%%%%%%%%%%%%%%%%%%%%%%%%%%
\subsection{The main results}\label{MainResults}
%%%%%%%%%%%%%%%%%%%%%%%%%%%%%%%%%
Theorem \ref{Thm:ShortTime} and Proposition \ref{Prop:LongTime} combine to say that the Yamabe flow exists for all time in our setting. 
Theorem \ref{Prop:Dichotomy} is a dichotomy which describes what can happen to sequences $u(t_k)$ as $t_k\to \infty$. This is our main convergence result, and says that subsequences either converge, or the volume concentrates at finally many points. The second case leads the so-called formation of bubbles. Proposition \ref{Prop:u+delta} says that given a non-trivial upper bound on the initial scalar curvature, the Yamabe flow has convergent subsequences and one gets a constant scalar curvature metric. Proposition \ref{Prop:Eigenvalue1} presents a criterion (in terms of the first eigenvalue of the Laplacian) for when the entire flow converges and not just subsequences.  In Section \ref{section: nonconvergent}, we present a detailed computation of a counter example due to Viaclovsky, which shows that the Yamabe flow does not always have convergent subsequences. In Appendix \ref{App:Bubbles}, we give a detailed and quite general description of the bubbles which the dichotomy predicts will appear when the flow does not converge.

%%%%%%%%%%%%%%%%%%%%%%%%%%%%%%%%%
\section{A general existence theorem for the Yamabe flow}\label{Existence}
%%%%%%%%%%%%%%%%%%%%%%%%%%%%%%%%%
In this section, we take a step back and show how the short-time existence and uniqueness of many kinds of flow equations follows from abstract semigroup theory. We summarize how to apply these results to the Yamabe flow towards the end of the section. Several of the results are conveniently phrased in the language of Dirichlet spaces, so we first introduce these and the relation to singular spaces. We warn the reader that there are other definitions around. In particular, it is common to call just the function space $\cD(\cE)$ a Dirichlet space. The definition presented here is what is called a regular, strongly local Dirichlet space in \cite{ACM2}.
\begin{Def}\label{Def:DirichletSpace}
We consider a Dirichlet space $(\cM,d,\mu,\cE)$ where\begin{itemize}
\item $(\cM,d)$ is compact with diameter $D$,
\item $\mu$ is a probability measure,
\item $(\cM,d,\mu,\cE)$ is regular, meaning $\cC^0(\cM)\cap\cD(\cE)$ is dense in $\cC^0(\cM)$ and in $\cD(\cE)$.
\item $(\cM,d,\mu,\cE)$ is strongly local, which means that for any $u,v\in \cD(\cE)$ such that $u$ is constant on a neighbourhood of the support of $v$, we have 
$\cE(u,v)=0$.
\end{itemize}
 In that case, there is a bilinear map the \textit{energy measure} $\Gamma\colon \cD(\cE)\times \cD(\cE)\rightarrow \mathbb{M}(\cM)$  such that 
\[u,v\in \cD(\cE)\colon\ \cE(u,v)=\int_\cM d\Gamma(u,v).\]
Here $\mathbb{M}(\cM)$ is the dual of $\cC^{0}(\cM)$, which we identify (using the Riesz–Markov–Kakutani representation theorem) with the space of signed Radon measure on $\cM$)
The energy measure is  determined by the formula 
$$\cE(\phi u,u)-\frac12 \cE(\phi,u^2)=\int_M \phi \,d\Gamma(u,u)\quad {\rm for}\ \ {\forall}u\in \cD(\cE),\ {\forall}\phi\in \cD(\cE) \cap \ \cC^0(\cM)\,.$$ 
Note that for $u\in  \cD(\cE), d\Gamma(u,u)$ is a positive Radon measure.
The energy measure  satisfies the Leibniz and chain rules: 
$$d\Gamma(uv,w)=ud\Gamma(v,w)+vd\Gamma(u,w)\quad {\rm for}\ \ {\forall}u, v, w\in \cD(\cE)\,,$$
$$d\Gamma(f(u),v)=f'(u)d\Gamma(u,v)\quad {\rm for}\ \ {\forall}u, v \in \cD(\cE),\ {\forall}f\in \mathrm{Lip}(\R)\,.$$
Then we make the following supplementary assumptions:
\begin{itemize}
\item We assume that $d$ is  compatible with the Caratheodory distance or intrinsic pseudo-distance $d_\cE$
 defined by
\[d_\cE(x,y)=\sup\left\{u(x)-u(y)\ |\ u\in \cD(\cE)\cap \ \cC^0(M)\ \ \mathrm{and}\ \ d\Gamma(u,u)\le d\mu\right\},\]
where $d\Gamma(u,u)\le d\mu$ means that there exists a function $f \le 1$ such that 
$d\Gamma(u,u)=fd\mu$. 
\item  For some integer $n>2$: $(\cM,d,\mu)$ is $n-$Ahlfors regular, meaning there is a constant $\uptheta$ such that for any 
$r\in (0,D]$ and any $x\in \cM$ we have
\[\uptheta^{-1}\le \frac{\mu(B(x,r))}{r^n}\le \uptheta.\]
\item  $(\cM,d,\mu,\cE)$ satisfies the Poincar\'e inequalities, which is to say
there is a constant $\upgamma$ such that for any 
$r\in (0,D]$ and any $x\in \cM$, the following holds 
\begin{equation*}
\left\|v-v_B\right\|^2_{L^2(B(x,r))}\le \upgamma r^2 \int_{B(x,r)} d\Gamma(v,v)\quad {\rm for}\ \ {\forall}v \in \cD(\cE)\,,
\end{equation*} 
where $v_B=\frac{1}{\mu(B(x,r))} \int_{B(x,r)} v d\mu$.
\end{itemize}
\end{Def}
It is well known that then there is a positive constant $A$ depending only on $n,\uptheta,\upgamma$ such that the following Euclidean Sobolev inequality holds
\begin{equation*}
A\, \left(\int_{\cM} |v|^{\frac{2n}{n-2}}d\mu\right)^{1-\frac 2n}\le  \int_{\cM} d\Gamma(v,v)+\int_{\cM} v^2d\mu\  \quad {\rm for}\ \ {\forall}\, v \in \cD(\cE)\,.
\end{equation*} 
In this case, we know that $\mathrm{Lip}(\cM)$ is dense in $\cD(\cE)$.
\begin{Rem}
\label{Rem:Dirichlet}
In the case of a dense Riemannian manifold $(M,g)\subset \overline{M}$ where the Riemannian distance coincides with $d$ on $M$, the reader may mentally substitute $\cE (u,v)=\int_M \ip{\nabla u}{\nabla v}_g\, d\Vol_g$, i.e. $d\Gamma(u,v)=\ip{\nabla u}{\nabla v}_g\, d\Vol_g$. The generator $L$ is then $L=-\Delta_g$, where the Laplacian $\Delta_g$ is negative definite like in the rest of the article.
All in all, this Dirichlet space setting covers the framework studied in the previous section with $\overline{M}=\cM$.
\end{Rem}
\begin{Rem}
By \cite[Proposition 3.1]{Schrodinger}, a compact stratified space with an iterated edge metric satisfies these conditions. 
\end{Rem}

\noindent Before we state our main tool, we recall what a sectorial operator is.
\begin{Def}
\label{Def:Sectorial}
Let $X$ be a Banach space. A closed, densely defined operator $A\colon \cD(A)\rightarrow X$  is called sectorial (of angle $0<\delta\leq \frac{\pi}{2}$) if there is a sector
\[\Sigma_{\pi/2+\delta}:=\left\{\lambda\in \C: |\arg \lambda|<\frac \pi 2+\delta\right\}\]
 included in the resolvent set of $A$ and if for each $\epsilon\in (0,\delta)$ there is a constant $C_\epsilon$ such that for any $\lambda\in \Sigma_{\pi/2+\epsilon}\setminus\{0\}$ :
\[\left\|\left(A-\lambda\right)^{-1}\right\|_{X\to X}\le \frac{C_\epsilon}{|\lambda|}.\]
\end{Def}

We are going to use the following theorem, which is an adaptation of \cite[Theorem 8.1.1 and Corollary 8.3.8]{Lu} :
\begin{Thm}\label{general} Assume that $\cD,X$ are Banach spaces with $\cD\subset X$ dense and  continuously embedded, i.e.
\[\|\bullet\|_X\le \|\bullet\|_\cD.\]
Let $\cO\subset \cD$ be an open set and let $F\colon \cO\rightarrow X$ be a smooth  map such that for any $u\in \cO$, the Frechet derivative $L_u=DF(u)\colon \cD\rightarrow X$ is sectorial, and there is a positive constant $C(u)$ such that for any $\xi\in \cD$ we have\footnote{This says that the graph norm of $L_u$ is equivalent to the norm on $\cD$.}
\[\frac{1}{C(u)}\norm{\xi}_\cD\le \norm{L_u(\xi)}_X+\norm{\xi}_X\le C(u) \norm{\xi}_\cD.\]
Then for any $u_0\in \cO$ there is some $\tau>0$ and  $u\colon [0,\tau)\rightarrow \cD$ continuous with $u(0)=u_0$ and such that $u\colon (0,\tau)\rightarrow \cD$ is smooth and for any $t\in (0,\tau)$:
\[u'(t)=F(u(t)).\]
Moreover such a solution is unique.
\end{Thm}

\begin{Thm}
\label{Thm:FlowExistence}
 Assume that $(\cM,\mu,\cE)$ is a Dirichlet space\footnote{For this theorem, we do not need the strong locality condition, i.e. that $\cE(f,g)=\int_{\mathcal{M}}d\Gamma(f,g)$ for some carr\'{e} du champ $\Gamma$. Nor do we need the regularity condition. These conditions only becomes important in Section \ref{Section:Dichotomy}.} with $\mu(\cM)=1$, and assume that it satisfies the Sobolev inequality of dimension $\nu>2$, i.e.
\begin{equation}\label{eq:SemiGroupSobolev}
\forall \, w\, \in \cD(\cE)\colon \,\, \norm{w}^2_{L^{\frac{2\nu}{\nu-2}}(\cM,\mu)}\le A\cE(w)+B\norm{w}_{L^2(\cM,\mu)}^2.
\end{equation}
Assume that $\beta\in \R$, and $Q\in L^{p}(\cM,\mu)$ with $p>\frac{\nu}{2}$. Consider $L$ the generator of $\cE$ and 
\[H^{2,p}(\cM)\coloneqq \{f\in L^p(\cM,\mu)\,\colon\,  Lf\in L^p(\cM,\mu)\}.\]
Then for any positive function $u_0\in H^{2,p}(\cM) $ satisfying for some positive constants $c,C$
\[c\le u_0\le C,\]
 there exists a unique  $u\colon [0,\tau)\rightarrow H^{2,p}(\cM)$ continuous with $u(0)=u_0$ such that $u\colon (0,\tau)\rightarrow H^{2,p}(\cM)$ is smooth and for all $t\in (0,\tau):$
\begin{equation}
u'(t)=-u(t)^\beta\left(Lu(t)+Qu(t)\right).
\label{eq:TimeEvol}
\end{equation}
\end{Thm}

\proof \noindent\textbf{Step 1 -- Some consequence of the Sobolev inequality :}
\medskip

\noindent The Sobolev inequality implies that if $1\le p_1\leq p_2\le \infty$, then
there is a constant $C$ such that for any $t\in (0,1)$:
\begin{equation}
\left\|e^{-tL}\right\|_{L^{p_1}\to L^{p_2}}\le \frac{C}{t^{\frac{\nu}{2}\left(\frac{1}{p_1}-\frac{1}{p_2} \right)}}.
\label{eq:HeatKernelBound}
\end{equation}
This kind of inequality goes back to \cite[Eq. 8]{Nash}. See also \cite[Theorem, p. 259]{HardyLittlewood} (the cited result is for $p_1=1$ 
and $p_2=\infty$, the intermediate values are handled by interpolation).
It also implies that for any $1\le p_1\leq p_2\le \infty$,
there is a constant $C$ such that for any $t>0$ 
\begin{equation}
\left\|e^{-t(L+2)}\right\|_{L^{p_1}\to L^{p_2}}\le \frac{Ce^{-t}}{t^{\frac{\nu}{2}\left(\frac{1}{p_1}-\frac{1}{p_2} \right)}}.
\label{eq:HeatKernelShift}
\end{equation}
For any $p>\nu/2$, this will imply that we have a continuous Sobolev embedding $H^{2,p}(\cM)\subset L^\infty(\cM)$, meaning there is a constant $D$ such that for any $f\in H^{2,p}(\cM) $
\begin{equation}
\|f\|_\infty\le D\|f\|_{H^{2,p}(\cM)}:=D\left(\|Lf\|_{L^p(\cM,\mu)}+\|f\|_{L^p(\cM,\mu)} \right).
\label{eq:SobolevEmbedding}
\end{equation}
To see this, we use the formula $(L+2)^{-1}=\int_0^\infty e^{-t(L+2)}\, dt$ and \eqref{eq:HeatKernelShift} with $p_1=p$ and $p_2=\infty$ to deduce
\[\norm{(L+2)^{-1}}_{L^p\to L^\infty} \leq C\int_0^\infty t^{-\frac{\nu}{2p}}e^{-t}\, dt=C\Gamma\left(1-\frac{\nu}{2p}\right)<\infty,\]
where the last step uses $2p>\nu$.
\medskip

\noindent\textbf{Step 2 -- Sectoriality of the generator $L$:}\medskip

\noindent In the case of a Dirichlet space $(\cM,\mu,\cE)$, we know by \cite[Theorem 1, section III]{Stein} that the semi-group $\left(e^{-tL}\colon L^p\to L^p\right)_{\{t\ge 0\}}$ has a bounded analytic extension on the sector $\Sigma_{\frac \pi 2\left(1-|1-2/p|\right)}$, i.e.
\[\forall\, z\in \Sigma_{\frac \pi 2\left(1-|1-2/p|\right)}\, \colon\,  \left\|e^{-zL}\right\|_{L^p\to L^p}\le 1.\]
 Hence (see the proof of \cite[Theorem 4.6]{Engel} ) this implies that for any $p\in (1,\infty)$, $-L$ is sectorial on $L^p(\cM,\mu)$ with angle $\frac \pi 2\left(1-|1-2/p|\right).$
\medskip

\noindent\textbf{Step 3- Sectoriality of the operator $\rho L$:} \medskip

\noindent Assume that $(\cM,\mu,\cE)$ is  Dirichlet space and that $\rho\colon M\rightarrow \R_+$ is a positive measurable function such that for some positive constants $C>c>0\colon$
\[c\le \rho\le C,\ \mu-\mathrm{almost \, \, everywhere}.\]
Then the operator $\rho L$ is associated with the space $(\cM,\rho^{-1}\mu,\cE)$. Its domain is the set of functions $w\in \cD(\cE)$ such that
there is a constant $C$ with :
\[\forall \varphi\in \cD(\cE)\,\colon \, |\cE(w,\varphi)|\le C\|\varphi\|_{L^2(\cM,\rho^{-1}\mu)}\,.\]
 It is easy to see that $\cD(\rho L)=\cD(L)$. We also have that for any $q\in (1,\infty)$,
 $-\rho L\colon H^{2,q}(\cM)\rightarrow L^q $ is sectorial on $L^q(\cM,\rho^{-1}\mu)=L^q(\cM,\mu)$ and   $H^{2,q}(\cM)=\{f\in L^q(\cM,\mu),\rho Lf\in L^q(\cM,\mu)\}$ 
with
\[\forall f\in H^{2,q}(\cM)\colon\frac{1}{C}\left(\norm{\rho Lf}_{L^q}+\norm{f}_{L^q}\right) \le \norm{Lf}_{L^q}+\norm{f}_{L^q}\le \frac{1}{c}\left(\norm{\rho Lf}_{L^q}+\norm{f}_{L^q} \right),\]
where $L^q=L^q(\cM,\mu)$. We will also need one additional result in order to finish the proof:

\begin{Lem}\label{crucial}Assume that $(\cM,\mu,\cE)$ is a Dirichlet space with  $\mu(\cM)=1$, and assume that it satisfies the Sobolev inequality \eqref{eq:SemiGroupSobolev} of dimension $\nu>2$ . Let $\rho\colon \cM\rightarrow \R_+$ be  a positive measurable function such that for some positive constants $C>c>0\colon$
\[c\le \rho\le C,\ \mu-\mathrm{almost\,\, everywhere}\]
  and $V\in L^p(\cM,\mu)$ with $p>\nu/2$.
 Then the operator $H\coloneqq -(\rho L+V)\colon  H^{2,p}(\cM)\rightarrow L^p(\cM,\mu) $ is sectorial and there is a constant $\eta>0$ depending only on $p,\nu, c,C, \norm{V}_{L^p(\cM,\mu)}$ and of the Sobolev constants $A,B$ such that 
 \[\forall f\in H^{2,p}(\cM)\colon \eta \left(\|H f\|_{L^p}+\|f\|_p\right) \le \|Lf\|_{L^p}+\|f\|_{L^p}\le \frac{1}{\eta }\left(\|H f\|_{L^p}+\|f\|_{L^p} \right).\]
\end{Lem}
We prove this right after finishing the current proof.
We finally need to check the hypothesis of the Theorem \ref{general}. We introduce 
\[\cO:=\{f\in H^{2,p}(M)\,\colon\, \exists\, c>0,\,\  c\le f\ \mu-\mathrm{a.e.}\},\]
and observe that $\cO$ is an open set of $H^{2,p}(M)$.
Let $F\colon \cO\rightarrow L^p(M,\mu)$ be defined by $F(u)=-u^\beta(Lu+Qu)$.
Note that $u\in\cO\mapsto Lu+Qu \in L^p(M,\mu)$ is linear and continuous, and the multiplication $L^\infty\times L^p\mapsto L^p$ is a continuous bilinear map. 
By \eqref{eq:SobolevEmbedding}, $\cO\subset L^\infty(M)$, and on $\{ f\in L^\infty(M),  \colon\, \exists\, c>0\,\colon\,  c\le f\ \mu-\mathrm{a.e.}\}$, the map 
$f\mapsto f^\beta$ is smooth. So $F$ is smooth on $\cO$. Moreover, if $h\in H^{2,p}(M)$ and $u\in \cO$ then the linearisation reads
\[-DF(u)h=\rho Lh+Vh\] 
where
$\rho=u^\beta$ and $V=(\beta+1)u^\beta Q+\beta u^{\beta-1}Lu$. Lemma \ref{crucial} insures that $DF(u)\colon H^{2,p}(M)\rightarrow L^p(M,\mu)$ is sectorial and that its graph norm is equivalent to the $H^{2,p}(M)$-norm.
\par

\endproof
\begin{proof}[Proof of the Lemma \ref{crucial}] It is enough to show that for some large enough $\lambda\colon$
\[\left\| V(\rho L+\lambda)^{-1}\right\|_{L^p\to L^p}<1,\]
and this will be very similar to the proof of \eqref{eq:SobolevEmbedding}
Note that the Dirichlet space $(\cM,\rho^{-1}\mu,\cE)$ satisfies the Sobolev inequality \ref{eq:SemiGroupSobolev} with constants that depend on $A,B$ and of $C,c$, hence we have, as in \eqref{eq:HeatKernelBound}  some constant such that 
for any $t>0$ there holds
\[\left\|e^{-t(\rho L+1)}\right\|_{L^{p}\to L^{\infty}}\le \frac{Ce^{-t/2}}{t^{\frac{\nu}{2p}}}.\]
Hence we can estimate
\[\left\|Ve^{-t(\rho L+1+\lambda)}\right\|_{L^{p}\to L^{p}}\le \norm{V}_{L^p}\frac{Ce^{-t/2-\lambda t}}{t^{\frac{\nu}{2p}}},\]
and, using $(\rho L+1+\lambda)^{-1}=\int_0^\infty e^{-t(\rho L+1+\lambda)}\, dt$, we conclude
\begin{align*}
\left\| V(\sigma L+1+\lambda)^{-1}\right\|_{L^p\to L^p}&\le  \norm{V}_{L^p}\int_0^\infty \frac{Ce^{-t/2-\lambda t}}{t^{\frac{\nu}{2p}}}dt\\ &= C \norm{V}_{L^p} (1/2+\lambda)^{\frac{\nu}{2p}-1}\Gamma\left(1-\frac{\nu}{2p}\right).
\end{align*}
The result then follows  because we assumed $p>\nu/2$.
\end{proof}

 We end this section by applying the above general result to the Yamabe flow.
 \begin{Thm}
 \label{Thm:ShortTime}
 
 Assume a Riemannian manifold $(M,g)$ satisfies Assumption \ref{Y-assump} and \ref{admissible-assump}.  Then the Yamabe flow \eqref{eq:YamabeFlow} has a unique solution $u\in C^{\infty}((0,T),H^{2,p}(M))$, for some $T\leq \infty$, with\footnote{This limit is in $H^{2,p}(M)$, so by \eqref{eq:SobolevEmbedding} this convergence is also in $L^\infty(M)$.} $\lim\limits_{t\to 0} u=u_0$. Furthermore, the scalar curvature $S$ lies in $H^{2,p}(M,g)$ for $t\in (0,T)$ with $\lim\limits_{t\to 0} S=S_0$, where 
 \[H^{2,p}(M,g)\coloneqq \{f\in L^p(M,g) \, \colon\, \Delta f\in L^p(M,g)\}\]
 and $\Delta$ is the Laplacian associated to $g$.
 
 \end{Thm}
 
\begin{proof}
 We may consider the Yamabe flow without normalization (corresponding to removing $\sigma$ from \eqref{eq:YamabeFlow}). From \eqref{eq:YamabeFlow} without $\sigma$, we have
\[\partial_t u=u^{-\frac{4}{n-2}}\left((n-1)\Delta_0 u -\frac{n-2}{4} S_0 u\right),\]
which is an equation of the form \eqref{eq:TimeEvol} with $\beta=-\frac{4}{n-2}$, $L=-(n-1)\Delta_0$, $Q=\frac{n-2}{4}S_0$, and $\nu=n$ (the dimension of $M$). The statement of Theorem \ref{Thm:FlowExistence} is then that the Yamabe flow has a unique solution as long as $u$ is bounded from below and above and as long as the scalar curvature  $S$ remains bounded in $L^q(M,g)$ for some $q>\frac{n}{2}$.  

By the evolution equation \eqref{eq:YamabeFlow} and $\partial_t u\in H^{2,p}(M)$, it follows that $u\cdot S\in H^{2,p}(M)$ for $t\in (0,T)$. Hence $uS\in L^{\infty}(M)$ for $t\in (0,T)$. But $u$ is bounded above and below for $t\in (0,T)$, so $L\in L^\infty(M)$ for $t\in (0,T)$. For the Laplacian, we use $uS\in H^{2,p}(M)$ and compute
\[\Delta_0(uS)=u\Delta_0 S +2\ip{\nabla u}{\nabla S}_{g_0}+S\Delta_0 u\in L^p(M).\]
The Laplacian associated to $g$ is given by
\[\Delta \Phi =u^{-\frac{4}{n-2}}\left(\Delta_0 \Phi+2u^{-1}\ip{\nabla u}{\nabla \Phi}_{g_0}\right),\]
so
\[\Delta_0(uS)=u^{\frac{n+2}{n-2}}\Delta S+S\Delta_0 u,\]
or
\[\Delta S=u^{-\frac{n+2}{n-2}}\Delta_0(uS)-Su^{-\frac{n+2}{n-2}}\Delta_0 u.\]
Now $u^{-\frac{n+2}{n-2}}\in L^{\infty}(M)$ and $Su^{-\frac{n+2}{n-2}}\in L^{\infty}(M)$ and $\Delta_0(uS), \Delta_0u\in L^p(M)$. Hence $\Delta S\in L^p(M)=L^p(M,g)$. 

\end{proof}

%%%%%%%%%%%%%%%%%%%%%%%%%%%%%%%%%
\section{Gain in regularity and uniform bounds for scalar curvature}\label{scalar gain}
%%%%%%%%%%%%%%%%%%%%%%%%%%%%%%%%%

We recall some evolution equations and inequalities for the scalar curvature and consequences. A proof can be found in \cite[Lemma 2.1]{LongTime2}.

\begin{Lem}
Let $g= u^{\frac{4}{n-2}}g_0$ 
be a family of metrics evolving according to the normalized Yamabe flow 
\eqref{eq:YamabeFlow}.  Then
$S$ evolves according to
\begin{equation}
\partial_t S-(n-1)\Delta S= S(S-\sigma).
\label{eq:ScalarEvol}
\end{equation}
where $\Delta$ denotes the Laplacian with respect to the time-evolving metric $g$.

We write $S_{+}:= \max\{S,0\}$ and $S_{-}:= -\min\{S,0\}$. 
Then $S_{\pm}\in H^1(M,g)$ for all time and satisfy
\begin{align}
&\partial_t S_{+}-(n-1)\Delta S_{+} \leq S_+(S_+-\sigma),
\label{eq:S+Evol} \\
&\partial_t S_- -(n-1)\Delta S_- \leq -S_-(S_-+\sigma).
\label{eq:S-Evol}
\end{align} 
\end{Lem}

\begin{Rem} The equation \eqref{eq:ScalarEvol}
is to be understood in the weak sense: 
for any compactly supported 
smooth test function $\phi \in C^\infty_c(M)$ we have
\begin{equation*}
\int_M \partial_t S \cdot \phi \, d\Vol_g + (n-1) \int_M \ip{\nabla S}{\nabla \phi}_g 
d\Vol_g = \int_M S(S-\sigma)  \cdot \phi \,  d\Vol_g.
\end{equation*}
Similarly for the partial differential inequalities \eqref{eq:S+Evol} and \eqref{eq:S-Evol}
and $\phi \geq 0$
\begin{equation*}
\int\limits_M \partial_t S_\pm  \cdot \phi \,  d\Vol_g + (n-1) \int\limits_M \ip{\nabla S_\pm}{\nabla \phi}_g 
d\Vol_g \leq \pm \int\limits_M S_\pm(S_\pm \mp \sigma)  \cdot \phi \,  d\Vol_g.
\end{equation*}
By Assumption \ref{admissible-assump}, $C^\infty_c(M)$
is dense in $H^1(M) = H^1(M,g)$.
Hence we can as well assume $\phi \in H^1(M,g)$ in the weak formulation above.
\end{Rem}

\begin{Prop}{\cite[Proposition 2.3, Lemma 4.2]{LongTime2}}
\label{Prop:SEvol}
Let $S= S(t)\in  H^1(M,g)$ with $(S_0)_-\in L^{\infty}(M)$ and $S_0 \in L^{q}(M)$. for some $q>\frac{n}{2}$. Then 
\begin{align}
&\norm{S}_{L^{\frac{n}{2}}(M,g)}\leq \norm{S_0}_{L^{\frac{n}{2}}(M)},
\label{eq:Sn/2Bound} \\
&\norm{S_+}_{L^{\frac{n}{2}}(M,g)}\leq \norm{(S_0)_+}_{L^{\frac{n}{2}}(M)},
\label{eq:S+n2Bound} \\
&\norm{S_-}_{L^{\infty}(M)}\leq \norm{(S_0)_-}_{L^{\infty}(M)}.
\label{eq:LowerSBound}
\end{align}

\end{Prop}

\begin{Rem}
The main advantage of splitting the evolution equation into $S_{\pm}$ is that \eqref{eq:S-Evol} gives us the lower bound \eqref{eq:LowerSBound} \textit{without} having to appeal to the maximum principle, which in general fails to hold in our setting.
\end{Rem}

\begin{Rem}
\label{Rem:S_0Bounded}
The condition $(S_0)_-\in L^{\infty}(M)$ is harmless. For $S_0\in L^{q}(M)$  for some $q>\frac{n}{2}$, Theorem \ref{Thm:ShortTime} ensures $S(t)\in L^{\infty}(M)$ for small $t>0$. So by redefining the starting time, we may without loss of generality assume $S_0\in L^{\infty}(M)$.
\end{Rem}
\medskip
In \cite{LongTime2}, the last two authors proved that the Yamabe flow exists for all time under slightly stronger assumptions on the scalar curvature than in this work. We therefore sketch a modified proof of this fact here.

\begin{Prop}
\label{Prop:LongTime}
 Assume a Riemannian manifold $(M,g)$ satisfies Assumption \ref{Y-assump} and \ref{admissible-assump}. Then the maximal flow time of the Yamabe flow is $T=\infty$.
\end{Prop}
\begin{proof}
By Theorem \ref{Thm:ShortTime}, we need to show that for any finite time $T>0$, one can find bounds $c(T),C(T)>0$ and some $q>\frac{n}{2}$ such that $c(T)\leq u\leq C(T)$ and $\norm{S}_{L^q(M,g)}\leq C(T)$. By   \cite[Prop. 3.1, Theorem 3.2]{LongTime2}, whose proofs still hold, we get the upper and lower bounds on $u$. By redefining the starting time to be slightly later, we may assume $S_0\in L^\infty(M)$, so \cite[Theorem 4.1]{LongTime2} works to ensure $S\in L^\infty(M)$ for any finite time $T$. Alternatively the arguments of \cite[Lemma 2.5]{Brendle} go through verbatim, since they only rely on (cleverly) choosing test functions in the evolution equations and using  the Sobolev inequality \eqref{eq:Sobolev}, and this gives a bound on $\norm{S}_{L^q(M,g)}$ for some $q>\frac{n}{2}$ for any finite $T$.
\end{proof}

To study convergence, we need to get time-independent bounds.
In what follows, we combine results and arguments from \cite{Struwe}, \cite{Brendle}, and \cite{LongTime2} to deduce time-independent bounds on the scalar curvature $S$. Let us start by recording a little lemma which allows us to apply the chain rule.
\begin{Lem}
Assume $v\in H^1(M)\cap L^{\infty}(M)$ and $f\in C^1(\R)$. Then $f\circ v\in H^1(M)\cap L^{\infty}(M)$ and the chain rule applies;
\[\nabla (f\circ v)=f'(v)\nabla v.\]
\label{Lem:Chain}
\end{Lem}
\begin{proof}
The composition $f(v)$ is bounded, and thus in $L^2(M)$ (since $\Vol(M)<\infty$). The function $f'(v)$ is bounded, hence $f'(v)\nabla v\in L^2(M)$.  
\end{proof}

\begin{Rem}
This is similar to the lemma which was used in \cite{LongTime2} and \cite{ACM}, namely that if $v\in H^1(M)$ and $f\in C^1(\R)$ with $f'\in L^{\infty}(\R)$, then $f\circ v\in H^1(M)$ and the chain rule holds.
\end{Rem}

For convergence we need to derive some bounds which do not blow up as $T\to \infty$. According to Proposition \ref{Prop:SEvol}, we already have such bounds on $\norm{S_-}_{L^{\infty}(M)}$ and  $\norm{S}_{L^{\frac{n}{2}}(M,g)}$ respectively. We start with the simplest one.
\begin{Prop}
\label{Prop:SL2Conv}
\begin{equation}
\lim_{t\to \infty} \int_M (S(t)-\sigma(t))^2\, d\Vol_{g(t)}=0.
\label{eq:RhoAtInf}
\end{equation}

\end{Prop}
\begin{proof}
We will only prove this for $n\geq 4$. It is true for $n=3$ as well, but the proof requires the general machinery 
which we introduce when discussing Proposition \ref{Prop:Sn/2+delta}. \medskip

\noindent
We recall that by Remark \ref{Rem:S_0Bounded}, we may assume that $S_0\in L^{\infty}(M)$.

\noindent
By the monotonicity of $\sigma$ \eqref{eq:rhoEvol} and the fact that $\sigma$ is bounded from below by $Y(M,g_0)$, we conclude that 
\begin{equation}
\lim_{T\to \infty} \int_0^T \sigma'(t)\, dt=-\frac{n-2}{2}\lim_{T\to \infty} \int_0^T \int_M \left( S(t)-\sigma(t)\right)^2 d\Vol_g\, dt
\label{eq:L2Integral}
\end{equation}
 exists.
From the evolution equations \eqref{eq:ScalarEvol} and \eqref{eq:DVol}, we compute
\begin{align}
&\partial_t \int_M (S-\sigma)^2\, d\Vol_g =2\int_M \partial_t S(S-\sigma)\, d\Vol_g +\int_M (S-\sigma)^2\, \partial_t d\Vol_g \notag \\
&=-2(n-1) \int_M \vert \nabla (S-\sigma)\vert^2_g \, d\Vol_g +2\int_M S(S-\sigma)^2\d\Vol_g -\frac{n}{2}\int_M (S-\sigma)^3\, d\Vol_g\notag \\
&= -2(n-1) \int_M \vert \nabla (S-\sigma)\vert^2_g \, d\Vol_g \notag \\ &+\left(2-\frac{n}{2}\right) \int_M S(S-\sigma)^2\, d\Vol_g +\frac{n}{2}\sigma \int_M (S-\sigma)^2\, d\Vol_g.
\label{eq:L2EvolEq}
\end{align}
For $n\geq 4$, we use Proposition \ref{Prop:SEvol} to approximate this as
\begin{align*}\partial_t \int_M (S-\sigma)^2\, d\Vol_g &\leq \left(\frac{n}{2}-2\right) \norm{ (S_0)_-}_{L^{\infty}(M)} \int_M (S-\sigma)^2\, d\Vol_g 
\\ &+\frac{n}{2}\sigma(0) \int_M (S-\sigma)^2\, d\Vol_g,
\end{align*}
meaning
\begin{equation}
\partial_t \norm{S-\sigma}_{L^2(M,g)}^2\leq C \norm{S-\sigma}^2_{L^2(M,g)}
\label{eq:F2Der}
\end{equation}
for $C=\left(\frac{n}{2}-2\right) \norm{ (S_0)_-}_{L^{\infty}(M)}+\frac{n}{2}\sigma(0)$, a time-independent constant. 
Writing 
\[F_2(t)\coloneqq \norm{S(t)-\sigma(t)}_{L^2(M,g(t))}^2,\] 
 and integrating \eqref{eq:F2Der} from $s$ to $t$ yields
\[F_2(t)\le F_2(s)+C \int_{s}^t F_2(\tau)d\tau.\]
After integrating this again for $s\in (t-1,t)$, we get 
\[F_2(t)\le (C+1) \int_{t-1}^t F_2(\tau)d\tau\le (C+1)\int_{t-1}^{+\infty} F_2(\tau)d\tau.\]
Hence
\[\limsup_{t\to \infty} F_2(t)=0.\]
\end{proof}

We first show that we can do slightly better than a uniform $n/2$-norm bound on $S$. The arguments in the following proof are essentially due to \cite{SS, Brendle, Struwe}. We write some of the details, however to demonstrate how the arguments go through in our setting. A key observation is that Lemma \ref{Lem:Chain} along with Theorem \ref{Thm:ShortTime} (which in particular says $S(t)\in L^\infty(M)$ for finite times) allow us to use the chain rule freely.

\begin{Prop}[{\cite[Lemma 3.3]{SS}, \cite[Proposition 3.1]{Brendle}, \cite[4.14]{Struwe}}]
\label{Prop:Sn/2+delta}
For any $1<p<\frac{n}{2}+1$ 
\[\lim_{t\to \infty} \norm{S(t)-\sigma(t)}_{L^p(M,g(t))}=0.\]
In particular, there exists, $C>0$ independent of time $t$ and $q>\frac{n}{2}$ such that 
\begin{equation}
\norm{S(t)}_{L^{q}(M,g(t))}\leq C.
\end{equation}
\end{Prop}
\begin{Rem}
One can improve upon this, and use the arguments from \cite[pp. 70-71]{SS} to deduce $\lim\limits_{t\to \infty}\norm{S(t)-\sigma(t)}_{L^p(M,g(t))}=0$ for any $p<\infty$.
\end{Rem}

\begin{proof}
We follow \cite[pp. 68-71]{SS} with minor modifications. Introduce 
\[F_q(t)\coloneqq \int_M \vert S-\sigma\vert^q\, d\Vol_g =\norm{S-\sigma}^q_{L^q(M,g)}.\]
The first thing we need to establish is that for $p<\frac{n}{2}+1$, we have
\begin{equation}
\int_T^\infty F_p(t)\,dt <\infty.
\label{eq:FpInt}
\end{equation}
The argument for this is easier for $n=4$ than the general case $n\geq 3$, so we show this first.
\medskip

\noindent\textbf{The case $n=4$.} We return to \eqref{eq:L2EvolEq}, which for $n=4$ reads
\[\partial_t F_2(t)=-6\int_M \vert \nabla (S-\sigma)\vert^2_g\, d\Vol_g+\sigma F_2(t).\]
By the conformal invariance of the Yamabe constant, \eqref{eq:YamabeConst}, we get the Yamabe-inequality
\begin{equation}
Y(M,[g_0])\norm{f}^2_{L^{\frac{2n}{n-2}}(M,g)}\leq 4\frac{n-1}{n-2}\norm{\nabla f}^2_{L^2(M,g)}+ \int S f^2\, d\Vol_g,
\label{eq:YamabeIneq}
\end{equation}
and we use this with $f= S-\sigma$ (and $n=4$) to say
\[  -6\int_M \vert \nabla (S-\sigma)\vert^2_g\, d\Vol_g\leq -Y(M,[g_0])F_4(t)^{\frac{1}{2}} +\int_M S(S-\sigma)^2\, d\Vol_g.\]
The last term we handle using the Hölder inequality
\begin{align*}
\int_M S(S-\sigma)^2\, d\Vol_g&\leq \int_M \vert S-\sigma\vert^3\, d\Vol_g +\sigma \int_M(S-\sigma)^2\, d\Vol_g \\
&\leq F_2(t)F_4(t)^{\frac{1}{2}}+\sigma F_2(t).
\end{align*}
Combined, we have
\[(Y-F_2(t))F_4(t)^{\frac{1}{2}}\leq -\partial_t F_2(t)+2\sigma F_2(t),\]
where we write $Y\coloneqq Y(M,[g_0])$. 
By Proposition \ref{Prop:SL2Conv}, we know $F_2(t)\xrightarrow {t\to \infty} 0$, and since $Y>0$, there is some $T>0$ such that $Y-F_2(t)>\frac{Y}{2}>0$ for all $t\geq T$.
Integrating, we thus find
\[ \int_T^\infty F_4(t)^{\frac{1}{2}}\, dt \leq \frac{2}{Y} \left(F_2(0)+2\sigma(0)\int_T^\infty F_2(t)\, dt\right)<\infty.\]
The H\"{o}lder inequality again and the boundedness of $F_2$ therefore yield
\[\int_T^\infty F_3(t)\, dt\leq \int_T^\infty F_2(t)F_4(t)^{\frac{1}{2}}\,dt <\infty,\]
which establishes \eqref{eq:FpInt} for $n=4$. \medskip

\noindent\textbf{The general case.} The basic idea is still the same, but the estimates are more intricate.
Let $1< 2\beta< \frac{n}{2}$ be arbitrary, 
and  let $\gamma>\norm{(S_0)_-}_{L^\infty}$ be a constant, so that by Proposition \ref{Prop:SEvol}, $S+\gamma>0$ for all time.  
Using $(S+\gamma)^{2\beta-1}$ as a test function\footnote{Thanks to Theorem \ref{Thm:ShortTime}, the scalar curvature $S$ is bounded for time $0<t< T$. The functions $(S+\gamma)^{\beta}$ is therefore in $H^1(M,g)\cap L^{\infty}(M)$ and the chain rule applies. So we do not need approximation functions here as in \cite{ACM} and \cite{LongTime2}. } in \eqref{eq:ScalarEvol} and that $\partial_t d\Vol_g=-\frac{n}{2}(S-\sigma)d\Vol_g$, we arrive at 
\begin{align}
&\frac{\beta}{2(2\beta-1)} \partial_t \int_M (S+\gamma)^{2\beta}\, d\Vol_g +(n-1)\int_M \vert \nabla  (S+\gamma)^{\beta}\vert^2\, d\Vol_g \notag \\
&=\frac{\beta}{2\beta-1} \int_M \beta S(S-\sigma)(S+\gamma)^{2\beta-1}-\frac{n}{4}(S-\sigma)(S+\gamma)^{2\beta}\, d\Vol_g\notag \\
&= \frac{\beta}{2\beta-1} \left( \beta-\frac{n}{4}\right) \int_M  (S+\gamma)^{2\beta}(S-\sigma) \, d\Vol_g \notag \\
&-\frac{\beta^2\gamma}{2\beta-1} \int_M (S+\gamma)^{2\beta-1}(S-\sigma)\, d\Vol_g.
\label{eq:n/2+deltaStep1}
\end{align}
Since 
\[\int_M  (S-\sigma)\, d\Vol_g=0,\]
we may add multiples of this freely in the above expression.
We may therefore write
\begin{align}
&\frac{\beta}{2(2\beta-1)} \partial_t \int_M (S+\gamma)^{2\beta}\, d\Vol_g +(n-1)\int_M \vert \nabla (S+\gamma)^{\beta}\vert^2\, d\Vol_g \notag\\
&= \frac{\beta}{2\beta-1} \left( \beta-\frac{n}{4}\right) \int_M(S-\sigma)\left( (S+\gamma)^{2\beta}-(\gamma+\sigma)^{2\beta}\right)\, d\Vol_g\notag \\
&- \frac{\beta^2\gamma }{2\beta-1} \int_M (S-\sigma)\left((S+\gamma)^{2\beta-1}-(\sigma+\gamma)^{2\beta-1}\right)\, d\Vol_g. 
\label{eq:Intermediate}
\end{align}
We now need the elementary estimate that for any $p\geq 1$, we have 
\[ (x^p-y^p)(x-y) \geq  \vert x-y\vert^{p+1}\]   
 for any $x,y\in \R_{\geq 0}$. The way to see this, is to observe that the function $f(t)\coloneqq (1-t^p)^{1/p}$ for $t\in [0,1)$ is concave.
 Using this bound with $x=S+\gamma$, $y=\sigma+\gamma$, (and recalling that $2\beta>1$) we have
  \[(S-\sigma)\left( (S+\gamma)^{2\beta}-(\gamma+\sigma)^{2\beta}\right)\geq \vert S-\sigma\vert^{2\beta+1}.\]
  The very last term we handle slightly differently\footnote{If $2\beta>2$ one can use the same argument, but this does not work for $n=3$ or $n=4$.}, and observe that we have $(x-y)(x^p-y^p)\geq 0$ for any $p>0$ and $x,y>0$. Using this with $x=S+\gamma$, $y=\sigma+\gamma$ to discard the very last term in \eqref{eq:Intermediate}, we arrive at
\begin{align}
&\frac{\beta}{2(2\beta-1)} \partial_t \int_M (S+\gamma)^{2\beta}\, d\Vol_g +(n-1)\int_M \vert \nabla (S+\gamma)^{\beta}\vert^2\, d\Vol_g \notag \\
&\leq \frac{\beta}{2\beta-1} \left( \beta-\frac{n}{4}\right) \int_M\vert S-\sigma\vert^{2\beta+1}\, d\Vol_g 
\label{eq:n/2+deltaStep2}
\end{align}
We drop (for now) the gradient term and deduce
\[\partial_t \int_M (S+\gamma)^{2\beta}\, d\Vol_g  \leq 2\left( \beta-\frac{n}{4}\right) \int_M\vert S-\sigma\vert^{2\beta+1}\, d\Vol_g.\]
We integrate this and get
\[\int_0^t \int_M\vert S(\tau)-\sigma(\tau)\vert^{2\beta+1}\, d\Vol_{g(\tau)} \, d\tau\leq \frac{2}{4\beta-n}\norm{S_0+\gamma}^{2\beta}_{L^{2\beta}(M)}.\]
We conclude that
\[\lim_{t\to \infty} \int_0^t \int_M\vert S(\tau)-\sigma(\tau)\vert^{2\beta+1}\, d\Vol_{g(\tau)} \, d\tau\]
exists for $1<2\beta<\frac{n}{2}$, hence
\[\liminf_{t\to \infty}\int_M\vert S(\tau)-\sigma(\tau)\vert^{2\beta+1}\, d\Vol_{g(\tau)} \, d\tau=0.\]
This proves \eqref{eq:FpInt} for $n\geq 3$.
\medskip

\noindent
 To deduce that $\lim\limits_{t\to \infty}  F_{2\beta+1}(t)=0$, we need bounds on $F_{2\beta+1}$, exactly as in the proof of Proposition \ref{Prop:SL2Conv}. For this we need \cite[Equation 39]{SS} or \cite[Lemma 2.3]{Brendle}, which states that for $p>\max\left\{\frac{n}{2},2\right\}$, there is a time-independent $C>0$ such that
\[\frac{d}{dt} F_p(t) \leq C F_p(t)^{\frac{2p-n+2}{2p-n}}+CF_p(t).\]
The proof of this differential inequality is via similar arguments to the ones used so far, using the evolution equation for $S$, the H\"{o}lder inequality, and the Yamabe inequality \eqref{eq:YamabeIneq} to estimate the gradient term in \eqref{eq:n/2+deltaStep2}). We leave out these argument. From here, one follows the arguments from \cite[pp. 69-71]{SS} to deduce $\limsup\limits_{t\to \infty} F_{2\beta+1}(t)=0$ as well, hence the claim.
\end{proof}

The above gain in regularity is sufficient to guarantee a uniform Sobolev inequality for $H^1(M,g)$.
\begin{Prop}
\label{Prop:Sobolev}
Assume $g=g(t)=u(t)^{\frac{4}{n-2}}g_0$ is a family of metrics so that there exist (time-independent) constants $C>0$ and $q>\frac{n}{2}$ so that  $\norm{S}_{L^{q}(M,g)}\leq C$ holds for all time. Assume the Yamabe constant is positive, $Y(M,g_0)>0$. Then
the Sobolev inequality holds for all $f\in H^1(M,g)$ independently of time, meaning one can find time-independent constants $A,B>0$ such that
\begin{equation}
\norm{f}^2_{L^{\frac{2n}{n-2}}(M,g)}\leq A \norm{\nabla f}_{L^2(M,g)}^2 + B\norm{f}^2_{L^2(M,g)}
\label{eq:SobolevForg}
\end{equation}
holds for all $f\in H^1(M,g)$.
\end{Prop}

\begin{proof}
We introduce $c_n=4\frac{n-2}{n-1}$. As mentioned in the proof of Proposition \ref{Prop:Sn/2+delta}, the conformal invariance of the Yamabe constant immediately gives that for any $f\in H^1(M,g)$,
\begin{equation}
Y(M,g_0) \norm{f}^2_{L^{\frac{2n}{n-2}}(M,g)}\leq  c_n\norm{ \nabla f}^2_{L^2(M,g)} +\int_M S f^2\, d\Vol_g.
\label{eq:YamabeInv}
\end{equation}
We handle the last term using the H\"{o}lder inequality with $q$ and $p=\frac{q}{q-1}$;
\[\int_M S f^2\, d\Vol_g\leq \norm{S}_{L^q(M,g)} \norm{f ^2}_{L^p(M,g)}\leq C\norm{f ^2}_{L^p(M,g)},\]
where we have inserted the assumed bound on $ \norm{S}_{L^q(M,g)}$. Since $q>\frac{n}{2}$, we have $1<p<\frac{n}{n-2}$, and we may interpolate between the $L^1$ and the $L^{\frac{n}{n-2}}$-norms as follows\footnote{The general statement is this. Fix $p_0<p<p_1$ and choose $\theta$ so that $\frac{1}{p}=\frac{1-\theta}{p_0}+\frac{\theta}{p_1}$. Then $\norm{f}_{L^p}\leq \norm{f}_{L^{p_0}}^{1-\theta}\norm{f}^{\theta}_{L^{p_1}}$, and one checks this by applying H\"{o}lder to $f=f^{1-\theta} f^{\theta}$.}  
\[\norm{f^2}_{L^p(M,g)}\leq \norm{f^2}_{L^1(M,g)}^{1-\theta} \norm{f^2}^\theta_{L^{\frac{n}{n-2}}(M,g)},\]
where $\theta=\frac{n}{2q}<1$. To this product we apply Young's inequality $ab\leq \theta (\epsilon^\theta a)^{\frac{1}{\theta}} +(1-\theta)(\epsilon^{-\theta}b)^{\frac{1}{1-\theta}}$ for any $\epsilon>0$ to deduce
\[\norm{f^2}_{L^p(M,g)}\leq  \theta \epsilon \norm{f}^2_{L^{\frac{2n}{n-2}}(M,g)}+ (1-\theta) \epsilon^{-\frac{\theta}{1-\theta}} \norm{f}^2_{L^2(M,g)}.\]
Inserting this back into \eqref{eq:YamabeInv} and abbreviating $Y\coloneqq Y(M,g_0)$ leaves us with
\[Y \norm{f}^2_{L^{\frac{2n}{n-2}}(M,g)}\leq c_n\norm{ \nabla f}^2_{L^2(M,g)} + C\left( \theta \epsilon \norm{f}^2_{L^{\frac{2n}{n-2}}(M,g)}+ (1-\theta) \epsilon^{-\frac{\theta}{1-\theta}} \norm{f}^2_{L^2(M,g)}\right),\]
which can be written as
\[\left( Y- C \theta \epsilon\right) \norm{f}^2_{L^{\frac{2n}{n-2}}(M,g)}\leq c_n\norm{ \nabla f}^2_{L^2(M,g)}+ C(1-\theta) \epsilon^{-\frac{\theta}{1-\theta}} \norm{f}^2_{L^2(M,g)}.\]
Choosing $\epsilon$ small enough ensures the left hand side is non-negative (here we are using $Y(M,g_0)>0$) and we deduce \eqref{eq:SobolevForg}. 
\end{proof}

%%%%%%%%%%%%%%%%%%%%%%%%%%%%%%%%%
\section{Concentration--compactness dichotomy and bubbling}\label{Section:Dichotomy}
%%%%%%%%%%%%%%%%%%%%%%%%%%%%%%%%%
\label{Section:Convergence}

In this section we turn to the convergence of the solution $u$ and the measure $d\Vol_g$. 
As already noted, the average scalar curvature $\sigma(t)$ always converges, being a monotone (see \eqref{eq:rhoEvol}) and bounded (by $Y(M,g_0)$) function. We write 
\[\sigma_{\infty}=\lim_{t\to \infty} \sigma(t).\]
In the classic (smooth and compact) setting there is a famous dichotomy describing what can happen to solutions of equations like \eqref{eq:YamabeFlow} as $t\to \infty$. See \cite[Theorem 3.1]{SS}.  We formulate the analogue as Theorem \ref{Prop:Dichotomy} in our setting below.  

\noindent
We will assume in this section and in the next one that $M$ satisfies also  Assumption \ref{P-assump}. \medskip

\noindent
We start by some analytic preliminaries. The arguments given here will be valid on any Dirichlet space satisfying the requirements of Definition \ref{Def:DirichletSpace}.
\subsection{Analytic tools}

To prove the dichotomy, we need a kind of Harnack inequality which we state but do not prove, referring instead to \cite[Section 4]{ACM2} (see Remark \ref{Rem:Dirichlet} for how to translate into the language of Dirichlet spaces).
\begin{Prop}
\label{regularH} 
Let $B(x,2r)\subset \overline{M}$ be an open ball of radius $2r$ around $x\in \overline{M}$. Let $w\colon B(x,2r)\to \R$ be a weak solution to the equation
\[-\Delta_0 w=Vw,\]
where the potential $V\colon B(x,2r)\to \R$ satisfies
\[r^{2p} \fint_{B(x,2r)} \vert V\vert^p\, d\mu\coloneqq \frac{r^{2p}}{\mu(B(x,2r))} \fint_{B(x,2r)} \vert V\vert^p\, d\mu  \leq \Lambda \]
for some $\Lambda>0$ and $p> \frac{n}{2}$.
Then there is $\alpha\in (0,1)$ depending on $n$ and $p$ and a constant $C=C(n,p,\Lambda)$ such that 
\[\sup_{y\in B(x,r)}|w(y)|^2\le C \fint_{B(x,2r)} w^2\,d\mu\]
and $w$   has a H\"older continuous representative on $B(x,r)$ satisfying
\[ |(w(y)-w(z)|\le C\ \left(\frac{d(y,z)}{r}\right)^\alpha \,  \left(\fint_{B(x,2r)} w^2d\mu\right)^{\frac 12} \,\,\,\, \forall\, y,z\in B(x,r).\]
Finally, if $w\ge 0$ then $w$ is essentially positive and
\[\mathrm{ess} \sup_{B(x,r)} w\le C\mathrm{ess}\inf_{B(x,r)} w,\]
where $\mathrm{ess}\sup$ and $\mathrm{ess}\inf$ are the essential supremum and infimum, respectively.
\end{Prop}
The next lemma is the key estimate that we will need in the proof of our result.
 \begin{Lem}
 \label{Lem:LocaluBound}
 Assume that $u\in H^{2,p}(M)$ is such that the conformal metric $g=u^{\frac{4}{n-2}} g_0$ satisfies
\[\int_M u^{\frac{2n}{n-2}}d\mu=1\, \text{ and }\, \int_M |S|^p d\Vol_g \le \Lambda\]
for some $\Lambda>0$ and $p>\frac{n}{2}$. 
If, for some $r<\diam (\overline{M},g)$ and all $x\in \overline{M}$, we have 
\begin{equation}
\left(\int_{B(x,r)} S_+^{n/2} d\Vol_g \right)^{\frac 2n}\le (1-\varepsilon) \cS(B(x,r)),
\label{eq:LocalSBound}
\end{equation} 
for some $\varepsilon\in (0,1]$, 
then there is a constant $C$ depending only on the geometry of $M$, $p$, $r$, $\varepsilon$ and $\Lambda$ such that 
\[\text{on}\ B(x,r/8)\colon C^{-1}\le u\le C\ \text{ and }\  \int_{B(x,r/8)} |\Delta_0 u|^p d\mu\le C.\]
 \end{Lem}
\begin{Rem}
The factor $\frac{1}{8}$ is not relevant, the same conclusion holds on any balls $B(x,\uptheta r)$ with $\uptheta\in (0,1)$, but with a constant $C$ depending also on $\uptheta$.
\end{Rem}
\begin{Rem}\label{rem:regularity}
Using a slightly more elaborate argument (see the proof of \cite[Proposition 1.8]{ACM}), we can a priori only assume that $u\in H^1(M)$ solves weakly
the Yamabe equation \[-c_n\Delta_0 u+S_0 u=S u^{\frac{n+2}{n-2}}.\]
In particular, if $u\in H^1(M),\ u\ge 0$ solves weakly
the Yamabe equation \[-c_n\Delta_0 u+S_0 u=\sigma u^{\frac{n+2}{n-2}},\]
where $\sigma$ is a constant, then $u\in H^{2,p}(M)$. And if $u$ is not identically zero, then there is a positive constant $C$ such that $$ C^{-1}\le u\le C.$$
\end{Rem}

\proof 
We may assume that  $\cS(B(x,r))>0$. The reason being that \cite[Lemma 1.3]{ACM} tells us that $\lim\limits_{r\to 0}\cS(B(x,r))\geq Y_{\ell}>0$.
We may thus also assume 
\[\cS_r:=\inf_{x\in \overline{M}} \cS(B(x,r))>0.\]
Recall the formula \eqref{eq:EvolvingScalar} for the scalar curvature of $g$:
\[-c_n \Delta_0 u+S_0 u=S u^{\frac{n+2}{n-2}},\]
with $c_n=4\frac{n-1}{n-2}$. For any $\beta>1$ we know that 
\[-\Delta_0 u^\beta\le -\beta u^{\beta-1} \Delta_0 u.\]
Let $\chi$ be a cut-off function  with support in $B(x,r)$ and which is $1$ on $B(x,r/2)$. We may assume $\chi$ is Lipschitz with Lipschitz constant $\frac{2}{r}$. Using the definition \eqref{eq:SobolevConst} of the Sobolev constant,  we get
\begin{align*}
\cS(B(x,r))\, \left(\int_{B(x,r)}\left(\chi u^\beta\right)^{\frac{2n}{n-2}}d\mu\right)^{1-\frac 2n}&\leq  c_n\int_{B(x,r)} \vert \nabla (\chi u^\beta)\vert^2_{g_0}\, d\mu\\
 &=  +c_n\int_{B(x,r)} u^{2\beta}\vert \nabla \chi \vert^2_{g_0}-c_n\int_{B(x,r)}u^{\beta}\chi ^2 \Delta_0 u^\beta d\mu\end{align*}
 Using \eqref{eq:EvolvingScalar}, the last term becomes
 \begin{align*}
 &-c_n\int_{B(x,r)}u^{\beta}\chi ^2 \Delta u^\beta d\mu \le  \beta \norm{(S_0)_{-}}_{L^{\infty}(M)}\int_{B(x,r)}(\chi u^{\beta})^2 d\mu+\beta\int_{B(x,r)} S_+ u^{\frac{4}{n-2}}\left(\chi u^\beta\right)^2d\mu\\
 &\le \beta\norm{(S_0)_-}_{L^{\infty}(M)}\int_{B(x,r)}(\chi u^{\beta})^2 d\mu+\beta \cS(B(x,r))(1-\varepsilon) \left(\int_{B(x,r)}\left(\chi u^\beta\right)^{\frac{2n}{n-2}}d\mu\right)^{1-\frac 2n},
\end{align*}
where in the last line, we use the H\"older inequality and \eqref{eq:LocalSBound}.
Hence we get 
\begin{align*}
&\left(1-\beta(1-\varepsilon)\right) \cS(B(x,r))\left(\int_{B(x,r)}\left(\chi u^\beta\right)^{\frac{2n}{n-2}}d\mu\right)^{1-\frac 2n} \\
&\leq \left(\frac{4c_n}{r^2}+ \beta\norm{(S_0)_-}_{L^{\infty}(M)} \right) \int_{B(x,r)}u^{2\beta}d\mu,
\end{align*}
where we have used $\vert \nabla \chi \vert^2_{g_0}\leq \frac{4}{r^2}$ and $\chi^2\leq 1$.
We can choose $\beta\in (1,n/(n-2))$ such that $1-\beta(1-\varepsilon)>0$ and we find 
\begin{equation}
\left(\int_{B(x,r/2)}u^{\frac{2\beta n}{n-2}}d\mu\right)^{1-\frac 2n}\le C(r,\epsilon,\norm{(S_0)_-}_{L^\infty(M)} )
\label{eq:betauBound}
\end{equation}
Using \eqref{eq:EvolvingScalar} again, we write
\[- \Delta_0 u=\frac{1}{c_n}\left(-S_0 +S u^{\frac{4}{n-2}}\right) u \eqqcolon Vu.\]
Moreover, for $\alpha$ defined by
\[\frac{1}{\alpha}=\frac{1}{p}\left(1-\frac 1\beta\right)+\frac 2 n \frac 1\beta,\] 
we may use the bound on $\norm{S}_{L^p(M,g)}\leq \Lambda$ and \eqref{eq:betauBound} to deduce
\begin{equation}
\left(\int_{B(x,r/2)} |V|^\alpha d\mu\right)^{\frac1 \alpha} 
\le C\left(r,\varepsilon, \Lambda, \norm{S_0}_{L^\infty(M)}\right),
\label{eq:LocalVBound}
\end{equation}
where the argument runs as follows.
\[\norm{V}_{L^{\alpha}(B(x,r/2))}\leq \frac{1}{c_n}\norm{S_0}_{L^\infty(B(x,r/2))} +\frac{1}{c_n} \norm{S_+u^{\frac{4}{n-2}}}_{L^\alpha(B(x,r/2))}.\]
The second term we handle by writing $(S_+u^{\frac{4}{n-2}})^{\alpha}=\left(S_+^\alpha u^{\frac{2\alpha n}{p(n-2)}}\right)\cdot u^{\frac{2\alpha}{n-2} (2-\frac{n}{p})}$ and use the H\"{o}lder inequality with $\frac{p}{\alpha}>1$ as exponent
\begin{align*}
&\norm{\left(S_+ u^{\frac{2 n}{p(n-2)}}\right)\cdot u^{\frac{2}{n-2} (2-\frac{n}{p})}}^{\alpha}_{L^{\alpha}(B(x,r/2))} \\
&\leq 
\norm{S_+^\alpha u^{\frac{2\alpha n}{p(n-2)}}}_{L^{p/\alpha}(B(x,r/2))}\norm{u^{\frac{2\alpha}{n-2} (2-\frac{n}{p})}}_{L^{\frac{p}{p-\alpha}}(B(x,r/2))} \\
& =\norm{S_+}_{L^p(B(x,r/2),g)}^\alpha \norm{u}_{L^{\frac{2n\beta}{n-2}}(B(x,r/2))}^{\kappa},
\end{align*}
for $\kappa=\frac{2n\beta}{n-2}\cdot\frac{p-\alpha}{p}$, where we have also recalled $d\Vol_g=u^{\frac{2n}{n-2}} d\mu$.

As $\alpha>\frac n 2$, we can use Proposition \ref{regularH}, to get that $u$ is uniformly bound on $B(x,r/4)$, then we obtain a bound on the $L^p$ norm of $V$ on $B(x,r/4)$ and if we again apply Proposition \ref{regularH}, we get the result.
\endproof

So we first need a definition of the H\"{o}lder spaces.
\begin{Def}
For $\gamma\in (0,1)$ and $U\subset \overline{M}$, we set
\[C^\gamma(U)\coloneqq \{ f\in C^0(U)\, \colon\,  |f(x)-f(y)|\leq d(x,y)^{\gamma}\, \forall \, x,y\in U\}.\]
We equip $C^\gamma(U)$ with the norm
\[\norm{f}_{C^\gamma(U)}\coloneqq \norm{f}_{L^\infty(U)}+ \sup_{x\neq y} \frac{|f(x)-f(y)|}{d(x,y)^{\gamma}}.\]
\end{Def}
\begin{Lem}
\label{Lem:HolderLemma}
For $0<\alpha<\beta$ and $K\subset \overline{M}$ compact, the following inclusion is compact
\[C^\beta(K)\hookrightarrow C^\alpha(K).\]
\end{Lem}

We also need a Sobolev embedding result
\begin{Thm}[{\cite[Theorem 4.8]{ACM2}}]
\label{Thm:SobolevHolder}
Assume $f\in H^{2,p}(M)$ for some $p>\frac{n}{2}$. Then there is $\gamma=\gamma(n,p)>0$ such that $f\in C^\gamma(M)$.
\end{Thm}

Finally, we define local versions of the H\"{o}lder spaces.
\begin{Def}
Let $U\subset \overline{M}$ be open. We say $f\in C^\alpha_{loc}(U)$ if $f\in C^\alpha(K)$ for any compact $K\subset U$. 
\end{Def}

\subsection{Asymptotic behavior of the Yamabe flow}

\begin{Thm}\label{Prop:Dichotomy}
Let $t_k\to +\infty$ be some sequence, $g_k= u^{\frac{4}{n-2}}(t_{k})g_0\eqqcolon u_k^{\frac{4}{n-2}}g_0$, and let  $S_{g_{k}}$ be the scalar curvature of $g_{k}$.  Then, after potentially passing to a subsequence, there is a  $u_\infty \in H^{2,p}(M)$  solving the Yamabe equation
\[-c_n\Delta_0 u_\infty+S_0 u_\infty=\sigma_\infty u_\infty^{\frac{n+2}{n-2}}\]
and 
one of the mutually exclusive statements hold:
\begin{enumerate}[label=\roman*)]
\item There is some $\alpha\in (0,1)$ such that $u_k$  converges strongly in $H^1$ and in $\cC^\alpha$ to $u_\infty$,
\item $u_k$  converges weakly in $H^1$ to $u_\infty$ and there is a finite set $F\coloneqq \{x_1,\dots,x_L\}\subset \overline{M}$, such that for any compact $K\subset \overline{M}\setminus F$, there is $C(K)>0$ such that $\norm{u_k}_{H^{2,p}(K)}\leq C(K)$ uniformly and there is $\alpha>0$ such that $u_{k}\to u_{\infty}\in C^\alpha_{loc}(\overline{M}\setminus F)$
\end{enumerate} 
\end{Thm}

\begin{Rem}
One could formulate an analogous dichotomy for an arbitrary sequence of conformal factors $u_k\in H^{2,p}(M)$ satisfying $\Vol_{g_k}(M)=1$ and $\norm{S_{g_k}}_{L^p(M)}\leq \Lambda$ uniformly for some $p>\frac{n}{2}$. The proof would hardly be changed.
\end{Rem}

\begin{Rem}
In case i), there is more one can say about regularity of the solution $u_\infty$. In particular, one can give an expansion of the solution near the singular stratum, but the details of this will depend on the singularity. We refer the interested reader to \cite[Section 3]{ACM} for such statements. 
\medskip

Case ii) is referred to as ''concentration'' or  ''bubbling'', and we will have more to say about this in what follows, Section \ref{section: nonconvergent}, and the appendix. Note that $u_\infty$ could be $0$ in this case. 
\end{Rem}

\subsection{Proof of Theorem \ref{Prop:Dichotomy}}

We start by the following convenient lemma:

\begin{Lem}
\label{Lem:WeakConv}
If $t_k\to +\infty$, then there is a subsequence $g_s= u^{\frac{4}{n-2}}(t_{k_s})g_0\eqqcolon u_s^{\frac{4}{n-2}}g_0$ such that 
\begin{itemize}
\item $u_s\rightharpoonup u_\infty$ in $H^1(M)$. 
\item $d\Vol_{g_s}:= u_s^{\frac{2n}{n-2}}d\mu \rightharpoonup\!\!\!\!\!\!{}^*\ \  d\mu_\infty$
\item  $\vert S_{g_s}\vert^{\frac{n}{2}}d\Vol_{g_s}  \rightharpoonup\!\!\!\!\!\!{}^*\ \ d\nu_\infty=\sigma_\infty^{\frac{n}{2}}d\mu_\infty,$
\end{itemize}
where $S_{g_s}$ is the scalar curvature of $g_s$, $u_s \rightharpoonup u_{\infty}$ denotes the weak limit in $H^1(M)$, and $d\Vol_{g_s}\rightharpoonup\!\!\!\!\!\!{}^* \,\,\,d\mu_\infty$ denotes the weak $*$-limits, i.e. 
\[\lim_{s\to \infty} \int_M \psi d\Vol_{g_s}=\int_M \psi d\mu_{\infty}\]
for any $\psi\in C^0(\overline{M}).$
\end{Lem}
\begin{proof}
By \eqref{eq:EvolvingScalar}, we have 
\[S u^{\frac{n+2}{n-2}}=L_0(u)=S_0 u -\frac{4(n-1)}{n-2}\Delta_0 u,\]
which we integrate and deduce (with $c_n=4\frac{n-1}{n-2}$)
\[c_n \norm{\nabla u}_{L^2(M)}^2 \leq \norm{S_0}^{\frac{n}{2}}_{L^{\frac{n}{2}}(M)} \norm{u}^2_{L^{\frac{2n}{n-2}}(M)} + \norm{S}_{L^1(M,g)}.\]
The right hand side is bounded since $\norm{u}_{L^{\frac{2n}{n-2}}(M)}=1$ and Proposition \ref{Prop:Sn/2+delta}. We of course have $\norm{u}^2_{L^2(M)}\leq \norm{u}_{L^{\frac{2n}{n-2}}(M)}=1$, hence 
\[\norm{u}_{H^1(M)}^2=\norm{\nabla u}_{L^2(M)}^2+\norm{u}^2_{L^2(M)}\]
is bounded uniformly in time. By the Banach-Alaoglu theorem (and the  Riesz–Fr\'{e}chet representation theorem to identify $H^1(M)$ with its dual), there is therefore a weakly convergent subsequence $u_s$.

Similarly, for any $\psi\in C^0(\overline{M})$, $\norm{u}_{L^{\frac{2n}{n-2}}(M)}=1$ implies
\[\int_M \psi d\Vol_g\leq \norm{\psi}_{L^\infty},\]
hence $\psi\mapsto \int_M \psi d\Vol_g$ is a uniformly bounded linear functional on $C^0(\overline{M})$. By  the Banach-Alaoglu theorem again, there is a weakly convergent subsequence in the dual space, which we identify as the signed measures.

The last statement is very similar, where one additionally uses Proposition \ref{Prop:Sn/2+delta} to argue $\lim\limits_{t\to \infty} \norm{ S(t)-\sigma(t)}_{L^{\frac{n}{2}}(M,g)}=0$.
\end{proof}

The next fact is about $u_\infty$.
\begin{Lem}$u_\infty \in H^{2,p}(M)$  and it solves the Yamabe equation
\[-c_n\Delta_0 u_\infty+S_0 u_\infty=\sigma_\infty u_\infty^{\frac{n+2}{n-2}}.\]
If $u_\infty\neq 0$ (the $0$-function), there is a positive constant $C$ such that 
$$C^{-1}\le u_\infty\le C.$$
\end{Lem}
\begin{proof} According to Remark \ref{rem:regularity}, it is enough to show that $u_\infty$ is a weak solution of the Yamabe equation, that is that for every $\varphi\in H^1(M)$:
\[c_n\int_M\langle \nabla\varphi,\nabla u_\infty\rangle d\mu+\int_M S_0\varphi u_\infty d\mu=\int_M\varphi\sigma_\infty u_\infty^{\frac{n+2}{n-2}}d\mu.\]
But by weak $H^1$ convergences
\[c_n\int_M\langle \nabla \varphi,\nabla u_\infty\rangle d\mu+\int_M S_0\varphi u_\infty d\mu=\lim_{s\to \infty} c_n\int_M\langle \nabla\varphi,\nabla u_s\rangle d\mu+\int_M S_0\varphi u_s d\mu.\]
By definition of $S_s$, we have 
\[c_n\int_M\langle \nabla \varphi,\nabla u_s\rangle d\mu+\int_M S_0\varphi u_s d\mu=\int_M\varphi S_s u_s^{\frac{n+2}{n-2}}d\mu.\]
The right hand side we may write as
\[\int_M\varphi S_s u_s^{\frac{n+2}{n-2}}d\mu=\int_M(\varphi u_s) S_s u_s^{\frac{4}{n-2}}d\mu.\]
Proposition \ref{Prop:Sn/2+delta} tells us
\[\lim_{s\to \infty} \int_M \left| (S_s-\sigma_\infty) u_s^{\frac{4}{n-2}}\right|^{\frac n2}d\mu=0,\]
 hence with the fact that $\varphi u_s$ is uniformly bounded in $L^{n/(n-2)}(M)$, it is sufficient to check that
\[\lim_{s\to \infty} \int_M\varphi  u_s^{\frac{n+2}{n-2}}d\mu=\int_M\varphi  u_\infty^{\frac{n+2}{n-2}}d\mu.\]
But $u_s\rightharpoonup u_\infty$ in $H^1(M)$ hence by Sobolev embedding \eqref{eq:Sobolev},
$u_s\rightharpoonup u_\infty$ in $L^{\frac{2n}{n-2}}(M)$ and also
$u^{\frac{n+2}{n-2}}_s\rightharpoonup u^{\frac{n+2}{n-2}}_\infty$ in $L^{\frac{2n}{n+2}}(M)$.
Notice that  $\varphi\in L^{\frac{2n}{n-2}}(M)\simeq \left(L^{\frac{2n}{n+2}}(M)\right)^*$ hence 
\[\lim_{s\to \infty} \int_M\varphi  u_s^{\frac{n+2}{n-2}}d\mu=\int_M\varphi  u_\infty^{\frac{n+2}{n-2}}d\mu.\]
\end{proof}

We now give the proof of Theorem \ref{Prop:Dichotomy}. 

\begin{proof}[Proof of Theorem \ref{Prop:Dichotomy}]
We introduce for $x\in \overline{M}$:
$$\cS(x):=\lim_{r\to 0^+} \cS(B(x,r)).$$ We have $S_0\in L^{p}(M,g_0)$ for some  $p>\frac{n}{2}$, and we may appeal to \cite[Proof of Lemma 1.3]{ACM} to get
\[\cS(x):=\lim_{r\to 0^+} \cS(B(x,r))=\lim_{r\to 0^+} Y(B(x,r)),\]
so for any $x\in \overline{M}\colon \ \cS(x)^{\frac{n}{2}}\ge Y_\ell^{\frac{n}{2}}$.
For $x\in \overline{M}$, one of two things can happen. Either ($\nu_{\infty}$
is defined in Lemma \ref{Lem:WeakConv})
\begin{equation}
\nu_{\infty}(\{x\})<\cS(x)^{\frac{n}{2}}
\label{eq:Pointmassbound}
\end{equation}
or
\begin{equation}
\nu_{\infty}(\{x\})\geq \cS(x)^{\frac{n}{2}}.
\label{eq:Pointmassunbound}
\end{equation}
Let 
\[F\coloneqq \{ x\in \overline{M}\, \colon \, \nu_{\infty}(\{x\})\geq \cS(x)^{\frac{n}{2}}\}.\]

\noindent\textbf{Step 1: $F$ is a finite set.} 
Assume there are $x\in \overline{M}$ where \eqref{eq:Pointmassunbound} holds. 
Then
\[\sigma_{\infty}^{\frac{n}{2}} \mu_{\infty}(\{x\})=\nu_{\infty}(\{x\})\geq \cS(x)^{\frac{n}{2}}\geq Y_{\ell}^{\frac{n}{2}},\]
so we obtain
\begin{equation}
\mu_{\infty}(\{x\})\geq \left(\frac{Y_{\ell}}{\sigma_{\infty}}\right)^{\frac{n}{2}}.
\label{eq:muBound}
\end{equation}
By the volume normalization, we have
\[1=\lim_{k\to \infty} \int_M 1\cdot d\Vol_{g_s}=\mu_{\infty}(\overline{M})\geq \sum_{x\in F} \mu_{\infty}(\{x\}),\]
so by the uniform lower bound \eqref{eq:muBound}, $F$ must be finite with a uniform bound
\begin{equation}
L\coloneqq \#F \leq \left(\frac{\sigma_{\infty}}{Y_{\ell}}\right)^{\frac{n}{2}}
\label{eq:L-Bound}
\end{equation}
There can therefore only be finitely many points where \eqref{eq:Pointmassbound} fails. 
\medskip

\noindent\textbf{Step 2: Bounds near $x\notin F$.} 
For any $x\notin F$, for any  $\epsilon\in (0,1)$, there is $r=r_x(\epsilon)>0$ and  such that
\[\nu_{\infty}(B(x,2r))\leq \left((1-\epsilon/2)\cS(B(x,r))\right)^{\frac{n}{2}}.\]
For $s\geq s_0$ for some large $s_0$, we therefore get
\[\left(\int_{B(x,r)} \vert S_{g_s}\vert^{\frac{n}{2}}\, d\Vol_{g_s}\right)^{\frac{2}{n}}\leq (1-\epsilon)\cS(B(x,r)).\]
 By Lemma \ref{Lem:LocaluBound} we get uniform bounds near $x$, meaning some $C>0$ such that $C^{-1}\leq u_s\leq C$ on $B(x,r/8)$. 
 \medskip
 
 \noindent We also note for later use that
\[\nu_{\infty}(\{x\})=\sigma_{\infty}^{\frac{n}{2}} \mu_{\infty}(\{x\})=\sigma_{\infty}^{\frac{n}{2}} \lim_{\rho\to 0}\lim_{s\to \infty} \int_{B(x,\rho)} u_s^{\frac{2n}{n-2}}\, d\mu.$$
And that for $\rho\le r/8$
$$\int_{B(x,\rho)} u_s^{\frac{2n}{n-2}}\, d\mu \le C^{\frac{2n}{n-2}}\mu(B(x,\rho)).\]
By Ahlfors regularity \eqref{eq:Ahlfors-Regular}, we get 
 \[\nu_{\infty}(\{x\})=0.\]
 Summarizingm, we therefore conclude
\[\nu_{\infty}(\{x\})<\cS(x)\implies \nu_{\infty}(\{x\})=0.\]

\noindent\textbf{Step 3: Case i).} 
If \eqref{eq:Pointmassbound} holds for all $x\in \overline{M}$. Then for every $x\in \overline{M}\colon \nu_{\infty}(\{x\})=0$. Let $\epsilon\in (0,1)$ Then for every $x\in \overline{M}$, there is a  $r=r_x(\epsilon)>0$ and  such that
\[\nu_{\infty}(B(x,2r))\leq \left((1-\epsilon/2)\cS(B(x,r))\right)^{\frac{n}{2}}.\]  
As $ \overline{M}$ is compact we can cover $\overline{M}$ with finitely many balls $B_i=B(x_i,r_{x_i}(\epsilon)/8)$. 
Using the above argumentation for each if these balls, 
 we get uniform bounds on all of $M$. The convergence statements follow immediately from Theorem \ref{Thm:SobolevHolder}, and Lemma \ref{Lem:HolderLemma}. 
\medskip

\noindent\textbf{Step 4: Case ii).} 
Away from the finitely many points $x_i\in F$, we are in case $i)$ and the arguments there go through when restricting to a compact set $K\subset \overline{M}\setminus F$.

\end{proof}

\begin{Rem}
Taking a different convergent subsequence, different things could happen in Theorem \ref{Prop:Dichotomy}. The upper bound 
\eqref{eq:L-Bound} on the set $F$ is absolute, however:
\[F=\{x\in \overline{M}\colon  \nu_\infty(\{x\})\ge \cS(x)^{\frac{n}{2}}\}, \quad L= \# F\le \left(\sigma_\infty/Y_\ell\right)^{\frac n2}.\]\end{Rem}

We can extract further information out of the dichotomy and arrive at the following result
\begin{Prop}
\label{Prop:rho0NoConc}
Assume $(M,g_0)$ has an average scalar curvature satisfying 
\begin{equation}
\sigma(0)^{\frac{n}{2}}\leq Y(M,[g_0])^{\frac{n}{2}}+Y_{\ell}(\overline{M},[g_0])^{\frac{n}{2}},
\label{eq:LowAverageScalar}
\end{equation}
and assume $u_\infty$ is the limit of some subsequence $u_k$. Then either 
\begin{enumerate}
\item We are in case i) of Theorem \ref{Prop:Dichotomy}.
\item or we are in case ii) of Theorem \ref{Prop:Dichotomy} with $u_\infty=0$, $L=1$ and $\sigma_\infty\geq Y_{\ell}$.
\end{enumerate}
\end{Prop}
\begin{Rem}
We recall that for a smooth manifold $\overline{M}$, we have $Y_{\ell}=Y(\S^n,g_{round})$, so the above proposition reduces to parts of \cite[Theorem 1.2]{SS} in the smooth setting. 
\end{Rem}
\begin{proof}
The monotonicity \eqref{eq:rhoEvol} of $\sigma(t)$ says $\sigma_{\infty}\leq \sigma(0)$, and equality happens if and only if $S_0$ is already a constant scalar curvature metric. So we may assume  $\sigma_\infty<\sigma(0)$ in \eqref{eq:LowAverageScalar}.
\medskip

By the discussion of $\mu_{\infty}$ in the proof of Theorem \ref{Prop:Dichotomy}, we can write
\[d\mu_\infty=u_\infty^{\frac{2n}{n-2}}d\mu+\sum_{x\in F} m_x \delta_x\]
Note that the volume normalization gives
\begin{equation}
1=\int_{M} u_\infty^{\frac{2n}{n-2}}d\mu+\sum_{x\in F} m_x
\label{eq:TotalMass}
\end{equation}
and that for each $x\in F\colon $
\begin{equation}
\label{eq:MassBound}
\sigma_\infty \cdot m_x^{\frac 2n}\ge \cS(x)\ge Y_\ell.
\end{equation}
What happens next splits into two cases. Either $u_{\infty}=0$ everywhere or not. Assume first that $u_\infty= 0$. By \eqref{eq:TotalMass} and \eqref{eq:MassBound}, we get
\[\sigma_\infty^{\frac{n}{2}} \geq L Y_{\ell}^{\frac{n}{2}}.\]
Combining this with \eqref{eq:LowAverageScalar}, we deduce
\[L Y_{\ell}^{\frac{n}{2}} \leq \sigma_{\infty}^{\frac{n}{2}} < \sigma(0)^{\frac{n}{2}} \leq Y(M,[g_0])^{\frac{n}{2}}+Y_{\ell}^{\frac{n}{2}}\leq 2 Y_{\ell}^{\frac{n}{2}},\]
meaning $L<2$. The option $L=0$ is impossible due to \eqref{eq:TotalMass}, so we conclude $L=1$.
\medskip

Assume next that $u_\infty\neq 0$ (meaning it is not identically the $0$-function). 
By definition of the Yamabe constant\footnote{Writing $N\coloneqq \frac{2n}{n-2}$, the extra factor of $\norm{u}_{L^{N}(M)}^{\frac{2}{n}}$ comes about since $Y=\inf \frac{\int S\, d\Vol_g}{\norm{u}^2_{L^{N}(M)}}$ whereas $\sigma=\frac{\int S\, d\Vol_g}{\norm{u}^{N}_{L^{N}(M)}}$.}  \eqref{eq:YamabeConst}, 
\[\left(\int_{M} u_\infty^{\frac{2n}{n-2}}d\mu\right)^{\frac{2}{n}}\sigma_{\infty} \ge  Y(M,[g_0]),\]
hence
\begin{equation}
\sigma_\infty^{\frac  n2}\ge Y^{\frac n2}(M,[g_0])+\sum_{x\in F} \cS(x)^{\frac{n}{2}}\ge Y^{\frac n2}(M,[g_0])+L Y_\ell^{\frac{n}{2}} .
\label{eq:SFBound}
\end{equation}
This contradicts \eqref{eq:LowAverageScalar} unless $L=0$, and we are in case 1 of Theorem \ref{Prop:Dichotomy}.
\end{proof}

%%%%%%%%%%%%%%%%%%%%%%%%%%%%%%%%%
\subsection{Small initial energy yields convergence of the solution}\label{small energy}
%%%%%%%%%%%%%%%%%%%%%%%%%%%%%%%%%

We end this section with a small energy criterion (\eqref{eq:SmallInitialEnergyDef} below) which will ensure that the flow lands in case i) of Theorem \ref{Prop:Dichotomy}.

\begin{Prop} [{\cite[Proposition 2.3]{ACT}}]
\label{Prop:u+delta}
Assume 
\begin{equation}
\norm{(S_0)_+}_{L^{\frac{n}{2}}(M)}<Y_{\ell}(\overline{M},g_0),
\label{eq:SmallInitialEnergyDef}
\end{equation}
which we refer to as the small initial energy condition. Then any convergent subsequence $u(t_k)$  lands in case i) of Theorem \ref{Prop:Dichotomy}. 
\end{Prop}

\begin{proof}
Most of the work is already done. By  \cite[Lemma 1.3]{ACM} and Proposition \ref{Prop:Sn/2+delta}, we have
\[Y_{\ell}(\overline{M},g_0)=\cS_{\ell}(\overline{M})=\inf_{x\in \overline{M}}\cS(x).\] 
Combining  \eqref{eq:S+n2Bound} and \eqref{eq:SmallInitialEnergyDef} then gives us
\[\left(\int_M S_+^{\frac{n}{2}}\, d\Vol_g\right)^{\frac{2}{n}}\leq \left(\int_M (S_0)_+^{\frac{n}{2}}\, d\mu\right)^{\frac{2}{n}}<
Y_{\ell}(\overline{M},g_0) = \cS_{\ell}(\overline{M})\leq \lim_{r\to 0} \cS(B(x,r))\]
for all $x\in \overline{M}$. Hence there is some $\epsilon\in (0,1)$ such that
\[\left(\int_{B(x,r)} S_+^{\frac{n}{2}}\, d\Vol_g\right)^{\frac{2}{n}}\leq\left(\int_M S_+^{\frac{n}{2}}\, d\Vol_g\right)^{\frac{2}{n}}\leq  (1-\epsilon)\cS(B(x,r))\]
for all $r>0$ small enough and all $x\in \overline{M}$. This implies by \eqref{eq:Pointmassunbound}
that $F= \varnothing$. This places us in case i) of Theorem \ref{Prop:Dichotomy}.
\end{proof}

%%%%%%%%%%%%%%%%%%%%%%%%%%%%%%%%%
\section{Eigenvalue criterion for prevention of concentration}\label{Section: Eigenvalues}
%%%%%%%%%%%%%%%%%%%%%%%%%%%%%%%%%
Theorem \ref{Prop:Dichotomy} leaves open two questions. Can we ensure $L=0$? And does the flow itself converge, and not just subsequences? In this section, we give a partial answer to the second question and an additional answer to the first question.
In the smooth setting, Matthiesen \cite[Theorem 1.2]{Matthiesen} has come up with a criterion to avoid bubbling, and this involves imposing a bound on the first eigenvalue of $\Delta$, the time-varying Laplace operator. We show here that such a bound would work in our non-smooth setting as well, but we lack a criterion on the initial data to ensure this bound is satisfied.\medskip

 We write $\sigma_{\infty}=\lim_{t\to \infty} \sigma$ as in Section \ref{Section:Convergence}. We start by a criterion for ruling out the existence of several convergent subsequences with different limits. Assume $t_k\to \infty$ is some subsequence for which $u(t_k)\to u_{\infty}$.
\begin{Prop}
\label{Prop:Eigenvalue1}
Assume $\frac{\sigma_{\infty}}{n-1}$ is not an eigenvalue of the Laplace operator of $g_{\infty}\coloneqq u_{\infty}^{\frac{4}{n-2}}g_0$. Then $u(t)$ converges weakly to $u_{\infty}$.
\end{Prop}
\begin{proof}
We start by recalling a classical fact in dynamical systems, namely that the set of possible limits $u_\infty$ of subsequences $u_{t_k}$ is a closed and connected set. The reason being that the set of possible limits can be described as
\[\mathcal{L}:=\bigcap_{T>0} \overline{\{ u_t\, ,t\geq T\}}^{\ \text{weak } H^1},\]
i.e. an intersection of compact connected sets. It therefore suffices to show that $\{u_\infty\}$ is a closed and open set in $\mathcal{L}$ .
We know that if $u\in \mathcal{L}$ then for any $p>n/2$, $u\in H^{2,p}(M)$ and $w=u/u_\infty$
 solves the Yamabe equation
\[-4\frac{n-1}{n-2} \Delta_\infty w+\sigma_\infty w=\sigma_\infty w^{\frac{n+2}{n-2}} .\]
The  linearisation of the Yamabe equation at $u_\infty$ is  
\[-\Delta_\infty v=\frac{\sigma_{\infty}}{n-1} v.\]
By assumption, this linear equation has no non-trivial solution, so by the inverse function theorem, $u_\infty$ is an isolated point\footnote{A priori it is isolated for the topology in which we can apply the inverse function theorem, that is $H^{2,p}(M)$. But the regularity estimate for solution of the Yamabe equation, implies that there is some $\epsilon>0$ such that 
$B^{ H^1}(u_\infty,\epsilon)\cap \mathcal{L}=\{u_\infty\}$, where $B^{ H^1}(u_\infty,\epsilon)$ denotes the ball of radius $\epsilon$ in $H^1$-norm centred on $u_\infty$.} in the set of solutions of the Yamabe equation hence $\{u_\infty\}$ is a closed and open set in $\mathcal{L}$ and $\mathcal{L}$ reduces to  $\{u_\infty\}$.
\end{proof}
The above result combines with Theorem \ref{Prop:Dichotomy} to say that if there is no concentration and Proposition \ref{Prop:Eigenvalue1} holds, then the entire flow converges strongly in $H^{2,p}(M)$ (without passing to subsequences). \medskip

We now come to the eigenvalue criterion following Matthiesen \cite{Matthiesen}.

\begin{Prop}[\cite{Matthiesen}]
\label{Prop:Eigenvalue2}
 Assume that there is a constant $\lambda>0$ such that for any $t\ge 0$ the first non-zero eigenvalue of $\Delta_{g_t}$ is bounded from below by $\lambda$.  Let $u_k=u(t_k)\to u_{\infty}$ be a convergent subsequence. Then either
\begin{itemize}
\item $u_{\infty}$ is bounded from above and below (away from $0$) and there is no concentration.
\item $u_\infty=0$ and there is only one concentration point.

\end{itemize}
\end{Prop}
\begin{Rem}
The proposition does not rule out that for different subsequence, different scenario occur.
\end{Rem}
\begin{Rem}
The conclusion is exactly the same as in Proposition \ref{Prop:rho0NoConc}. In that proposition, we impose an initial average scalar curvature bound, whereas Proposition \ref{Prop:Eigenvalue2} requires a lower bound on the first eigenvalue along the flow.
\end{Rem}

\proof Let assume that there is at least one subsequence with a concentration point $x$. Let $0<r<1$. We first show that 
\[\lim_{k\to +\infty} \int_{\overline{M}\setminus B(x,r)} u_k^{\frac{2n}{n-2}}d\mu=0.\]
Let $\varphi\ge 0$ be the cut-off function such that
\[\varphi(z)=\begin{cases}1 & d(x,z)<r^2\\ \frac{\log\left(\frac{d(x,z)}{r}\right)}{\log(r)} & r^2\leq d(x,z)\leq r \\ 0 & d(x,z)>r.\end{cases}\]
This is a Lipschitz function, since it is the composition of the two Lipschitz functions $\varphi(z)=\psi\left(\frac{d(x,z}{r}\right),$ where
\[\psi(t)=\begin{cases}1 & t<r\\ \frac{\log(t)}{\log(r)} & r\leq t\leq 1 \\ 0 & t>1.\end{cases}\]
Then with $\overline{\varphi}_k\coloneqq \fint \varphi d\Vol_{g_k}$, the Poincar\'{e} inequality (i.e. the min-max principle for finding $\lambda$) reads 
\begin{equation}
\lambda \int_{M} \left(\varphi-\overline{\varphi}_k\right)^2d\Vol_{g_k} \leq \int_M\vert \nabla \varphi\vert^2_{g_k}\, d\Vol_{g_k}=  \int_{B(x,r)} \vert \nabla \varphi\vert^2_{g_k}\, d\Vol_{g_k}.
\label{eq:Poincare}
\end{equation}
We break the left integral into 3 regions and estimate by dropping the first two integrals:
\begin{align*}
& \int_{M} \left(\varphi-\overline{\varphi}_k\right)^2d\Vol_{g_k}\\ &= \left(1-\overline{\varphi}_k\right)^2\int_{B(x,r^2)} d\Vol_{g_k}+ \int_{B(x,r)\setminus B(x,r^2)} \left(\varphi-\overline{\varphi}_k\right)^2d\Vol_{g_k}+ \overline{\varphi}_k^2\int_{M\setminus B(x,r)} d\Vol_{g_k}\\
 &\geq \overline{\varphi}_k^2\Vol_{g_k}(\overline{M}\setminus B(x,r)).
\end{align*}
As in the proof of Theorem \ref{Prop:Dichotomy}, we have 
\[\lim_{k\to \infty} \overline{\varphi}_k=\int_M \varphi d\mu_\infty=\int_M \varphi u_{\infty}^{\frac{2n}{n-2}} \, d\mu +\sum_{y\in F} m_y \varphi(y),\]
and by \eqref{eq:MassBound}, we get a lower bound
\[\lim_{k\to \infty} \overline{\varphi}_k\geq m_x \varphi(x)=m_x\geq \left(\frac{Y_\ell}{\sigma_\infty}\right)^{\frac{n}{2}}.\]
For all $k$ large enough, we therefore have
\[ \overline{\varphi}_k\geq \frac{1}{2}\left(\frac{Y_\ell}{\sigma_\infty}\right)^{\frac{n}{2}}.\]
Inserting this into \eqref{eq:Poincare}, we deduce
\begin{equation}
\Vol_{g_k}(\overline{M}\setminus B(x,r)) \leq C  \int_{B(x,r)} \vert \nabla \varphi\vert^2_{g_k}\, d\Vol_{g_k},
\label{eq:MassOutsidex}
\end{equation}
for all $k$ large enough, where 
\[C\coloneqq \frac{1}{4\lambda} \left(\frac{\sigma_\infty}{Y_\ell}\right)^n.\]
We note how $C$ is independent of $k$ and $r$. We will show that the right hand side tends to $0$ as $r\to 0$, meaning no mass concentrates outside of $\{x\}$.

We first estimate the right hand side of \eqref{eq:MassOutsidex} by the H\"{o}lder inequality and that $\varphi$ is constant outside of $B(x,r)\setminus B(x,r^2)$. 
\[\int_{B(x,r)} \vert \nabla \varphi\vert^2_{g_k}\, d\Vol_{g_k}\leq \left( \int_{B(x,r)\setminus B(x,r^2)}\  \vert \nabla \varphi\vert_{g_k}^{n}d\Vol_{g_k}\right)^{\frac{2}{n}}.\]
By conformal invariance\footnote{Recall $\vert \nabla \varphi\vert_{g_k}^{2}=u_k^{-\frac{4}{n-2}} \vert \nabla \varphi\vert_{g_0}^{2}$, and $d\Vol_{g_k}=u_k^{\frac{2n}{n-2}}d\mu$. So $\vert \nabla \varphi\vert_{g_k}^{n}d\Vol_{g_k}=\vert \nabla \varphi\vert_{g_0}^{n} d\mu$.}
\[\int_{B(x,r)\setminus B(x,r^2)}\  \vert \nabla \varphi\vert_{g_k}^{n}d\Vol_{g_k}= \int_{B(x,r)\setminus B(x,r^2)} \vert \nabla \varphi\vert_{g_0}^{n}  \, d\mu, \] 
and we compute
\[ \vert \nabla \varphi\vert_{g_0}^{ n } = \frac{1}{d(x,z)^n \log(1/r)^{n}}.\]
Hence we need to compute
\[\frac{1}{\log(1/r)^n}\int_{B(x,r)\setminus B(x,r^2)} \frac{d\mu(z)}{d(x,z)^n}.\]
We estimate this integral using the Cavalieri principle and Ahlfors regularity. Write $\rho(z)=d(x,z)$. For any measurable $A\subset X$, and $f\in C^0((0,\infty))$, we have
\[\int_A f\circ \rho(z)\, d\mu(z)=\int_0^\infty f(t)dV(t),\]
where $dV(t)$ is the Stieltjes measure associated to the function $V(t)= \mu(A\cap B(x,t))$. Since in our case $f(t)=t^{-n} \in C^1((0,\infty))$ and tends to $0$ as $t\to \infty$, we may integrate by parts and write
\[\int_0^\infty f(t)dV(t)=-\int_0^\infty  V(t)f'(t)\, dt +\lim_{b\to \infty} f(b)V(b)- \lim_{a\to 0} f(a)V(a).\]
The two boundary terms drop out since $V(a)=0$ for all $a<r^2$ and $f(b)V(b)=V(r) f(b)\xrightarrow{b\to \infty} 0$ for $b>r$, hence
\begin{equation}
\int_A f\circ \rho(z)\, d\mu(z)=\int_0^\infty f(t)dV(t)=-\int_0^\infty  V(t)f'(t)\, dt
\label{eq:Cavalieri}
\end{equation}  
In our case, $f(t)=t^{-n}$ and $A=B(x,r)\setminus B(x,r^2)$, so
\[\frac{1}{\log(1/r)^n}\int_{B(x,r)\setminus B(x,r^2)} \frac{d\mu(z)}{d(x,z)^n}=\frac{1}{\log(1/r)^n}\int_{r^2}^r \frac{n \mu(B(x,t)}{t^{n+1}} )\, dt.\] 
 The volume $\mu(B(x,t))$ is bounded by $Ct^n$ (by the Ahlfors regularity \eqref{eq:Ahlfors-Regular}) for some uniform constant $C>0$, so we may estimate the integral \eqref{eq:Cavalieri} as
\[ \int_{B(x,r)\setminus B(x,r^2)} \vert \nabla \varphi\vert_{g_0}^{n}  \, d\mu \leq \frac{C}{\log(1/r)^{n}} \int_{r^2}^r \frac{n t^n}{t^{n+1}}\, dt = \frac{nC}{\log(1/r)^{n-1}}.\]
 Inserting this into \eqref{eq:MassOutsidex}, we get
\[\Vol_{g_k}(\overline{M}\setminus B(x,r^2)) \leq  \frac{(nC)^{\frac{2}{n}}}{\log(1/r)^{\frac{2(n-1)}{n}}}.\]
Sending $k\to \infty$ and $r\to 0$, we deduce
 $\mu_\infty(\overline{M}\setminus \{x\})=0$.
\endproof
With a stronger assumptions, we can rule out concentration and ensure convergence. This is a consequence of \cite[Proof of Proposition 2.14]{SS}.
\begin{Prop}[\cite{SS}]
If for some $\Lambda> \frac{\sigma_\infty}{n-1}$, the first non-zero of  the eigenvalue of $\Delta=\Delta_{g_t}$ is bounded from below by  $\Lambda$, then there is no concentration along the flow and we get convergence to a positive function $u_{\infty}$.
\end{Prop}
\begin{proof}
Let 
\[F_2(t)\coloneqq \int_M |S(t)-\sigma(t)|^2 d\Vol_{g(t)}=\norm{S(t)-\sigma(t)}^2_{L^2(M,g(t))}\]
and
 \[G_2(t)\coloneqq \int_M \vert \nabla S(t)\vert^2_{g(t)} \, d\Vol_{g(t)}=\norm{\nabla S(t)}^2_{L^2(M,g(t))}.\]
Then by \cite[Lemma 4.4]{SS}, there is a function $\epsilon(t)$ with $\lim\limits_{t\to\infty}  \epsilon(t)=0$  such that
\[\frac{d}{dt} F_2(t) \le -2c_n G_2(t)+\frac{8}{n-2}\left(\sigma(t)+\epsilon(t)\right)F_2(t). \]
The assumed eigenvalue bound gives us
\[-G_2(t)\leq -\Lambda F_2(t) = -\left(\frac{\sigma_\infty}{n-1} +\frac{\delta}{n-1}\right)F_2(t),\]
for some fixed $\delta>0$. Writing $\sigma(t)=\sigma_\infty + \tilde{\epsilon}(t)$ where $\tilde{\epsilon}(t)\to 0$ and absorbing this $\tilde{\epsilon}$ in the definition of $\epsilon$, we get
\[\frac{d}{dt} F_2(t)\le \frac{8}{n-2}\left(-\delta +\epsilon(t)\right)F_2(t).\]
For all $t$ large enough, we have $\epsilon(t)<\frac{\delta}{2}$, so with $\mu\coloneqq \frac{4\delta}{n-2}$ we get
\[\frac{d}{dt} F_2(t)\le -\mu F_2(t),\] 
hence $F_2(t)$ is converging exponentially fast to $0$. Then the argument of \cite[Proof of theorem 1.2]{SS} implies that there is non concentration and by Proposition \ref{Prop:Eigenvalue1}, we get convergence.
\end{proof}

\subsection{The role of the positive mass theorem}
As announced at the start of the section, we do not have conditions on $(M,g_0)$ which ensure the eigenvalue bounds in this section are satisfied. In the smooth case, this is where the positive mass theorem enters. Indeed, \cite[Equation 57]{SS} is essentially an eigenvalue bound, and the validity of this uses the local version of the positive mass theorem \cite[Equation 61]{SS}. We cannot imitate these arguments, since we lack conformal normal coordinates at all points in $\overline{M}$ and an expansion of the Green's function akin to \cite[Equation 61]{SS}.

%%%%%%%%%%%%%%%%%%%%%%%%%%%%%%%%%
\section{A non-convergent example $-$ the Eguchi-Hanson space}\label{section: nonconvergent}
%%%%%%%%%%%%%%%%%%%%%%%%%%%%%%%%%
If $(M^4,g_0)$ is an ALE (asymptotically locally Euclidean) gravitational instanton, the conformal compactification of $(M^4,g_0)$ is a smooth orbifold $(\overline{M},g_\psi)$ with one singular point $\infty$ modelled on $\C^2/\Gamma$ where $\Gamma$ is a finite subgroup of $\mathrm{SU}(2)$.  Viaclovsky \cite{Orbifold}  has shown that there is no Lipschitz conformal deformation of  $(\overline{M},g_\psi)$ of constant scalar curvature. In this case
$Y(\overline{M},g_\psi)=Y(\S^4,g_{round})/\sqrt{\#\Gamma}.$
Hence in this case, we know by the dichotomy of Theorem \ref{Prop:Dichotomy} that any Yamabe flow $u^{\frac{4}{n-2}}(t)g_\psi$ on $(\overline{M},g_\psi)$ converges weakly to $0$ in $H^{1}(M)$ and develops  a spherical bubble at the singular point $\infty$ as $t\to \infty$. 
We will study the simplest of these in great detail, namely (a conformal compactification of) the Eguchi Hanson space, with $\Gamma=\mu_2\coloneqq \{\pm 1\}$.\medskip

We take as smooth manifold $M=T^*\C\P^1$, the cotangent bundle of $\C\P^1$, and think of this as the blow-up (in the algebraic-geometric sense) of $\C^2/\{\pm 1\}$, where the action by $-1$ is $\mathbf{z}\mapsto -\mathbf{z}$. Removing the zero section $\C\P^1$ we get a manifold biholomorphic to $\left(\C^2\setminus \{0\} \right)/\{\pm 1\}$, and we will perform most of our analysis on the double cover $\C^2\setminus \{0\}$. We equip $M$ with the Eguchi-Hanson metric, which is a Ricci-flat K\"{a}hler metric introduced in \cite[Equation 2.33a]{EH}. In complex coordinates $(z^1,z^2)\in \C^2\setminus \{0\}$, the metric (thought of as a $2\times 2$ hermitian matrix) reads\footnote{For two vectors $u,v$ the notation $u\otimes v$ denotes the matrix $A$ with entries $A_{ij}=u_i v_j$.}
\begin{equation}
g_{EH}=\frac{\sqrt{a^4+r^4}}{r^2}\left(    \mathds{1} -\frac{a^4}{a^4+r^4} \frac{\overline{z}\otimes z}{r^2}\right),
\label{eq:EHMetric}
\end{equation}
where $r$ is the Euclidean distance to the origin and $a$ is a fixed real number. The significance of $a$ is that 
\begin{equation}
(g_{EH})_{\vert \C\P^1}=a^2 g_{FS},
\label{eq:EHonCP1}
\end{equation}
 where $g_{FS}$ denotes the Fubini-Study metric on $\C\P^1$, and the way to see this is as follow. Introducing the notation $\overline{z}\cdot dz\coloneqq \overline{z}^1 dz^1+\overline{z}^2dz^2$ and $\vert dz\vert^2\coloneqq dz^1d\overline{z}^1+ dz^2d\overline{z}^2$,  the Eguchi-Hanson line element reads
\[ds^2_{EH}=\sqrt{a^4+r^4}\left(\frac{\vert dz\vert^2}{r^2}-\frac{\vert \overline{z}\cdot dz\vert^2}{r^4}\right)+\frac{\vert \overline{z}\cdot dz\vert^2}{\sqrt{a^4+r^4}}. \]
The term in the brackets we recognize as the Fubini-Study metric on $\C\P^1$ written in homogeneous coordinates. The last term goes to $0$ as $r\to 0$ , and this establishes \eqref{eq:EHonCP1}.  One readily checks that $\det(g_{EH})=1$, which implies Ricci-flatness.\medskip

As singular manifold $\overline{M}$, we take the one-point compactification of $T^*\C\P^1$. We remark that this compactified space can be identified as the weighted projective space $\C\P^2_{1,1,2}$ where we use the notation of \cite{Weighted}, but we will not need this explicit identification. We now conformally change the Eguchi-Hanson metric by a conformal factor $\psi^2$ where $\psi$ has the properties that $\psi_{\vert \C\P^1}=1$ and $\psi\in \mathcal{O}(r^{-2})$ as $r\to \infty$. This is a conformal compactification $(\overline{M},\psi^2 g)$ with a singular (an orbifold singularity) point at infinity. 
\medskip

We can arrange for $g_{\psi}\coloneqq \psi^2 g$ to satisfy 
\begin{itemize} 
\item $(\overline{M},g_{\psi})$ is compact.
\item scalar curvature $S_\psi$ of $g_{\psi}$ satisfies $S_\psi \in L^{\infty}(M)$, $S_\psi > 0$ on $M$.

\end{itemize}
and we will demonstrate this below by a particular choice of $\psi$.
\medskip

Observe that the metric $g_{\psi}$ is $U(2)$-symmetric, so the family of metrics evolving according to Yamabe-flow will (by uniqueness of solutions) be $U(2)$-symmetric as well. This forces the conformal factor $u$ to be rotationally symmetric at all time, and we may write $u(t,x)=v(t,r)$ for some $v\colon [0,\infty)\times (0,\infty)\to (0,\infty)$ for the conformal factor restricted to $T^*\C\P^1\setminus \C\P^1$.

For a K\"{a}hler metric $h$, the Laplacian is given by (see \cite[equation 7.27]{Kahler}) $\Delta=4h^{\overline{\nu}\mu} \partial_{\mu}\partial_{\overline{\nu}}=4\tr(h^{-1} \nabla^2)$, where we think of $h$ as a hermitian matrix.\footnote{The extra factor of 4 ensures that this agrees with the real Laplacian.}
The inverse metric of the Eguchi-Hanson metric \eqref{eq:EHMetric} can be written\footnote{A computational trick for checking this is that $\frac{\overline{z}\otimes z}{r^2}\cdot \frac{\overline{z}\otimes z}{r^2}=\frac{\overline{z}\otimes z}{r^2}$.}
\begin{equation}
g^{-1}_{EH}=\frac{r^2}{\sqrt{r^4+a^4}}\left(\mathds{1}+\frac{a^4}{r^4}\frac{\overline{z}\otimes z}{r^2}\right),
\label{eq:EH-Invers}
\end{equation}
and the Laplacian of $g$ acting on radially symmetric function $u(t,x)=v(t,\vert x\vert)$ reads\footnote{Note that the derivatives are with respect to $r^2$, as this turns out to be a better coordinate than $r$ when computing with the Eguchi-Hanson metric.}
\begin{equation}
\Delta_{EH} u=\frac{4r^2}{\sqrt{r^4+a^4}}\left( \left(2+\frac{a^4}{r^4}\right)\partial_{r ^2} v + \left(1+\frac{a^4}{r^4}\right) r^2 \partial^2_{r^2} v\right).
\label{eq:Delta0EH}
\end{equation}
The scalar curvature of $g_{\psi}$ reads (using the expressions given in \eqref{eq:YamabeFlow} for the conformal Laplacian). 
\begin{equation}
S_{\psi}=-6\psi^{-3} \Delta_{EH} \psi,
\label{eq:Spsi}
\end{equation}
so the condition $S_{\psi}>0$ is equivalent to $-\Delta_{EH} \psi>0$. \medskip

 Notice that the volume element of $g_{\psi}$ reads $d\Vol_{g_{\psi}}=\psi^4 d\mu$ where $d\mu=d\Vol_{Euc}$ since $\det(g_{EH})=1$.  We also record that \cite{Orbifold} has computed the Yamabe constant in this case, and the result is $Y(\overline{M},[g_{EH}])=Y(\S^4)/\sqrt{2}=8\sqrt{3}\pi$, and this equals the local Yamabe constant, $Y_{\ell}(\overline{M},g_{EH})= Y(\overline{M},[g_{EH}])$. We stress that this is independent of the choice of conformal compactification $\psi$. \medskip
 
We now make a particular choice, namely 
\[\psi(r)^2=\frac{a^4}{a^4+r^4}.\]
This choice satisfies $\lim\limits_{r\to \infty} \psi r^2=a^2<\infty$, so this is an allowed conformal factor.
Using this, one easily checks\footnote{A computational remark: We get a factor of $\Vol(\R\P^3)=\frac{1}{2}\Vol(\S^3)=\pi^2$ since we are looking at $\C^2/\{\pm 1\}$ and not $\C^2$.} 
\[\Vol(M,g_{\psi})=\int_{0}^\infty \int_{\R\P^3} \psi(r)^4 r^3 d\Omega_{\R\P^3}\, dr= \frac{\pi^2 a^4}{4}.\]
Furthermore, the radial lines $\gamma(t)=t z$ for $z\in \S^3\subset \C^2$ are geodesics connecting $0$ and $\infty$, so we can compute the distance as
\[d(0,\infty)=\int_0^\infty \sqrt{g_{\psi}(\dot{\gamma},\dot{\gamma})}\, dt=\frac{\vert a\vert }{2}\int_0^\infty \frac{dt}{(1+t^2)^{\frac{3}{4}}}<\infty,\]
so the resulting space is compact. Furthermore, using \eqref{eq:Delta0EH} 
we see
\[\Delta_{EH} \psi=\frac{4r^2}{\sqrt{r^4+a^4}}\left(-\frac{a^2(2r^4 +a^4)}{r^2(r^4+a^4)^{\frac{3}{2}}} + \frac{a^2(2r^4-a^4)}{r^2(a^4+r^4)^{\frac{3}{2}}}\right)=-\frac{8a^6}{(a^4+r^4)^2},\]
and with \eqref{eq:Spsi}, we arrive at
\begin{equation}
S_{\psi}=\frac{48}{\sqrt{a^4+r^4}}.
\label{eq:SpecificSpsi}
\end{equation}
This shows $S_{\psi}>0$ everywhere on $M$ and $S_{\psi}\in L^{\infty}(M)$. One can also compute
\[\norm{S_{\psi}}_{L^2(M,g_{\psi})}^2 =288\pi^2,\]
and we note that $\norm{S_{\psi}}_{L^2(M,g_{\psi})}=12\sqrt{2}\pi>8\sqrt{3}\pi =Y_{\ell}(\overline{M},[g_{EH}])$, so the small-energy condition of Proposition \ref{Prop:u+delta} is violated, and we do not get uniform bounds on the solution $u$ (as already predicted by Viaclovsky's result). The condition of Proposition \ref{Prop:rho0NoConc} is violated as well, since
$\sigma(0)^2=\pi^{10}> 2\cdot (8\sqrt{3}\pi)^2=Y(\overline{M},g_{EH})^2+Y_{\ell}(\overline{M},g_{EH})^2$.
\medskip

By the dichotomy of Section \ref{Section:Dichotomy}, we have the formation of bubbles. The energy  $\norm{S_{\psi}}_{L^2(M,g_{\psi})}^2 =288\pi^2$ is smaller than the corresponding energy for $\S^4$, $\norm{S_{round}}^2_{L^2(\S^4,g_{round})}=384\pi^2$, so there can be no bubbles forming in smooth points due to Remark \ref{Rem:SchoenYau}, that
\[\lim_{R\to 0} Y(B_R(p))=Y(\S^n,[g_{round}])\]
holds for any smooth point $p$. This means the bubble has to form at the point at infinity. We will have a closer look at this next.

We can write down the Yamabe flow explicitly in this case.  Let us first compute $\Delta_{\psi}$, the Laplacian associated to $\psi^2 g$. This is no longer a K\"{a}hler metric, so the above formula no longer applies. We therefore use the more general formula
\[\Delta_{\psi} f=\frac{1}{\sqrt{\det(g_{\psi})}} \partial_k \left( g_{\psi}^{kl} \sqrt{\det(g_{\psi})} \partial_l f\right)=\psi^{-2} \Delta_{EH} f +2\psi^{-3} \ip{\nabla \psi}{\nabla f}_{g_{EH}},\]
where we have inserted $\det(g_{\psi})=\psi^8$. When $f=f(r^2)$, one easily uses \eqref{eq:EH-Invers} to check
\[\ip{\nabla \psi}{\nabla f}_{g_{EH}}=4\sqrt{r^4+a^4} (\partial_{r^2} f)(\partial_{r^2}\psi).\]
and thus
\[\Delta_{\psi} f=4\left(\frac{\sqrt{r^4+a^4}}{r^2} \partial_{r^2} f +\frac{(r^4+a^4)^{\frac{3}{2}}}{a^4} \partial_{r^2}^2 f\right).\]
Introducing the scale-less coordinate $x\coloneqq \psi =\frac{a^2}{\sqrt{a^4+r^4}}$ and  changing $t\mapsto \frac{a^2}{36} t$ to get rid of $a$, we can write the Yamabe flow \eqref{eq:YamabeFlow} as the following PDE:

\begin{align}
 \left( \partial_t v^3 - \tilde{\sigma}(t) v^3\right)&=\partial_x\Big( \partial_x \Big(x(1-x^2) v\Big)+ 2x^2v\Big) \notag \\ 
 &=x(1-x^2)\partial_x^2 v +2(1-2x^2)\partial_x v -2xv \notag  \\
 &=\partial_x \left(\left(1-x^2\right)\partial_x \left(xv\right)\right). 
\label{eq:NormalizedSpecialYamabe2}
\end{align}
where $\tilde{\sigma}(t)=\frac{1}{12\pi}\sigma(t)$, meaning in particular $\tilde{\sigma}(0)=\frac{\pi^4}{12}$ and a lower bound $\tilde{\sigma}(t)\geq \frac{a^2}{24}\frac{Y_{\ell}}{\sqrt{\Vol(M,g_{\psi})}}=\frac{2}{\sqrt{3}}$.
Note that $x=0$ corresponds to $r=\infty$ in these coordinates, and the interval $0\leq r\leq \infty$ has been mapped to $0\leq x\leq 1$.
\medskip

We will derive some uniform bound on $v(t,x)$ for $x>0$. 
We observe that the right hand side of \eqref{eq:NormalizedSpecialYamabe2} is $-S_\psi \cdot v^3$. Since the initial scalar curvature \eqref{eq:SpecificSpsi} is non-negative, Proposition \ref{Prop:SEvol} tells us that $S_\psi \geq 0$ for all time. Hence
\[ \partial_x \left(\left(1-x^2\right)\partial_x \left(xv\right)\right)\leq 0.\]
Integrating this from $x$ to $1$ yields
\[-(1-x^2)\partial_x(xv)\leq 0,\]
or $\partial_x(xv)\geq 0$ and thus
\[v(t,1)\geq xv(t,x)\]
holds for all time. The left hand side we bound as follows. From Proposition \ref{Prop:SEvol}, we have
\[\norm{S_\psi}_{L^2(M,g)}^2 =\Vol(\R\P^3)\int_0^1 S_\psi(t,x)^2 v(t,x)^4 x\, dx\leq \Vol(\R\P^3)\Lambda \coloneqq \norm{S_0}_{L^2(M,g_0)}^2.\]
Introduce the Green kernel 
\[G(x)\coloneqq \frac{1}{x} \log\left(\frac{1+x}{1-x}\right).\]
Multiplying \eqref{eq:NormalizedSpecialYamabe2} by $G$ and integrating with the measure $x dx$ and integrating yields
\[2v(t,1)=\int_0^1 G(x)S_\psi(t,x)v(t,x)^3 x\, dx.\]
By the H\"{o}lder inequality, we get
\[\int_0^1G(x)S_\psi(t,x)v(t,x)^3 x\, dx\leq \left(\int_0^1 G(x)^4 x\, dx\right)^{1/4} \sqrt{\Lambda} \left(\int_0^1 v(t,x)^4 \, x dx\right)^{1/4}=C<\infty,\]
where we have used the volume normalization
\[\int_0^1 v(t,x)^4 \, x dx=2\]
and the fact that
\[\int_0^1 G(x)^4 x\, dx<\infty.\]
This gives a uniform bound
\[v(t,x)\leq \frac{C}{x}\]
for all time. This shows that the solution can only blow up at $x=0$. 
\medskip

One can of course try to numerically solve \eqref{eq:NormalizedSpecialYamabe2} directly. Figure \ref{fig:YamabePlot} shows the short-time evolution, and is gotten by solving the equation with an explicit time scheme. One sees the mass starting to accumulate near $x=0$.
\begin{figure}
\includegraphics[scale=0.6]{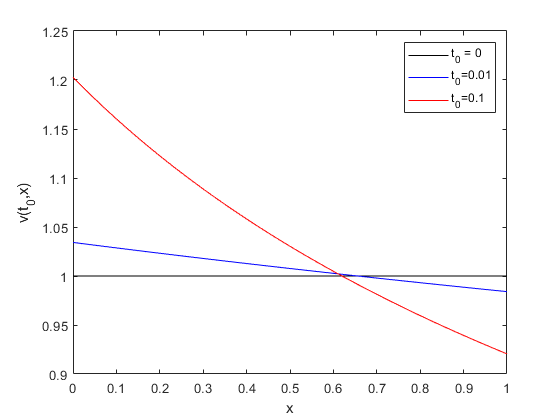}
\caption{The short-time evolution of the Yamabe flow on compactified Eguchi-Hanson space with the coordinate $x=\psi(r)=\frac{a^2}{\sqrt{a^4+r^4}}$. Numerical solution.}
\label{fig:YamabePlot}
\end{figure}

\begin{Rem}
The global version of the positive mass theorem, saying that the mass is positive for all asymptotically Euclidean spaces other than $\R^n$, famously fails if one relaxes the assumptions to allow locally asymptotically Euclidean spaces. The Eguchi-Hanson space  is the easiest counter example. In light of the role  the positive mass theorem  plays in the smooth Yamabe problem, it  is curious that the compactification of Eguchi-Hanson space  is an example of a singular space where the Yamabe problem does not have a solution. We suspect there is a connection here, but we leave the precise formulation of it as future work.
\end{Rem}

\begin{Rem}
In \cite{ACM}, the Yamabe problem was shown to always have a solution under the assumption $Y(M,g_0)<Y_{\ell}(\overline{M},g_0)$. This condition fails for the above example, where instead $Y(M,g_{EH})=Y_{\ell}(\overline{M},g_{EH})$. We do not know if this assumption alone is enough to guarantee that the Yamabe flow converges. If not, it would mean that we have spaces where the Yamabe problem has solutions which are not found by the Yamabe flow.
\end{Rem}

\appendix
\section{Description of the bubbles}
\label{App:Bubbles}

We describe here the bubbles decomposition of Palais-Smale sequences\footnote{for equations of Yamabe type} on smoothly stratified Riemannian pseudomanifolds. We will generalize the description of Struwe \cite{Struwebubble}. In our situation, it will give more informations on the blow-up behaviour along the Yamabe flow on this space. We believe that this decomposition could be useful for other types of non-linear equation on   smoothly stratified Riemannian pseudomanifold.

In this appendix, the singular stratum $\overline{M}\setminus M$ plays a bigger role, and we need to describe the metric $g_0$ in more detail. To conform with standard notation, we therefore change notation. The link to the main text is $\overline{M}=X$, and $M=X_{reg}=X\setminus X_{n-2}$. 

\subsection*{Some words on the geometry of such a space}

For more detailed information see \cite[Subsection 2.1]{ACM} or \cite[Section 3]{ALMP}. 
Let $X$ be a compact stratified space of dimension $n$ with empty boundary, which means  $X$ admits a stratification 
\[X_0\subset X_1\dots\subset X_{n-2}\subset X,\]
and for each $k$, $X_k\setminus X_{k-1}$ is a smooth manifold of dimension $k$. 
We endow $X$ with an iterated edge metric. That is $X_{reg}=X\setminus X_{n-2}$ is endowed with a smooth Riemannian metric $g_0$ that has the following behaviour nearby the singular strata. Let $p\in X_{k}\setminus X_{k-1}$, there is $Z$  is a compact stratified space of dimension $n-k-1$ and a homeomorphism
\begin{equation}
h\colon \bB^k\times C_{[0,\eta)}(Z)\rightarrow U\subset X
\label{eq:Homeom}
\end{equation} 
where\footnote{We write $\bB^k\coloneqq \{x\in \R^k \, \colon\, \vert x\vert <1\}$ and for $\Lambda>0$ we write $\bB^k(\Lambda)\coloneqq \{x\in \R^k \, \colon\, \vert x\vert <\Lambda\}$.}
writing $o$ for the tip of the cone over $Z$
\begin{enumerate}[label=\roman*)]
\item $C_{[0,\eta)}(Z)=\left([0,\eta)\times Z\right)/\left(\{0\}\times Z\right)$ is the cone over $Z$, 
\item $h\left(\bB^k\times \left( C_{[0,\eta)}(Z_{reg})\backslash \{o\} \right)\right)=U\cap X_{reg}$,
\item $h(0,o)=p$,where $o$ is the tip of the cone over $Z$,
\item  If $z_1,\dots,z_{n-k-1}$ are local coordinates on $Z_{reg}$
\begin{align*}
h^*g_0 &= dr^2 + \sum_{i, j = 1}^k h_{i, j} (y) dy^i dy^j   \\
&+ r^2 \sum_{i=1}^\ell\sum_{\alpha=1}^{n-k-1} b_{i,\alpha}(r,y,z) dy^i dz^\alpha \\ 
&+ r^2 \sum_{\alpha,\beta = 1}^{n-k-1} k_{\alpha,\beta}(r,y,z) dz^\alpha dz^\beta,
\end{align*}
where for each $(y,r)\in \bB^k\times [0,\eta)$, the bilinear form $k_{\alpha,\beta}(r,y,z) dz^\alpha dz^\beta$ extended to an  iterated edge metric on $Z$ with smooth dependence on $(y,r)$.
\end{enumerate}

We endow $X$ with a distance $d$ such that $(X,d)$ is the metric completion of $(X_{reg}, dist_{g_0})$ and a Radon measure $\mu$ induced by the Riemannian volume element $d\mu\coloneqq d\Vol_{g_0}$ for which $X_{reg}$ has full measure.\\

\subsection*{Tangent spaces and fake tangent spaces }
In order to analyse the blow-up behaviour of approximate solution of Yamabe type equations, we need to understand blow-up limits of our space
\begin{Def} A pointed metric space $(\underline{X},\underline{d},\underline{x})$ is a fake tangent space at $x\in X$ if it is the pointed Gromov-Hausdorff limit of a sequence $(X,\epsilon_i^{-1}d,x_i)$ where $\lim_i \epsilon_i=0$ and  $\lim_i x_i=x$.
\end{Def}
It means that we can find a sequence $\delta_i\to 0$ and maps
$h_i\colon B(\underline{x},\delta_i^{-1})\subset \underline{X}\rightarrow B(x_i,\epsilon_i\delta_i^{-1})$ such that
 \begin{itemize}
 \item $h_i(\underline{x})=x_i$,
\item for each $p\in B(x_i,\epsilon_i\delta_i^{-1})$ there is some $q\in  B(\underline{x},\delta_i^{-1})$ such that 
$$d(h_i(q),p)\le \delta_i\epsilon_i$$
\item $\forall p_0, p_1\in B(\underline{x},\delta_i^{-1})\colon \left|\underline{d}(p_0,p_1)-\epsilon_i^{-1}d(h_i(p_0),h_i(p_1)\,)\right|< \delta_i$.
\end{itemize}
We can then assume that there is also a  map $f_i\colon  B(x_i,\epsilon_i\delta_i^{-1})\rightarrow B(\underline{x},\delta_i^{-1})$ that satisfies
$\forall p\in  B(x_i,\epsilon_i\delta_i^{-1})\colon d(h_i(f_i(p)),p)\le \delta_i\epsilon_i$ and $\forall q\in  B(\underline{x},\delta_i^{-1})\colon \underline{d}(f_i(h_i(q)),q)\le \delta_i$.
Up to extraction of subsequence, one of the following cases occurs
\subsubsection*{First case}
If $x\in X_{reg}$ then the only fake tangent space at $x$ is the Euclidean space. 
Note that the same conclusion holds if $x\in X_{sing}:=X\setminus X_{reg}$ and $\lim_{i} \frac{d(x_i,X_{sing})}{\epsilon_i}=+\infty$.

\subsubsection*{Second case } $x\in X_k\setminus X_{k-1}$ and $\lim_{i} \frac{d(x,x_i)}{\epsilon_i}=\rho<\infty$.
Using the homeomorphism $h$, \eqref{eq:Homeom}),  and for any $\Lambda>0$, and for $i$ large enough, we  define
$\tilde h_i(\xi,s,z)=h(\epsilon_i \xi ,\epsilon_i s,z)$ on  $\bB^k(\Lambda) \times C_{[0,\Lambda)}(Z)$. Up to extraction of subsequence, we can assume that $\lim_i \tilde h_i^{-1}(x_i)=(0,\rho,z)$, so the fake tangent space is the tangent space at $x$, meaning 
it is $\R^k\times C_{[0,+\infty)}(Z)$ endowed with the metric $g^{Euc}_{\R^k}+dr^2+r^2k(0,0)$ but pointed at $(0,r,z)$.
\subsubsection*{Third case } $x\in X_k\setminus X_{k-1}$ and $\lim_{i} \frac{d(x,x_i)}{\epsilon_i}=\infty$. Let $r_i=d(x,x_i)$ and assume that $h(y_i,r_i,z_i)=x_i$ and we consider the map
\[\hat h_i(\xi,s,z)=h(y_i+\epsilon_i \xi ,r_i+\epsilon_i s,z)\]
then up the lower order terms, we have
\begin{align*}\epsilon_i^{-2} \hat h_i^*g\simeq_{\epsilon_i\to 0}  ds^2 &+ \sum_{\ell, j = 1}^k h_{\ell, j} (y_i+\epsilon_i \xi,y_i+\epsilon_i \xi) d\xi^\ell d\xi^j \\
&+  \frac{r_i^2}{\epsilon_i} \sum_{j=1}^\ell\sum_{\alpha=1}^{n-k-1} b_{j,\alpha}(r_i+\epsilon_i s,y_i+\epsilon_i \xi ,z) d\xi^j dz^\alpha \\
 &+  \frac{r_i^2}{\epsilon_i^2} \sum_{\alpha,\beta = 1}^{n-k-1} k_{\alpha,\beta}(r_i+\epsilon_i s,y_i+\epsilon_i \xi, z) dz^\alpha dz^\beta.
 \end{align*}
If one performs a further  rescaling by $\frac{r_i}{\epsilon_i}$ of $z_i\in Z$ one sees that the fake tangent space is a product
$\R^{k+1}\times\underline{Z}$ where $\underline{Z}$ is a fake tangent space of $(Z, k(0,0))$ at $z$ where $z=\lim_i z_i.$ \medskip

\noindent
We summarize our findings as follows.
\begin{Prop}
\label{Prop:TangentCones}
 Assume $X$ is a compact stratified space equipped with an iterated edge metric. Then fake tangent spaces at $x\in X$ are iterated tangent spaces at $x$. More precisely, a tangent space at $x\in X_k\setminus X_{k-1}$ is of the form
\[\underline{X}=\R^{k+\ell}\times C_{[0,\infty)}(\Sigma),\]
  where $\Sigma$ is a compact stratified space of dimension $n-k-\ell-1$ endowed with an   iterated edge metric and the base point can be any point of $\underline{X}$.\end{Prop}
\subsection*{Convergence of functions} We also need some notions of convergence of function along a sequence of metric spaces that converges in the Gromov-Hausdorff sense. A standard reference is the paper of Kuwae and Shioya \cite{KS}.

We consider a fake tangent space at $x$, $(\underline{X},\underline{d}, \underline{x})$. We know that $\underline{X}=\R^{k+\ell}\times C_{[0,\infty)}(\Sigma)=C_{[0,\infty)}(\underline{S})$ is a smoothly stratified Riemannian pseudomanifold that is the cone over the stratified space whose regular part is 
$(0,\pi/2)\times \S^{k+\ell-1}\times Z$ endowed with the Riemannian metric
\[d\theta^2+\cos^2(\theta)g_{\S^{k+\ell-1}}+\sin^2(\theta) k_Z.\]
It is the pointed Gromov-Hausdorff limit of the rescaled spaces $(X,\epsilon_i^{-1}d ,x_i)$ and 
for each $\Lambda>0$ and for sufficient large $i$, we have  a map
$H_i\colon B(x_i,\Lambda \epsilon_i)\rightarrow C_{[0,\Lambda)}(\underline{S})$ which satisfies 
 \begin{itemize}
 \item $H_i(x_i)=\underline{x}$,
\item for each $p\in C_{[0,\Lambda)}(\underline{S})$ there is some $q\in  B(x_i,\Lambda \epsilon_i)$ such that 
\[\underline{d}(H_i(q),p)\le \delta_i\]
\item $\forall p_0, p_1\in B(x_i,\Lambda \epsilon_i)\colon \left|\underline{d}(H_i(p_0),H_i(p_1))-\epsilon_i^{-1}d(p_0,p_1)\,)\right|< \delta_i$.
\end{itemize}
\subsubsection*{Convergence of points}We say that a sequence $y_i\in X$ converges to $\underline{y}\in \underline{X}$ if 
$H_i(y_i)$ converges to $\underline{y}$.
\subsubsection*{Uniform convergence }
A sequence of function $f_i\colon X\rightarrow \R$ is said to converge uniformly on compact set to $\underline{f}\colon \underline{X}\rightarrow \R$ if
for each $\Lambda$ 
\[\lim_i \left\| f_i-\underline{f}\circ H_i\right\|_{L^\infty(B(x_i,\Lambda\epsilon_i))}=0\]
\subsubsection*{Convergence of the measure} The geometry of $(\underline{X},\underline{d})$ comes from an iterated edge metric $\underline{g}$ and has an associated measure $\underline{\mu}$. We have in fact the pointed measure Gromov-Hausdorff convergence of $(X,\epsilon_i^{-1}d ,\epsilon_i^{-n}\mu, x_i)$ to 
$(\underline{X},\underline{d},\underline{\mu})$. 
\subsubsection*{Convergence in $L^p$}
A sequence $f_i\in L^p(X)$ is said to converge weakly to $f\in L^p(\underline{X})$ if 
 \begin{itemize}
 \item For any $\varphi_i\in \cC^0(X)$ with bounded support (meaning there is $\Lambda>0$ such that for all $i$ we have $\text{supp}(\phi_i)\subset B(x_i,\Lambda \epsilon_i)$) that converge uniformly to $\varphi\in \cC^0(\underline{X})$ (hence $\text{supp}(\phi)\subset B(\underline{x},\Lambda)$) we have
 \[\lim_i \int_Xf_i\varphi_i \frac{d\mu}{\epsilon_i^{n}}=\int_{\underline{X}}f \varphi d\underline{\mu}.\]
 \item $\sup_i \int_X \vert f_i\vert^p \frac{d\mu}{\epsilon_i^{n}}<\infty$.
 \end{itemize}
 If moreover $\lim_i \int_X \vert f_i\vert^p \frac{d\mu}{\epsilon_i^{n}}=\int_{\underline{X}}\vert f\vert^p  d\underline{\mu}$ we say that $f_i\in L^p(X)$ converges strongly to $f\in L^p(\underline{X})$.
 It is known that if $p\in (1,\infty)$ then a bounded sequence (i.e. satisfying $\sup_i \int_X \vert f_i\vert^p \frac{d\mu}{\epsilon_i^{n}}<\infty$ always has a weakly convergent subsequence. Moreover if $q=p/(p-1)$ is the conjugate exponent then if $f_i\in L^p(X)$ converges weakly to $f\in L^p(\underline{X})$ and $v_i\in L^q(X)$ converge strongly to $v\in L^q(\underline{X})$ then
 \[\lim_i \int_X f_i v_i \frac{d\mu}{\epsilon_i^{n}}=\int_{\underline{X}}f v  d\underline{\mu}.\]
 
 We can also define the notion of weak convergence in $L^p_{loc}$ requiring that for any $\Lambda>0$:
 \[\sup_i \int_{B(x_i,\Lambda \epsilon_i)}  |f_i|^p \frac{d\mu}{\epsilon_i^{n}}<\infty,\] 
 and of strong convergence in $L^p_{loc}$ requiring that for any $\Lambda>0$: $\lim_i \int_{B(x_i,\Lambda \epsilon_i)} \vert f_i\vert^p \frac{d\mu}{\epsilon_i^{n}}=\int_{B(\underline{x},\Lambda) }\vert f\vert^p  d\underline{\mu}$.
\subsubsection*{Convergence in $H^1$}
A sequence $f_i\in H^1(X)$ is said to converge weakly to $f\in H^1(\underline{X})$ if 
 \begin{itemize}
 \item  $f_i$ converges weakly to $f\in L^{\frac{2n}{n-2}}(X)$,
 \item $\sup_i \int_X \vert \nabla f_i\vert_{g_0}^2 \epsilon_i^{2-n} d\mu<\infty$.
 \end{itemize}
 Notice that our definition is slightly different the the usual one in order to take into account that the $L^2$ norm of $f$ and of $\nabla f$ do not rescaled in the same way, but the $L^{\frac{2n}{n-2}}$ norm of $f$ and the $L^2$-norm of $\nabla f$ do.
 Such a sequence is said to converge strongly in $H^1$ if it converges strongly in $L^{\frac{2n}{n-2}}$ and if
\[\lim_i  \int_X \vert \nabla f_i\vert_{g_0}^2 \epsilon_i^{2-n} d\mu=\int_{\underline{X}}\vert \nabla  f\vert_{\underline{g}}^2d\underline{\mu}.\]
For any $u\in H^1(\underline{X})$, there is a sequence $u_i\in H^1(X)$ that converges strongly in $H^1$ to $u$.

 We can similarly  define the weak and strong convergence in $H^1_{loc}(X)$.
  Moreover a bounded sequence in $H^1$ (i.e. satisfying $\sup_i \int_X |f_i|^{\frac{2n}{n-2}} \frac{d\mu}{\epsilon_i^{n}}<\infty$ and $\sup_i \int_X|\nabla f_i|^2 \epsilon_i^{2-n} d\mu<\infty$)   always has a weakly convergent subsequence $H^1$.
  Moreover, weak limit in $H^1$  converge strongly in $L^2_{loc}$.
\subsubsection*{Mosco convergence} The sequence of quadratic forms
$\cE_i(u)=\int_X |\nabla u|_{g_0}^2 \epsilon_i^{2-n} d\mu$ converges in the Mosco sense to $\underline{\cE}(u)=\int_{\underline{X}}|\nabla u|_{g_0}^2d\underline{\mu}.$ We will not give the definition here, but it implies for instance the convergence of the corresponding heat kernel.
\subsection*{The results}
Assume that $R\in L^\infty( X)$ and $\sigma>0$. Let $c_n=4(n-1)/(n-2)$ , and for $u\in H^1(X)$ we set:
\begin{equation}
I(u)=\frac12 \int_X \left[ c_n|\nabla u|_{g_0}^2+Ru^2\right]d\mu-\frac{n-2}{2n}\sigma\int_X |u|^{\frac{2n}{n-2}}d\mu
\label{eq:YamabeEnergy}
\end{equation}
\begin{Rem}
This is the form the Yamabe energy takes, with $R=S_0$. The additional term then serves as a Lagrange multiplier, ensuring the volume is normalized.
\end{Rem}
\begin{Def} A bounded sequence $(u_\alpha)$  in $H^1$ is said to be a Palais-Smale sequence for $I$ if the sequence 
$(I(u_\alpha))_\alpha$ converges and if $$DI(u_\alpha)\to 0\text{ in } \left(H^1\right)^*.$$
\end{Def}
The definition in particular says that there is a sequence of real numbers $(\delta_\alpha)$  converging to zero such that
\[\forall \upvarphi\in H^1(X)\colon \left| \int_X \left[ c_n\langle \nabla u_{\alpha},\nabla \upvarphi\rangle_{g_0}+Ru_{\alpha}\upvarphi\right]d\mu-\sigma\int_X |u_{\alpha}|^{\frac{4}{n-2}}u_{\alpha}\upvarphi d\mu\right|\le \delta_\alpha \|\upvarphi\|_{H^1}.\]

The blow up profile will be governed by bubbles:
\begin{Def}  A bubble is a sequence $\left(B_\alpha\right)$ defined by
\[B_\alpha(x)=\left(\frac{c\lambda\epsilon_\alpha}{(\lambda\epsilon_\alpha)^2+d(x,x_\alpha)^2}\right)^{\frac{n-2}{2}}\]
 associated to a sequence of points $(x_\alpha)$ and a sequence of positive real numbers$(\epsilon_\alpha)$ such that 
\begin{itemize}

\item $\lim\limits_{\alpha\to\infty} \epsilon_\alpha=0$
\item There is $x\in X$ such that $\lim\limits_{\alpha\to\infty} x_\alpha=x$.
\item the sequence of rescaled spaces $(X,\epsilon_\alpha^{-1}d,x_\alpha)$ converges in the Gromov-Hausdorff topology to  $(\underline{X},\underline{d},\underline{x})$ a fake tangent space  at $x$
\item $(\underline{X},\underline{d})$ is conical at $\underline{x}$.
\end{itemize}
We called $x_\alpha$ the center of the bubble and $\epsilon_\alpha$ will be called the scale of the bubble. 
\end{Def}

We will choose the following norm on $H^1(X,g_0)$:
\[ \|\upvarphi\|_{H^1(X,g_0)}=\sqrt{ \int_X |\nabla \upvarphi|_{g_0}^2d\mu+\left(\int_X|\upvarphi|^{\frac{2n}{n-2}}d\mu\right)^{1-\frac 2n}\, }.\]
By the Sobolev inequality \eqref{eq:Sobolev} this norm is equivalent to the traditional Hilbert norm, it has the advantage that it behaves nicely under rescaling
\[ \|\upvarphi\|_{H^1(X,g_0)}=\|\varepsilon^{\frac{n-2}{2}}\upvarphi\|_{H^1(X,\epsilon^{-2}g_0)}.\]
This will be very convenient when we try to understand the blow-up behaviour of Palais-Smale sequences.

We have the following observation:
\begin{Prop}\label{bubblecv} Assume that $\left(B_\alpha\right)$ is a bubble with center $x_\alpha $ and scale $\epsilon_\alpha$ such that $(X,\epsilon_\alpha^{-1}d,x_\alpha)$ converges to $(\underline{X},\underline{d},\underline{x})$. Then
for $v=\left(\frac{c\lambda}{\lambda^2+\underline{d}(x,\underline{x})^2}\right)^{\frac{n-2}{2}}$, we have that 
 $\left(B_\alpha\right)$ converges strongly to $v$ in $H^1$-norm.
\end{Prop}
\proof By definition,  
\[\left(\epsilon_\alpha^{\frac{n-2}{2}}B_\alpha\right)= \left(\frac{c\lambda}{\lambda^2+\frac{d(x,x_\alpha)^2}{\epsilon_{\alpha}^2}}\right)^{\frac{n-2}{2}}\]
 converges uniformly on compact sets to $v$, and following classical computations (see for instance \cite[section 2]{CAFST} ) we have
\[  \|\epsilon_\alpha^{\frac{n-2}{2}}B_\alpha\|_{L^{\frac{2n}{n-2}}(X,\epsilon_\alpha^{-2}g_0)}=\|B_\alpha\|_{L^{\frac{2n}{n-2}}(X,g_0)}=\left(\int_0^{+\infty} \frac{n2 r(c\lambda\epsilon_\alpha)^{n}}{\left((\lambda\epsilon_\alpha)^2+r^2\right)^{n+1}}\mu(B(x_\alpha,r))   dr\right)^{\frac{n-2}{2n}}.\]
Using the change of variable $r=\epsilon_\alpha\rho$, one gets:
\[ \|B_\alpha\|_{L^{\frac{2n}{n-2}}(X,g_0)}=\left(\int_0^{+\infty} \frac{2n \rho(c\lambda)^{n}}{\left(\lambda^2+\rho^2\right)^{n+1}}\frac{\mu(B(x_\alpha,\epsilon_\alpha\rho))}{\epsilon_\alpha^n}   d\rho\right)^{\frac{n-2}{2n}}.\]
Note that by Ahlfors regularity \eqref{eq:Ahlfors-Regular}, there is a constant $C$ such that for any $x\in X$ and $r>0$:
$\mu(B(x,r))\le C r^n$, so
\[ \|B_\alpha\|_{L^{\frac{2n}{n-2}}(X,g_0)}\leq 2n Cc^n \lambda^n \int_0^\infty \frac{\rho^{n+1}}{(\lambda^2 +\rho^2)^{n+1}}\, d\rho \]
is bounded independently of $\alpha$. 
 By the measure Gromov-Hausdorff convergence, we have for any $\rho>0$ that 
\[\lim_\alpha \frac{\mu(B(x_\alpha,\epsilon_\alpha\rho))}{\epsilon_\alpha^n}=\underline{\mu}(B(\underline{x},\rho).\]
Hence, by the dominated convergence theorem, we get
\[\lim_\alpha \|B_\alpha\|_{L^{\frac{2n}{n-2}}(X,g_0)}= \|v\|_{L^{\frac{2n}{n-2}}(X,g_0)}.\]
We similarly have
\begin{multline*}
\int_X \vert \nabla B_\alpha \vert_{g_0}^2 d\mu=\\
4\left(\frac{n-2}{2}\right)^2(c\epsilon_\alpha\lambda)^{n-2}\int_0^{+\infty} 2r\left( \frac{nr^2}{\left((\epsilon_\alpha\lambda)^2+r^2\right)^{n+1}}-\frac{1}{\left((\epsilon_\alpha\lambda)^2+r^2\right)^n}\right)\mu(B(x_\alpha,r))dr,\end{multline*}
and and by the same combination of arguments as before, we get 
\[\lim_\alpha \|\nabla B_\alpha\|_{L^2(X,g_0)}= \|\nabla v\|_{L^2(X,g_0)}.\]
Hence $(B_\alpha)$ converges strongly  to $v$  in $H^1$-norm.
\endproof
Our main result in this appendix is the following.
\begin{Thm}\label{main} Assume that the Ricci curvature of $g$ is bounded, meaning there is a $\Lambda$ such that
\[\forall x\in X_{reg}\colon \|\mathrm{Ricci}\|_{g_0}\le \Lambda\]
 and that the cone angles of the tangent spaces along $X_{n-2}\setminus X_{n-3}$ are always less than $2\pi$.
Let $(u_\alpha)$ be a  Palais-Smale sequence for $I$  of non-negative functions.  Then, up to extracting a subsequence, there is a non-negative function 
$u_\infty\in H^1(X)$ solving the equation
\[-c_n\Delta u_\infty+Ru_\infty=\sigma u_\infty^{\frac{n+2}{n-2}}\]
 and  a finite number of bubbles $B_\alpha^1,\dots,B_\alpha^L$ such that 
\begin{equation}\label{bubbles}\lim_{\alpha\to \infty} \left\| u_\alpha-u_\infty-\sum_{j=1}^L B_\alpha^j\right\|_{H^1}=0.\end{equation}
Moreover, if the bubble $(B_\alpha^j)$ has center $x_\alpha(j) $ and scale $\epsilon_\alpha(j)$, then for any $i\not=j$ we have
\begin{equation}\label{sepa}\lim_{\alpha\to+\infty} \frac{\epsilon_\alpha(j)}{\epsilon_\alpha(i)}+\frac{\epsilon_\alpha(i)}
{\epsilon_\alpha(j)}+\frac{d^2(x_\alpha(j),x_\alpha(i))}{\epsilon_\alpha(i)\epsilon_\alpha(j)}=+\infty.\end{equation}
\end{Thm}
We prove the theorem below. The assumptions on the Ricci curvature  are made in order to apply the rigidity result of Mondello \cite{Mondello} which yield the following.
\begin{Thm}\label{rigid} Under this assumptions of Theorem \ref{main}, if $\underline{X}$ is a fake tangent space of $X$ at $x$ and if $v\in H^1(\underline{X})$ is a non-negative function solving the  equation
\[-c_n\Delta v=\sigma v^{\frac{n+2}{n-2}},\] 
then we can find $\underline{x}\in \underline{X}$ such that $ \underline{X}$ is conical at $\underline{x}$ and a $\lambda>0$ such that
\[v(x)=\left(\frac{c\lambda}{\lambda^2+\underline{d}(x,\underline{x})^2}\right)^{\frac{n-2}{2}}\]
 with
\[c=\sqrt{\frac{n(n-1)}{\sigma}}.\]
\end{Thm}
\proof The hypotheses imply that $(\underline{X},\underline{g})$ is Ricci flat on its regular part and that it has no cone angles greater than $2\pi$. Moreover we know that $\underline{X}=C_{[0,+\infty)}(\underline{S})$ is a metric cone over some smoothly stratified Riemannian pseudomanifold $(\underline{S}, g_{\underline{S}})$ which is Einstein on its regular part with scalar curvature equal to $(n-1)(n-2)$. Hence we can conformally compactify $(\underline{X},\underline{g})$ and obtain
$\widehat{\underline{X}}$ whose regular part is $(0,\pi)\times \underline{S}_{reg}$ endowed with the metric $\hat g=d\theta^2+\sin^2(\theta)g_{\underline{S}}$ which is Einstein with scalar curvature $n(n-1)$ and has also no cone angles greater than $2\pi$.  So that the metric
\[v^{\frac{4}{n-2}}\underline{g}=f^{\frac{4}{n-2}} \hat g\] has constant scalar curvature. I. Mondello  has proven (see \cite[Theorem 4.1]{Mondello}) that there is another smoothly stratified Riemannian pseudomanifold $(\Sigma, g_{\Sigma})$ which is Einstein on its regular part with scalar curvature equal to $(n-1)(n-2)$ so that $\left(\widehat{\underline{X}},v^{\frac{4}{n-2}}\underline{g}\right)$ is isometric to the spherical suspension over $\Sigma$,
\[v^{\frac{4}{n-2}}\underline{g}=\frac{n(n-1)}{\sigma} \left(d\theta^2+\sin^2(\theta)g_\Sigma\right).\]
So
\begin{equation}
\label{eq:corespond}
v^{\frac{4}{n-2}}\underline{g}=\frac{n(n-1)}{\sigma} \frac{1}{(1+r^2)^2}\left(dr^2+r^2g_\Sigma\right)
\end{equation}
The metric $dr^2+r^2g_\Sigma$ is then conformal to $\underline{g}$ and has zero scalar curvature, that is to say
\[dr^2+r^2g_\Sigma=h^{\frac{4}{n-2}} \underline{g}\]
 where $h$ is harmonic and non negative. Furthermore, $(\underline{X},g_{\underline{X}})$ satisfies the elliptic Harnack inequalities (as in Proposition \ref{regularH}), hence any non negative harmonic function is constant. So there is some $\lambda>0$ so that 
$h=\lambda^{-\frac{n-2}{2}}$. If $\underline{x}$ is the tip of the cone $C_{[0,+\infty)}(\Sigma)$ then
\[r(x)=d_{dr^2+r^2g_\Sigma}(x,\underline{x})=\lambda^{-1} d_{\underline{g}}(x,\underline{x}).\]
Using (\ref{eq:corespond}) and $dr^2+r^2g_\Sigma= \lambda^{-2} \underline{g}$, we get  
\[v^{\frac{4}{n-2}}=\frac{n(n-1)}{\sigma}  \lambda^{-2} \frac{1}{(1+r^2)^2}=\frac{n(n-1)}{\sigma}  \frac{\lambda^2}{(\lambda^2+d^2_{\underline{g}}(x,\underline{x}))^2}.\] 
\endproof
With this result in hand we can adapt the classical proof. We will follow the nice exposition given by E. Hebey in \cite[section 3]{Hebey} which we find very suitable for generalization to a singular setting. There are several steps whose proof are identical to the one in the case of smooth manifolds, and we will refer to the corresponding statements in that monograph.

\subsection*{Proof of Theorem \ref{main}}
\proof
The proof is long, and will be split into several parts with lemmas.
\subsubsection*{Step 1 --  From weak to strong $H^1$ convergence}\medskip

\noindent 
\begin{Lem}[{\cite[Lemma 3.3]{Hebey}}] Assume that $(u_\alpha)$ is a Palais-Smale sequence for $I$ which converges weakly to $u$ in $H^1(X)$, then $v_\alpha\coloneqq u_\alpha-u$ defines a Palais-Smale sequence for 
\[I_0(\varphi)=\frac12 \int_Xc_n |\nabla \varphi|_{g_0}^2 d\mu-\frac{n-2}{2n}\sigma\int_X |\varphi|^{\frac{2n}{n-2}}d\mu.\]
Moreover the sequence of measures $|u_\alpha|^{\frac{2n}{n-2}}d\mu-|v_\alpha|^{\frac{2n}{n-2}}d\mu$ converges weakly to $|u|^{\frac{2n}{n-2}}d\mu$.
\end{Lem}

The next result says that in the setting of the previous lemma, the defect of strong convergence in $H^1(X)$ is measured by the concentration of the measure $|v_\alpha|^{\frac{2n}{n-2}}d\mu$. By \cite[Proposition 1.4a)]{ACM} we get the Sobolev inequality 
\begin{equation}
\label{eq:AppSobolev}
\forall \varphi\in H^1(X)\colon \frac 34 \cS_\ell(X)\left(\int_X |\varphi|^{\frac{2n}{n-2}}d\mu\right)^{1-\frac 2n}\le \int_X\left[ c_n  |\nabla \varphi|_{g_0}^2+B\varphi^2\right]d\mu,
\end{equation}
where $\cS_{\ell}(X)$ is the local Sobolev constant \eqref{eq:SobolevConst} and $B>0$ is some constant.
\begin{Lem} Assume that $(u_\alpha)$ is a Palais-Smale sequence for $I$ that converge weakly to $u$ in $H^1(X)$, and set  $v_\alpha=u_\alpha-u$.
Assume that for some $x\in X$ and $\delta>0$ we have
\[\sigma \left(\int_{B(x,\delta)} |v_\alpha|^{\frac{2n}{n-2}}d\mu\right)^{\frac 2n}\le \frac 12  \cS_\ell(X).\]
Then $\lim\limits_{\alpha\to \infty} \| v_\alpha\|_{H^1(B(x,\delta/2)}=0.$
\end{Lem}
The proof is identical to the first step in the proof of \cite[Theorem 3.2]{Hebey}. The following is a generalization of this result under a rescaling.
\begin{Lem}
\label{Lem:LocalConv}
 Assume that $(u_\alpha)$ is a Palais-Smale sequence for $I$. Assume that  $(\underline{X},\underline{d},\underline{x})$ is a fake tangent space  at $x$ i.e. for some sequence  of points $(x_\alpha)$ and some sequence of positive real numbers$(\epsilon_\alpha)$, the rescaled spaces $(X,\epsilon_\alpha^{-1}d,x_\alpha)$ converges for the Gromov-Hausdorff topology to $(\underline{X},\underline{d},\underline{x})$.
Assume that along this sequence $\epsilon_\alpha^{\frac{n-2}{2}}u_\alpha$ converges weakly to $u\in H^1(\underline{X})$ and that for any $\Lambda>0$, we can find $\alpha_0$ such that for any $\alpha\geq \alpha_0$ and any
$x\in B(x_{\alpha},\Lambda\epsilon_{\alpha})$, we have 
 \begin{equation}
 \sigma \left(\int_{B(x,\epsilon_\alpha)} |u_\alpha|^{\frac{2n}{n-2}}d\mu\right)^{\frac 2n}\le \frac 12  S_\ell(X).
 \label{eq:SobolevMassBound}
 \end{equation}
 Then $\epsilon_\alpha^{\frac{n-2}{2}}u_\alpha$ converges strongly to $u\in H_{loc}^1(\underline{X})$ , meaning
 for any $\Lambda>0$ we have
 \[\lim_{\alpha\to\infty} \int_{B(x_\alpha,\Lambda \epsilon_\alpha)} |u_\alpha|^{\frac{2n}{n-2}}d\mu=\int_{B(\underline{x},\Lambda )} |u|^{\frac{2n}{n-2}}d\underline{\mu}\]
  and 
  \[\lim_{\alpha\to\infty} \int_{B(x_\alpha,\Lambda \epsilon_\alpha)} |\nabla u_\alpha|_{g_0}^{2}d\mu=\int_{B(\underline{x},\Lambda )}  |\nabla u|_{\underline{g}}^{2}d\underline{\mu}.\]
\end{Lem}
\begin{proof}[Proof of Lemma \ref{Lem:LocalConv}] 
Let $\varphi\in H^1(\underline{X})$ be arbitrary. By assumption, we can find sequences $(\varphi_\alpha)$ and $(w_\alpha)$ in $H^1(X)$ such that when
\begin{equation}\label{CV1}
(X,\epsilon_\alpha^{-1}d,x_\alpha)\to (\underline{X},\underline{d},\underline{x}),
\end{equation}
 then  $(\epsilon_{\alpha}^{\frac{n-2}{2}}w_{\alpha})$ converges strongly in $H^1$ to $u$ (in the sense defined above) and $(\epsilon_{\alpha}^{\frac{n-2}{2}}\varphi_{\alpha})$ converges strongly in $H^1$ to $\varphi$.
We use $\varphi_{\alpha}$ as a test function in the definition of $DI(u_{\alpha})\to 0$ in $(H^1)^*$, and find 
\[\left\vert \int_X\left( c_n \ip{\nabla u_{\alpha}}{\nabla \varphi_{\alpha}}_{g_0} -\sigma \vert u_{\alpha}\vert^{\frac{4}{n-2}} u_{\alpha}\varphi_{\alpha}\, \right)\,d\mu \right\vert\leq  \delta_{\alpha} \norm{\varphi_{\alpha}}_{H^1} + \norm{R}_{L^{\infty}}\int_X u_{\alpha}\varphi_{\alpha} \, d\mu.\]
By the definition of $\epsilon_{\alpha}^{\frac{n-2}{2}} \varphi_{\alpha}$ converging to $\varphi$, the norm
\begin{align*}
\norm{\varphi_{\alpha}}^2_{H^1}&= \left(\int_X \vert \varphi_{\alpha}\vert^{\frac{2n}{n-2}} \,d\mu\right)^{\frac{n-2}{n}} + \int_X \vert \nabla \varphi_{\alpha}\vert^2_{g_0}\, d\mu \\ 
&=\left(\int_X \vert \epsilon_{\alpha}^{\frac{n-2}{2}}\varphi_{\alpha}\vert^{\frac{2n}{n-2}} \,\frac{d\mu}{\epsilon_{\alpha}^n}\right)^{\frac{n-2}{n}} + \int_X \vert \nabla \epsilon_{\alpha}^{\frac{n-2}{2}}\varphi_{\alpha}\vert^2_{g_0}\, \frac{d\mu}{\epsilon_{\alpha}^{n-2}} 
\end{align*}
  is bounded uniformly in $\alpha$. Furthermore, since $\epsilon_{\alpha}^{\frac{n-2}{2}}u_{\alpha}$ converges weakly to $u$ in $H^1$ and $\epsilon_{\alpha}^{\frac{n-2}{2}}\varphi_{\alpha}$ converges strongly to $\varphi$ in $H^1$, we have 
 \[\lim_{\alpha\to \infty} \int_X (\epsilon_{\alpha}^{\frac{n-2}{2}}u_{\alpha}) (\epsilon_{\alpha}^{\frac{n-2}{2}}\varphi_{\alpha}) \frac{d\mu}{\epsilon_{\alpha}^n}=\int_{\underline{X}} u\varphi\, d\underline{\mu}.\]
 This means 
  \[\int_X u_{\alpha}\varphi_{\alpha} \, d\mu=\epsilon_{\alpha}^2\int_X (\epsilon_{\alpha}^{\frac{n-2}{2}}u_{\alpha}) (\epsilon_{\alpha}^{\frac{n-2}{2}}\varphi_{\alpha}) \frac{d\mu}{\epsilon_{\alpha}^n} \xrightarrow{\alpha \to 0} 0,\]
hence
\[\lim_{\alpha\to 0}  \left\vert \int_X\left( c_n \ip{\nabla u_{\alpha}}{\nabla \varphi_{\alpha}}_{g_0} -\sigma \vert u_{\alpha}\vert^{\frac{4}{n-2}} u_{\alpha}\varphi_{\alpha}\, \right)\,d\mu \right\vert=0.\]
The weak convergence of $\epsilon_{\alpha}^{\frac{n-2}{2}}u_{\alpha}$ and strong convergence of $\epsilon_{\alpha}^{\frac{n-2}{2}}\varphi_{\alpha}$ combine to give
\[ \int_{\underline{X}} c_n \ip{\nabla \phi}{\nabla u}_{\underline{g}} \, d\underline{\mu} - \sigma \int_{\underline{X}} \vert u\vert^{\frac{4}{n-2}} u \phi \, d\underline{\mu}=0.\]
 Since $\phi$ was arbitrary, this means $u$ is a solution of the equation 
\begin{equation}\label{equau}
-c_n\Delta u=\sigma |u|^{\frac{4}{n-2}} u
\end{equation} 
in the weak sense on $\underline{X}$.

Furthermore, we claim $(w_\alpha)$ is also a Palais-Smale sequence for $I_0$.  If it were not the case, we could, up to a subsequence extraction, find $\eta>0$, $\varphi_\alpha\in H^1(X)$ with $\|\varphi_\alpha\|_{H^1(X)}=1$ such that 
\begin{equation}\label{inequau}\left|\int_Xc_n \langle \nabla w_\alpha,\nabla \varphi_\alpha\rangle_{g_0} d\mu-\sigma\int_X |w_\alpha|^{\frac{4}{n-2}}w_\alpha\varphi_\alpha d\mu\right|\ge \eta>0.\end{equation}
Again up to extraction of subsequence, we can assume that when $(X,\epsilon_\alpha^{-1}d,x_\alpha)\to (\underline{X},\underline{d},\underline{x})$ then
$\epsilon_\alpha^{\frac{n-2}{2}}\varphi_\alpha$ converges weakly to some  $\varphi\in H^1(\underline{X})$.
By scaling we can pass to the limit in the inequality (\ref{inequau}) and get 
\[\left|\int_{\underline{X}}c_n \langle \nabla u,\nabla \varphi\rangle_{\underline{g}} d\underline{\mu}-\sigma\int_{\underline{X}} |u|^{\frac{4}{n-2}}u\varphi d\underline{\mu}\right|\ge \eta,\]
 in contradiction with (\ref{equau}).

Then, we set $v_\alpha= u_\alpha-w_\alpha$. We show that $(v_\alpha)$ is a  Palais-Smale sequence for $I$. The difficulty is the non linear term and if we show that 
\[\Psi_\alpha:= |u_\alpha|^{\frac{4}{n-2}}u_\alpha-|w_\alpha|^{\frac{4}{n-2}}w_\alpha-|v_\alpha|^{\frac{4}{n-2}}v_\alpha\] tends to $0$ in $L^{\frac{2n}{n+2}}$ then it is easy to check that  $(v_\alpha)$ is a  Palais-Smale sequence for $I$.
Notice that 
\begin{align*}
\Psi_\alpha&=w_{\alpha}\left( \vert v_{\alpha}+w_{\alpha}\vert^{\frac{4}{n-2}} -\vert w_{\alpha}\vert^{\frac{4}{n-2}}\right)+ v_{\alpha}\left( \vert v_{\alpha}+w_{\alpha}\vert^{\frac{4}{n-2}} -\vert v_{\alpha}\vert^{\frac{4}{n-2}}\right)
\end{align*}
Observe that there is a constant $C_n$ such that for any real numbers $x,y$:
\[\left|  \vert x+y\vert^{\frac{4}{n-2}}(x+y)-\vert x\vert^{\frac{4}{n-2}}x -\vert y\vert^{\frac{4}{n-2}}y\right|\le  C_n\left(|x|^{\frac{4}{n-2}}|y|+|y|^{\frac{4}{n-2}}|x|\right).\]
Hence here are constants $C_n$ depending only on $n$ such that
\[|\Psi_\alpha|\le C_n\left(|w_\alpha|^{\frac{4}{n-2}}|v_\alpha|+|v_\alpha|^{\frac{4}{n-2}}|w_\alpha|\right).\]
%For $n\geq 6$, we can use the Bernoulli inequality to deduce
%\[|\Psi_\alpha|\le \frac{4}{n-2}\left(|w_\alpha|^{\frac{4}{n-2}}|v_\alpha|+|v_\alpha|^{\frac{4}{n-2}}|w_\alpha|\right).\]
%For $3\leq n\leq 5$ one can find similar bounds. So all in all, there are constants $C_n$ depending only on $n$ such that
%\[|\Psi_\alpha|\le C_n\left(|w_\alpha|^{\frac{4}{n-2}}|v_\alpha|+|v_\alpha|^{\frac{4}{n-2}}|w_\alpha|\right).\]
Hence to control $\Psi_{\alpha}$, it is enough to show that 
\[\lim_{\alpha\to \infty}\left\| |w_\alpha|^{\frac{4}{n-2}}|v_\alpha|\right\|_{L^{\frac{2n}{n+2}}(X)}=0\text{ and } \lim_{\alpha\to \infty}\left\| |v_\alpha|^{\frac{4}{n-2}}|w_\alpha|\right\|_{L^{\frac{2n}{n+2}}(X)}=0.\]
We define $\tilde w_\alpha=\epsilon_\alpha^{\frac{n-2}{2}}w_\alpha$ and $\tilde v_\alpha=\epsilon_\alpha^{\frac{n-2}{2}}v_\alpha$. 
For any $r\ge \frac{n-2}{2n}$,
$|\tilde w_\alpha|^{\frac 1 r}$ converges strongly in $L^{\frac{2rn}{n-2}}$ to $\vert u\vert^{\frac 1 r}$ along the convergence (\ref{CV1}).
For any $p\ge \frac{n-2}{2n}$,
$|\tilde v_\alpha|^{\frac 1 p}$ converges weakly  $L^{\frac{2pn}{n-2}}$ to $0$ along the convergence (\ref{CV1}). Indeed this is a bounded sequence with a unique sublimit because $(\tilde v_\alpha)$ converges weakly to $0$ in $H^1$. Hence if $\frac1r+\frac 1p=\frac{2n}{n-2}$ then
\[\lim_{\alpha\to \infty} \int_X|\tilde w_\alpha|^{\frac 1 r}|\tilde v_\alpha|^{\frac 1 p}\frac{d\mu}{\epsilon_\alpha^n}=0\]
Choosing $1/p=\frac{4}{n-2}\frac{2n}{n+2}$ and $1/r=\frac{2n}{n+2}$, we deduce
\begin{align*}
\norm{|v_\alpha|^{\frac{4}{n-2}}w_\alpha}_{L^\frac{2n}{n+2}(X)}^{\frac{n+2}{2}}& =\int_X \vert w_{\alpha}\vert^{\frac{1}{r}}\vert v_{\alpha}\vert^{\frac{1}{p}} d\mu \\ 
&= \int_X \vert \tilde{w}_{\alpha}\vert^{\frac{1}{r}}\vert \tilde{v}_{\alpha}\vert^{\frac{1}{p}} \frac{d\mu}{\epsilon^n}\\
& \xrightarrow{\alpha \to \infty} 0.
\end{align*} 
Swapping the roles of $p$ and $r$ we get the other result. This establishes that $v_\alpha$ is a Palais-Smale sequence for $I$.\medskip 

\noindent
We will use the argumentation in the proof of \cite[Lemma 3.3]{Hebey} and we are going to prove that if $y_\alpha\to \underline{y}\in \underline{X}$ then
\[\lim_{\alpha\to \infty}\int_{B(y_\alpha,\epsilon_\alpha/2)} |\nabla v_\alpha|_{g_0}^2d\mu=0\text{ and } \lim_{\alpha\to \infty}\int_{B(y_\alpha,\epsilon_\alpha/2)} |v_\alpha|^{\frac{2n}{n-2}}d\mu=0.\]
With a simple covering argument this will  imply that for any $\Lambda>0$
\[\lim_{\alpha\to \infty}\int_{B(x_\alpha,\Lambda\epsilon_\alpha)} |\nabla v_\alpha|_{g_0}^2d\mu=0\text{ and } \lim_{\alpha\to \infty}\int_{B(x_\alpha,\Lambda\epsilon_\alpha)} |v_\alpha|^{\frac{2n}{n-2}}d\mu=0.\]
That is to say that along the convergence $(B(x_\alpha,\Lambda\epsilon_\alpha),\epsilon_\alpha^{-1}d,x_\alpha)\to (B(\underline{x},\Lambda),\underline{d},\underline{x})$, 
the sequence $(\epsilon_{\alpha}^{\frac{n-2}{2}}u_\alpha)$ converges to $u$ strongly in $H^1$.
We use the cut-off function
\[\chi_\alpha(x)=\begin{cases}1&\text{ on } B(y_\alpha,\epsilon_\alpha/2)\\
2\left(1-d(x,y_\alpha)/\epsilon_\alpha \right)&\text{ on } B(y_\alpha,\epsilon_\alpha)\setminus B(y_\alpha,\epsilon_\alpha/2)\\
0 &\text{ outside } B(y_\alpha,\epsilon_\alpha).
\end{cases}\]
Then we get $|\nabla \chi_\alpha|_{g_0}\le 2/\epsilon_\alpha$ and with the Ahlfors regularity of the measure \eqref{eq:Ahlfors-Regular}, we find a constant such that for any $\alpha$
\begin{equation}\label{estLn} 
\|\nabla\chi_\alpha\|_{L^n(X)}^n=\int_{B(x_{\alpha},\epsilon_{\alpha})} \vert \nabla \chi_\alpha\vert_{g_0}^n d\mu \leq    2^nC.
\end{equation}
The fact that $v_{\alpha}$ is a Palais-Smale sequence yields 
\[\left|\int_Xc_n \langle \nabla v_\alpha,\nabla (\chi_\alpha^2v_\alpha)\rangle_{g_0} d\mu+\int_X R\chi_\alpha^2v_\alpha^2d\mu-\sigma\int_X |v_\alpha|^{\frac{4}{n-2}}\chi_\alpha^2v_\alpha ^2d\mu\right| \le o(1)\| \chi_\alpha^2 v_\alpha\|_{H^1(X)}.\]
We have 
\begin{align*}
\int_X |\nabla (\chi_\alpha^2 v_\alpha)|_{g_0}^2d\mu&\le \int_X \chi_{\alpha}^2  \vert \nabla v_\alpha\vert_{g_0} ^2d\mu+4\int_X  \vert \nabla  \chi_\alpha\vert_{g_0}^2v_\alpha^2d\mu+4\int  \vert v_{\alpha} \vert \nabla v_{\alpha}\vert_{g_0} \vert \chi_{\alpha} \vert \nabla \chi_{\alpha}\vert_{g_0}\, d\mu\\
&\leq  3 \int_X   \vert \nabla v_\alpha\vert_{g_0} ^2d\mu+6\int_X  \vert \nabla  \chi_\alpha\vert_{g_0}^2v_\alpha^2d\mu\\
& \le 3\int_X |\nabla v_\alpha|_{g_0}^2d\mu+6 \|\nabla \chi_\alpha\|_{L^n(X)}^2\left(\int_X  v_\alpha^{\frac{2n}{n-2}}d\mu\right)^{1-\frac 2n}\\
&\le C \|v_\alpha\|^2_{H^1(X)},
\end{align*}
where we use Young's inequality, the H\"older inequality and then the estimate (\ref{estLn}).
Notice that the sequence $(v_\alpha)$ is bounded in $H^1$, 
 and that along the convergence (\ref{CV1}), $(\tilde v_\alpha)$ converges weakly to $0$ in $H^1$ hence it converges strongly in $L_{loc}^2$ to $0$ and 
then
\begin{equation}\label{estL2}\int_{B(y_\alpha,\epsilon_\alpha)} v^2_\alpha d\mu=o(1)\epsilon_\alpha^2.\end{equation}
We also have
\[\int_X\langle \nabla v_\alpha,\nabla (\chi_\alpha^2v_\alpha)\rangle_{g_0} d\mu=\int_X |\nabla  (\chi_\alpha v_\alpha)|_{g_0}^2d\mu-\int_X |\nabla ( \chi_\alpha)|_{g_0}^2 v_\alpha^2d\mu=\int_X |\nabla  \chi_\alpha v_\alpha|_{g_0}^2d\mu+o(1).\]
Then we get
\begin{equation}\label{PS} 
\left|\int_Xc_n |\nabla (\chi_\alpha v_\alpha)|_{g_0}^2 d\mu-\sigma\int_X |v_\alpha|^{\frac{4}{n-2}}\chi_\alpha^2v_\alpha^2 d\mu\right| \le o(1).
\end{equation}
Then we use the above estimates on the function $\Psi_{\alpha}$. Notice that
\[u_{\alpha}\Psi_{\alpha}= \vert u_\alpha\vert^{\frac{2n}{n-2}}- \vert w_\alpha\vert^{\frac{2n}{n-2}}-  \vert v_\alpha\vert^{\frac{2n}{n-2}} - v_{\alpha}w_{\alpha} \vert w_{\alpha}\vert^{\frac{4}{n-2}} -v_{\alpha} w_{\alpha}\vert  v_{\alpha}\vert^{\frac{4}{n-2}}.\]
We computed above that
\[\norm{v_{\alpha} \vert w_{\alpha}\vert^{\frac{n+2}{n-2}}}_{L^1(X)}\leq \norm{v_{\alpha} \vert w_{\alpha}\vert^{\frac{n+2}{n-2}}}_{L^{\frac{2n}{n+2}}(X)}\to 0\]
and similarly for $w\leftrightarrow v$. By the Hölder inequality, we have
\[\norm{ u_{\alpha} \Psi_{\alpha}}_{L^1(B(y_\alpha,\epsilon_\alpha))}\leq \norm{u_{\alpha}}_{L^{\frac{2n}{n-2}}(B(y_\alpha,\epsilon_\alpha))} \norm{\Psi_{\alpha}}_{L^{\frac{2n}{n+2}}(X)},\]
where the first factors is bounded by \eqref{eq:SobolevMassBound} and the second factor goes to zero by our computations above. All in all, we conclude
\begin{align*}
\lim_{\alpha\to \infty } \left(\int_{B(y_\alpha,\epsilon_\alpha)} |u_\alpha|^{\frac{2n}{n-2}}d\mu- \int_{B(y_\alpha,\epsilon_\alpha)} |w_\alpha|^{\frac{2n}{n-2}}d\mu- \int_{B(y_\alpha,\epsilon_\alpha)} |v_\alpha|^{\frac{2n}{n-2}}d\mu\right)=0.
\end{align*}
For all $\alpha$ large enough, \eqref{eq:SobolevMassBound} gives us (since $y_{\alpha}\to x$) that
\[\sigma\left(\int_{B(y_\alpha,\epsilon_\alpha)} |v_\alpha|^{\frac{2n}{n-2}}d\mu\right)^{\frac 2n}\le \frac{5}{8} S_\ell(X).\]
Hence
\[\sigma\int_X |v_\alpha|^{\frac{4}{n-2}}\chi_\alpha^2v_\alpha^2 d\mu\le  \frac{5}{8} S_\ell\left(\int_X |\chi_\alpha v_\alpha|^{\frac{2n}{n-2}}d\mu\right)^{1-\frac 2n}.\]
and using the Sobolev inequality \eqref{eq:AppSobolev} and the estimate (\ref{estL2}) we get
\begin{align*}
 &\frac 34 S_\ell(X)\left(\int_X |\chi_\alpha v_\alpha|^{\frac{2n}{n-2}}d\mu\right)^{1-\frac 2n} \\ &\le \int_X\left[ c_n|\nabla(\chi_\alpha v_\alpha)|_{g_0}^2+B|\chi_\alpha v_\alpha|^2\right]d\mu\\ &=\int_Xc_n|\nabla (\chi_\alpha v_\alpha)|_{g_0}^2d\mu+o(1),
\end{align*} 
and with (\ref{PS}), we get 
\[ \frac 34 S_\ell(X)\left(\int_X |\chi_\alpha v_\alpha|^{\frac{2n}{n-2}}d\mu\right)^{1-\frac 2n}\le o(1)+\frac{5}{8} S_\ell\left(\int_X |\chi_\alpha v_\alpha|^{\frac{2n}{n-2}}d\mu\right)^{1-\frac 2n}.\]
Hence
\[\lim_{\alpha\to \infty }\int_X |\chi_\alpha v_\alpha|^{\frac{2n}{n-2}}d\mu=0.\]
Then (\ref{PS}) implies that 
\[\lim_{\alpha\to \infty }\int_X |\nabla (\chi_\alpha v_\alpha)|_{g_0}^{2}d\mu=0.\]
\end{proof}

 \subsubsection*{Step 3 --  Extraction of one bubble}
 \begin{Lem} \label{onebubble} Let $(u_\alpha)$ be a Palais-Smale sequence for $I$ of non-negative functions. Then, up to extraction of subsequence, either $(u_\alpha)$ converges strongly in $H^1$ or there is some bubble $B_\alpha$ and  $(v_\alpha)$ another   Palais-Smale sequence  for $I$ of non-negative functions such that 
 \[\lim_{\alpha\to\infty} \left\|u_\alpha-v_\alpha-B_\alpha\right\|_{H^1}=0.\]
 Moreover
 \[\lim_{\alpha\to\infty} \int_{X} u_\alpha^{\frac{2n}{n-2}}d\mu-\int_{X} v_\alpha^{\frac{2n}{n-2}}d\mu-\int_{X} B_\alpha^{\frac{2n}{n-2}}d\mu=0.\]
 \end{Lem}
 \begin{proof}[Proof of Lemma \ref{onebubble}]
  Let
 \[\epsilon_\alpha:=\min_{x\in X}\sup\left\{s>0\text{ such that } \forall r\le s\colon \sigma\left(\int_{B(x,r)} u_\alpha^{\frac{2n}{n-2}}d\mu\right)^{\frac 2n}\le S_\ell(X)/2\right\}.\]
 Then one of two scenarios hold. Alternative 1: $\inf \epsilon_\alpha>0$ and by the proof of  Lemma \ref{Lem:LocalConv} implies that $H^1$-weak sublimits of the sequence $(u_\alpha)$ are actually $H^1$- strong sublimits. 
\noindent
 Alternative 2. Up to extraction of a subsequence, we have
 \begin{itemize}
 \item $\lim\limits_{\alpha\to \infty} \epsilon_\alpha=0$
 \item There are $y_\alpha\in X$ such that $\sigma\left(\int_{B(y_\alpha,\epsilon_\alpha)} u_\alpha^{\frac{2n}{n-2}}d\mu\right)^{1-\frac 2n}=S_\ell/2$,
 \item $\lim\limits_{\alpha\to \infty} y_\alpha=x$
 \item $ (\underline{X},\underline{d},\underline{x})$ is a fake tangent space at $x\in X$, that is is the pointed Gromov-Hausdorff limit of the sequence $(X,\epsilon_\alpha^{-1}d,y_\alpha)$,
 \item along this convergence of spaces,
 $\epsilon^{\frac{n-2}{2}}u_\alpha$ converges weakly in $H^1$ to some $u\in H^1(\underline{X})$.
 \end{itemize}
 Note that $u\ge 0$.
 By construction we can apply Lemma \ref{Lem:LocalConv} and get that  $\epsilon_\alpha^{\frac{n-2}{2}}u_\alpha$ converges strongly in $H^1_{loc}$. In particular
 \[\sigma \left(\int_{B(\underline{x},1)} u^{\frac{2n}{n-2}}d\underline{\mu}\right)^{\frac 2n}=S_\ell(X)/2.\]
 Hence $u$ is not zero and we are in position to apply the Theorem \ref{rigid} and get that for some $\underline{x}\in \underline{X}$:
 \[u(x)=\left(\frac{c \lambda}{\lambda^2+\underline{d}(x,\underline{x})^2}\right)^{\frac{n-2}{2}}\]
  with $c=\sqrt{n(n-1)/\sigma}$.
 Hence we can find a bubble $(B_\alpha)$ with center $x_\alpha$ and scale $\epsilon_\alpha$ 
such that
$x_\alpha \to \underline{x}$, and along the convergence toward the fake tangent space
$(\epsilon_\alpha^{\frac{n-2}{2}}B_\alpha)$ converges strongly to $u$. The proof Lemma \ref{Lem:LocalConv} shows that $(u_\alpha-B_\alpha)$ is a  Palais-Smale sequence for $I$ . We let\footnote{where we used the notation $x_{\pm}=\frac{|x|\pm x}{2}=\max\{\pm x,0\}$.}
$v_\alpha=(u_\alpha-B_\alpha)_+$ and we get
\[u_\alpha=v_\alpha+B_\alpha-r_\alpha\]
 where
$r_\alpha=(u_\alpha-B_\alpha)_-$ .

For each $\Lambda>0$ we have
\[\int_X |r_\alpha|^{\frac{2n}{n-2}}d\mu\le \int_{X\setminus B(x_\alpha,\Lambda\epsilon_\alpha)} B_\alpha^{\frac{2n}{n-2}}d\mu+ \int_{B(x_\alpha,\Lambda\epsilon_\alpha)} \left|B_\alpha-u_\alpha\right|^{\frac{2n}{n-2}}d\mu.\]
The first integral can easily be estimate (using the same method as the one in Lemma (\ref{bubblecv}) , and we get
\[ \int_{X\setminus B(x_\alpha,\Lambda\epsilon_\alpha)} B_\alpha^{\frac{2n}{n-2}}d\mu\le \frac{C}{\Lambda^n}.\]
The strong convergence in $H^1_{loc}$ implies that for fixed $\Lambda>0$, $\lim\limits_{\alpha\to \infty} \int_{B(x_\alpha,\Lambda\epsilon_\alpha)} \left|B_\alpha-u_\alpha\right|^{\frac{2n}{n-2}}d\mu=0$. Hence
\[\limsup_{\alpha\to \infty} \int_X |r_\alpha|^{\frac{2n}{n-2}}d\mu\le \frac{C}{\Lambda^n}.\]
The same computation leads to the fact that $(r_\alpha)$ converges strongly to $0$ in $H^1$.
Hence 
\[\lim\limits_{\alpha\to\infty} \left\|u_\alpha-v_\alpha-B_\alpha\right\|_{H^1}=0.\]
The last assertion also follows from the same computation.
 
 \end{proof}
\subsubsection*{Step 4 -- Profile decomposition}We discuss now of the mass of a bubble that is to say of
\[\lim_{\alpha\to \infty} \int_{X} B_\alpha^{\frac{2n}{n-2}}d\mu.\]
We have already said that this is equal to the corresponding quantity of the associated fake tangent space, that is if $(x_\alpha)$ are the centres of the bubble and $\epsilon_\alpha$ are the scaled then $\left((X,\epsilon_\alpha^{-1}d,x_\alpha)\right)_\alpha$ converges in the pointed Gromov-Hausdorff topology to $(\underline{X},\underline{d}, \underline{x})$ and along this convergences of spaces $\epsilon_\alpha^{\frac{n-2}{2}}B_\alpha$ converges strongly in $H^1$ to $u$, which is a solution of the equation
\[-c_n\Delta u=\sigma u^{\frac{n+2}{n-2}}.\]
Hence 
\[\lim_\alpha \int_{X} B_\alpha^{\frac{2n}{n-2}}d\mu= \int_{\underline{X}}u^{\frac{2n}{n-2}}d\underline{\mu}=\left(\frac{n(n-1)}{\sigma} \right)^{\frac n2}\,\omega_n\Theta_{\underline{X}}(\underline{x}).\]
Where $\omega_n$ is the volume of the unit Euclidean ball and $\Theta_{\underline{X}}(\underline{x})=\frac{\underline{\mu} \left(B(\underline{x},r)\right)}{\omega_n r^n}$ (recall that $(\underline{X}, \underline{x})$ is conical at $\underline{x}$). But according to  \cite[Proposition 4.2]{Mondello}, $u$ realizes the Yamabe invariant of $\underline{X}$, hence
\[\sigma \left(\int_{\underline{X}}u^{\frac{2n}{n-2}}d\underline{\mu}\right)^{\frac 2 n}=Y(\underline{X})\ge Y_\ell(X)=S_\ell(X).\]
So that the mass of a bubble is always larger or equal to $\left(\frac{n(n-1)}{\sigma}\right)^{\frac n2}Y_\ell(X).$ Hence if one applies Lemma (\ref{onebubble}) several times, we can not extract more that $L$ bubbles where\footnote{Compare with \eqref{eq:L-Bound}.}
\[L=\left(\frac{\sigma}{Y_\ell(X)}\right)^{\frac n2}\limsup_\alpha \int_{X} u_\alpha^{\frac{2n}{n-2}}d\mu.\]
So that we get the existence of a finite number of bubbles and $u_\infty$ so that we get the strong $H^1$-convergence (\ref{bubbles}).
\subsubsection*{Step 5 -- On the localisation of the bubbles} 
We need now to explain why the bubble are separated that is we explain why (\ref{sepa}) is true. If it not true then up to extraction of a subsequence and change of the labelling of the bubbles one can assume that for each $j>i$:
$\lim\limits_{\alpha\to \infty} \frac{\epsilon_\alpha(j)}{\epsilon_\alpha(i)}+\frac{\epsilon_\alpha(i)}
{\epsilon_\alpha(j)}+\frac{d(x_\alpha(j),x_\alpha(i))^2}{\epsilon_\alpha(i)\epsilon_\alpha(j)}=\infty$ and that  $\left((X,\epsilon_\alpha(i)^{-1}d,x_\alpha(i))\right)_\alpha$ converges in the pointed Gromov-Hausdorff topology to $(\underline{X},\underline{d}, \underline{x}(i))$ and such that along this convergence
for $j\le i$ we have $\lim\limits_{\alpha\to \infty} x_\alpha(j)=\underline{x}(j)$ and for each $j\le i$
$\lim\limits_{\alpha\to \infty}\frac{\epsilon_\alpha(j)}{\epsilon_\alpha(i)}=c_i$, 
$\epsilon_\alpha(i)^{\frac{n-2}{2}}B_\alpha^{j}$ converges strongly in $H^1$ to $u(j)$ and
$\epsilon_\alpha(i)^{\frac{n-2}{2}}u_\alpha$ converges weakly to $u$. We will have
$u=\sum_{j=1}^i u(j)$ and $u$ moreover each of the $u(j)$ will solve the equation
\[-\Delta f=\sigma f^{\frac{n+2}{n-2}}.\]
But then
\[\sigma\left(\sum_{j=1}^i u(j)\right)^{\frac{n+2}{n-2}} =-\Delta u=-\sum_{j=1}^i \Delta u(j)=\sigma \sum_{j=1}^i u(j)^{\frac{n+2}{n-2}}.\]
This is impossible unless $i=1$, since $u(j)>0$ and $\frac{n+2}{n-2}>1$ .

\endproof

\end{document}